\documentclass[reqno,a4paper,twoside]{amsart}

\usepackage[a4paper,margin=0.97in,footskip=0.25in]{geometry}

\usepackage{enumitem}

\setenumerate{label=\textnormal{(\arabic*)}}

\allowdisplaybreaks[4]

\usepackage[raggedright]{titlesec}

\titleformat{\chapter}[display]
{\normalfont\huge\bfseries}{\chaptertitlename\\thechapter}{20pt}{\Huge}
\titleformat{\section}
{\normalfont\Large\bfseries\center}{\thesection}{1em}{}
\titleformat{\subsection}
{\normalfont\large\bfseries}{\thesubsection}{1em}{}
\titleformat{\subsubsection}
{\normalfont\normalsize\bfseries}{\thesubsubsection}{1em}{}
\titleformat{\paragraph}[runin]
{\normalfont\normalsize\bfseries}{\theparagraph}{1em}{}
\titleformat{\subparagraph}[runin]
{\normalfont\normalsize\bfseries}{\thesubparagraph}{1em}{}

\titlespacing*{\chapter} {0pt}{50pt}{40pt}
\titlespacing*{\section} {0pt}{3.5ex plus 1ex minus .2ex}{2.3ex plus .2ex}
\titlespacing*{\subsection} {0pt}{3.25ex plus 1ex minus .2ex}{1.5ex plus .2ex}
\titlespacing*{\subsubsection}{0pt}{3.25ex plus 1ex minus .2ex}{1.5ex plus .2ex}
\titlespacing*{\paragraph} {0pt}{3.25ex plus 1ex minus .2ex}{1em}
\titlespacing*{\subparagraph} {\parindent}{3.25ex plus 1ex minus .2ex}{1em}

\usepackage{tikz}
\usetikzlibrary{matrix}

\usepackage{float, subfig}
\usepackage{amsrefs}

\usepackage{titletoc}

\usepackage{stackengine}
\usepackage{scalerel}

\contentsmargin{2.55em}
\dottedcontents{section}[1.5em]{}{2em}{1pc}
\dottedcontents{subsection}[4.35em]{}{2.8em}{1pc}
\dottedcontents{subsubsection}[7.6em]{}{3.2em}{1pc}
\dottedcontents{paragraph}[10.3em]{}{3.2em}{1pc}

\usepackage[columns=3,rule=0pt]{idxlayout}

\setindexprenote{For convenience of the reader we list below almost all symbols used in this work, indicating the page in which each of them is introduced. We only except those that are extremely standard.}

\usepackage{amsmath,amssymb,dsfont,verbatim,bm,geometry,mathrsfs}

\usepackage{mathtools}

\makeindex

\newtheorem{theorem}{Theorem}[section]
\newtheorem{lemma}[theorem]{Lemma}
\newtheorem{proposition}[theorem]{Proposition}
\newtheorem{corollary}[theorem]{Corollary}

\theoremstyle{definition}
\newtheorem{definition}[theorem]{Definition}
\newtheorem{notations}[theorem]{Notations}
\newtheorem{notation}[theorem]{Notation}
\newtheorem{example}[theorem]{Example}

\theoremstyle{remark}
\newtheorem{remark}[theorem]{Remark}

\DeclareMathOperator{\ide}{id}

\DeclareMathOperator{\Ext}{Ext}

\DeclareMathOperator{\ima}{Im}
\DeclareMathOperator{\BC}{BC}
\DeclareMathOperator{\BN}{BN}
\DeclareMathOperator{\BP}{BP}
\DeclareMathOperator{\Ho}{H}
\DeclareMathOperator{\HH}{HH}
\DeclareMathOperator{\HC}{HC}
\DeclareMathOperator{\HP}{HP}
\DeclareMathOperator{\HN}{HN}
\DeclareMathOperator{\Hom}{Hom}
\DeclareMathOperator{\Sh}{Sh}
\DeclareMathOperator{\Tot}{Tot}
\DeclareMathOperator{\Tor}{Tor}

\newcommand{\ov}{\overline}
\newcommand{\ot}{\otimes}
\newcommand{\wh}{\widehat}
\newcommand{\wt}{\widetilde}

\newcommand{\ba}{\mathbf a}
\newcommand{\bx}{\mathbf x}

\newcommand{\brh}{\mathrm h}
\newcommand{\bv}{\mathrm v}
\newcommand{\byy}{\mathbf y}
\newcommand{\bz}{\mathbf z}
\newcommand{\hs}{\hspace{-0.5pt}}
\newcommand{\hsm}{\hspace{-0.2pt}}
\newcommand{\xcdot}{\hsm\cdot\hsm}
\newcommand{\xcirc}{\hsm\circ\hsm}

\newcommand{\ovb}{\bar}

\DeclareMathAlphabet{\mathpzc}{OT1}{pzc}{m}{it}

\numberwithin{equation}{section}

\usepackage[pdftex,breaklinks,pdfencoding=auto,psdextra,unicode,plainpages=false]{hyperref}

\begin{document}

\title{(Co)homology of Crossed Products in Weak Contexts}

\author{Jorge A. Guccione}
\address{Departamento de Matem\'atica\\ Facultad de Ciencias Exactas y Naturales-UBA, Pabell\'on~1-Ciudad Universitaria\\ Intendente Guiraldes 2160 (C1428EGA) Buenos Aires, Argentina.}
\address{Instituto de Investigaciones Matem\'aticas ``Luis A. Santal\'o"\\ Facultad de Ciencias Exactas y Natu\-ra\-les-UBA, Pabell\'on~1-Ciudad Universitaria\\ Intendente Guiraldes 2160 (C1428EGA) Buenos Aires, Argentina.}
\email{vander@dm.uba.ar}

\author{Juan J. Guccione}
\address{Departamento de Matem\'atica\\ Facultad de Ciencias Exactas y Naturales-UBA\\ Pabell\'on~1-Ciudad Universitaria\\ Intendente Guiraldes 2160 (C1428EGA) Buenos Aires, Argentina.}
\address{Instituto Argentino de Matem\'atica-CONICET\\ Savedra 15 3er piso\\ (C1083ACA) Buenos Aires, Argentina.}
\email{jjgucci@dm.uba.ar}

\author[C. Valqui]{Christian Valqui}
\address{Pontificia Universidad Cat\'olica del Per\'u - Instituto de Matem\'atica y Ciencias Afi\-nes, Secci\'on Matem\'aticas, PUCP, Av. Universitaria 1801, San Miguel, Lima 32, Per\'u.}
\email{cvalqui@pucp.edu.pe}


\begin{abstract} We obtain a mixed complex simpler than the canonical one the computes the type cyclic homologies of the weak crossed products $E\coloneqq A\times_{\chi}^{\mathcal{F}} V$ introduced in \cite{AFGR}. This complex is provided with a canonical filtration, whose spectral sequence generalizes the spectral sequence obtained in \cite{CGG}. Under suitable conditions, the above mentioned mixed complex is provided with another filtration, whose spectral sequence generalizes the Feigin-Tsygan spectral sequence. These results apply in particular to the crossed products of algebras by weak bialgebras.
\end{abstract}

\maketitle

\tableofcontents

\section*{Introduction}

Due to the relation of the actions of groups on algebras with non-commutative geometry, the problem of developing tools to compute the cyclic homology of smash products algebras $A\# k[G]$, where $A$ is an algebra and $G$ is a group, was considered in \cites{FT, GJ, N}. For instance, in the first paper the authors obtained a spectral sequence converging to the cyclic homology of $A\# k[G]$, and in \cite{GJ}, this result was derived from the theory of paracyclic modules and cylindrical modules developed by the authors. The main tool for this computation was a version for cylindrical modules of the Eilenberg-Zilber theorem. More recently, and also due to its connections with non-commutative geometry, the cyclic homology of algebras obtained through more general constructions (Hopf crossed products, Hopf Galois extensions, Braided Hopf crossed products, etcetera) has been extensively studied. See for instance \cites{AK, CGG, CGGV, GG, JS, V, ZH}. The method developed in \cites{GG, CGG, CGGV} has the advantage over others that it works for arbitrary cocycles. In this paper we use it to compute the Hochschild (co)homology and the type cyclic homologies of the weak crossed products introduced in \cite{AFGR}. Specifically, for such a crossed product $E\coloneqq A\times_{\chi}^{\mathcal{F}} V$, we construct a mixed complex, simpler than the canonical one, that computes the Hochschild, cyclic, negative and periodic homologies of $E$. The Hochschild and cyclic complexes of this mixed complex are provided with canonical filtrations whose spectral sequences generalize the Hochschild-Serre spectral sequence and the Feigin and Tsygan spectral sequence. We also study the Hochschild cohomology of these algebras. In the subsequent paper \cite{GGV1} we use the main results of this paper in order to compute the Hochschild (co)homology and the type cyclic homologies of crossed products of algebras by weak Hopf algebras with invertible cocycle. We hope that our method works also for partial crossed products (see \cite{ABDP})

\smallskip

The paper is organized as follows:

\smallskip

A crossed product system with preunit is a tuple $(A,V,\chi,\mathcal{F},\nu)$, where $A$ a $k$-algebra, $V$ a $k$-vector space and $\chi\colon V\ot_k A\to A\ot_k V$ and $\mathcal{F}\colon V\ot_k A\to A\ot_k V$ are maps satisfying suitable conditions. Each such a tuple has associated an algebra $E\coloneqq A\times_{\chi}^{\mathcal{F}} V$. If $\mathcal{F}$ is a cocycle that satisfied the twisted module condition (see Definition~\ref{twisted module and cociclo condiciones}), then $E$ is an associative algebra with unit, which is named the unitary crossed product of $A$ by $V$ associated with $\chi$ and $\mathcal{F}$. In Section~1 we give a review of the these crossed products (that includes the crossed products introduced in~\cite{Br}), we also recall the concept of mixed complex and the perturbation lemma. In Section~2 we construct a resolution of $E$ as an $E$-bimodule. In Section~3 and~4 we use this resolution to obtain complexes, simpler than the canonical ones, that compute the Hochschild homology and the Hochschild cohomology of $E$ with coefficients in an $E$-bimodule $M$. Then, in Section~5 we study the cup and the cap products of $E$, and, finally, in Section~6 we obtain a mixed complex, simpler that the canonical one, that computes the type cyclic homologies of $E$.

\section{Preliminaries}
In this article we work in the category of vector spaces over a field $k$. Hence we assume implicitly that all the maps are $k$-linear maps. The tensor product over $k$ is denoted by $\ot_k$. By an algebra we understand an associative algebra over $k$. Given an arbitrary algebra $K$, a $K$-bimodule $V$ and a natural number $n\ge 0$, we let $V^{\ot_{\hs K} n}$ denote the $n$-fold tensor product $V\ot_{\hs K}\cdots\ot_{\hs K} V$, which is considered as a $K$-bimodule via
$$
\lambda\xcdot (v_1\ot_{\hs K}\cdots \ot_{\hs K} v_n)\xcdot \lambda' \coloneqq \lambda\xcdot v_1\ot_{\hs K}\cdots \ot_{\hs K} v_n\xcdot \lambda'.
$$
Given $k$-vector spaces $U$, $V$, $W$ and a map $g\colon V\to W$ we write $U\ot_k g$ for $\ide_U\ot_k g$ and $g\ot_k U$ for $g\ot_k\ide_U$. Given a algebra $C$, we let $\mu\colon C\ot_k C\to C\index{zm@$\mu$|dotfillboldidx}$ denote its multiplication map. When $C$ is unitary we let $\eta\colon k\to C\index{zh@$\eta$|dotfillboldidx}$, denote its unit.

\smallskip

In some parts of this article we use the nowadays well known graphic calculus for monoidal and braided categories. As usual, morphisms will be composed from top to bottom and tensor products will be represented by horizontal concatenation from left to right. The identity map of a $k$-vector space will be represented by a vertical line. Given an algebra $A$, the diagrams
\begin{equation*}\label{eq1}
\begin{tikzpicture}[scale=0.4]
\def\mult(#1,#2)[#3]{\draw (#1,#2) .. controls (#1,#2-0.555*#3/2) and (#1+0.445*#3/2,#2-#3/2) .. (#1+#3/2,#2-#3/2) .. controls (#1+1.555*#3/2,#2-#3/2) and (#1+2*#3/2,#2-0.555*#3/2) .. (#1+2*#3/2,#2) (#1+#3/2,#2-#3/2) -- (#1+#3/2,#2-2*#3/2)}
\mult(0,0)[1];\node at (1.4,-0.5) {,};
\end{tikzpicture}
\qquad
\begin{tikzpicture}[scale=0.4]
\def\unit(#1,#2){\draw (#1,#2) circle[radius=2pt] (#1,#2-0.07) -- (#1,#2-1)}
\unit(0,0);\node at (0.4,-0.5) {,};
\end{tikzpicture}
\qquad
\begin{tikzpicture}[scale=0.4]
\def\laction(#1,#2)[#3,#4]{\draw (#1,#2) .. controls (#1,#2-0.555*#4/2) and (#1+0.445*#4/2,#2-1*#4/2) .. (#1+1*#4/2,#2-1*#4/2) -- (#1+2*#4/2+#3*#4/2,#2-1*#4/2) (#1+2*#4/2+#3*#4/2,#2)--(#1+2*#4/2+#3*#4/2,#2-2*#4/2)}
\laction(0,0)[0,1];
\end{tikzpicture}
\qquad
\begin{tikzpicture}[scale=0.4]
\node at (0,0) {and};
\end{tikzpicture}
\qquad
\begin{tikzpicture}[scale=0.4]
\def\raction(#1,#2)[#3,#4]{\draw (#1,#2) -- (#1,#2-2*#4/2) (#1,#2-1*#4/2)--(#1+1*#4/2+#3*#4/2,#2-1*#4/2) .. controls (#1+1.555*#4/2+#3*#4/2,#2-1*#4/2) and (#1+2*#4/2+#3*#4/2,#2-0.555*#4/2) .. (#1+2*#4/2+#3*#4/2,#2)}
\raction(0,0)[0,1];
\end{tikzpicture}
\end{equation*}
stand for the multiplication map, the unit (when $C$ is unital), the action of $C$ on a left $C$-module and the action of $C$ on a right $C$-module, respectively. We will also use the diagrams
\begin{equation*}
\begin{tikzpicture}[scale=0.45]
\def\braid(#1,#2)[#3]{\draw (#1+1*#3,#2) .. controls (#1+1*#3,#2-0.05*#3)
and (#1+0.96*#3,#2-0.15*#3).. (#1+0.9*#3,#2-0.2*#3)
(#1+0.1*#3,#2-0.8*#3)--(#1+0.9*#3,#2-0.2*#3) (#1,#2-1*#3) .. controls
(#1,#2-0.95*#3) and (#1+0.04*#3,#2-0.85*#3).. (#1+0.1*#3,#2-0.8*#3)
(#1,#2) .. controls (#1,#2-0.05*#3) and (#1+0.04*#3,#2-0.15*#3)..
(#1+0.1*#3,#2-0.2*#3) (#1+0.1*#3,#2-0.2*#3) -- (#1+0.37*#3,#2-0.41*#3)
(#1+0.62*#3,#2-0.59*#3)-- (#1+0.9*#3,#2-0.8*#3) (#1+1*#3,#2-1*#3) ..
controls (#1+1*#3,#2-0.95*#3) and (#1+0.96*#3,#2-0.85*#3)..
(#1+0.9*#3,#2-0.8*#3)}
\def\braidmamed(#1,#2)[#3]{\draw (#1,#2-0.5)  node[name=nodemap,inner
sep=0pt,  minimum size=10pt, shape=circle,draw]{$#3$}
(#1-0.5,#2) .. controls (#1-0.5,#2-0.15) and (#1-0.4,#2-0.2) ..
(#1-0.3,#2-0.3) (#1-0.3,#2-0.3) -- (nodemap)
(#1+0.5,#2) .. controls (#1+0.5,#2-0.15) and (#1+0.4,#2-0.2) ..
(#1+0.3,#2-0.3) (#1+0.3,#2-0.3) -- (nodemap)
(#1+0.5,#2-1) .. controls (#1+0.5,#2-0.85) and (#1+0.4,#2-0.8) ..
(#1+0.3,#2-0.7) (#1+0.3,#2-0.7) -- (nodemap)
(#1-0.5,#2-1) .. controls (#1-0.5,#2-0.85) and (#1-0.4,#2-0.8) ..
(#1-0.3,#2-0.7) (#1-0.3,#2-0.7) -- (nodemap)
}
\def\cocycle(#1,#2)[#3]{\draw (#1,#2) .. controls (#1,#2-0.555*#3/2) and
(#1+0.445*#3/2,#2-#3/2) .. (#1+#3/2,#2-#3/2) .. controls
(#1+1.555*#3/2,#2-#3/2) and (#1+2*#3/2,#2-0.555*#3/2) .. (#1+2*#3/2,#2)
(#1+#3/2,#2-#3/2) -- (#1+#3/2,#2-2*#3/2) (#1+#3/2,#2-#3/2)  node [inner
sep=0pt,minimum size=3pt,shape=circle,fill] {}}
\def\cocyclenamed(#1,#2)[#3][#4]{\draw (#1+#3/2,#2-1*#3/2)
node[name=nodemap,inner sep=0pt,  minimum size=9pt,
shape=circle,draw]{$#4$}
(#1,#2) .. controls (#1,#2-0.555*#3/2) and (#1+0.445*#3/2,#2-#3/2) ..
(nodemap) .. controls (#1+1.555*#3/2,#2-#3/2) and
(#1+2*#3/2,#2-0.555*#3/2) .. (#1+2*#3/2,#2) (nodemap)--
(#1+#3/2,#2-2*#3/2)
}
\def\cocycleTR90(#1,#2)[#3]{\draw (#1,#2) .. controls (#1,#2-0.555*#3/2)
and (#1+0.445*#3/2,#2-#3/2) .. (#1+#3/2,#2-#3/2) .. controls
(#1+1.555*#3/2,#2-#3/2) and (#1+2*#3/2,#2-0.555*#3/2) .. (#1+2*#3/2,#2)
(#1+#3/2,#2-#3/2) -- (#1+#3/2,#2-2*#3/2) (#1+#3/2,#2-#3/2)  node [inner
sep=0pt,minimum size=3pt,shape=isosceles triangle, rotate=90,fill] {}}
\def\comult(#1,#2)[#3,#4]{\draw (#1,#2) -- (#1,#2-0.5*#4) arc (90:0:0.5*#3
and 0.5*#4) (#1,#2-0.5*#4) arc (90:180:0.5*#3 and 0.5*#4)}
\def\counit(#1,#2){\draw (#1,#2) -- (#1,#2-0.93) (#1,#2-1)
circle[radius=2pt]}
\def\cuadrupledoublemap(#1,#2)[#3]{\draw (#1+1.5,#2-0.5) node
[name=doublesinglemapnode,inner xsep=0pt, inner ysep=0pt, minimum
height=9pt, minimum width=40pt,shape=rectangle,draw,rounded corners]
{$#3$} (#1+1,#2) .. controls (#1+1,#2-0.075) .. (doublesinglemapnode)
(#1,#2) .. controls (#1,#2-0.075) .. (doublesinglemapnode) (#1+2,#2) ..
controls (#1+2,#2-0.075) .. (doublesinglemapnode) (#1+3,#2) .. controls
(#1+3,#2-0.075) .. (doublesinglemapnode) (doublesinglemapnode)--
(#1+1,#2-1) (doublesinglemapnode)-- (#1+2,#2-1)
}
\def\doublemap(#1,#2)[#3]{\draw (#1+0.5,#2-0.5) node
[name=doublemapnode,inner xsep=0pt, inner ysep=0pt, minimum height=10pt,
minimum width=22pt,shape=rectangle,draw,rounded corners] {$#3$} (#1,#2) ..
controls (#1,#2-0.075) .. (doublemapnode) (#1+1,#2) .. controls
(#1+1,#2-0.075).. (doublemapnode) (doublemapnode) .. controls
(#1,#2-0.925)..(#1,#2-1) (doublemapnode) .. controls (#1+1,#2-0.925)..
(#1+1,#2-1)}
\def\doublesinglemap(#1,#2)[#3]{\draw (#1+0.5,#2-0.5) node
[name=doublesinglemapnode,inner xsep=0pt, inner ysep=0pt, minimum
height=10pt, minimum width=2pt,shape=rectangle,draw,rounded corners]
{$#3$} (#1,#2) .. controls (#1,#2-0.075) .. (doublesinglemapnode)
(#1+1,#2) .. controls (#1+1,#2-0.075).. (doublesinglemapnode)
(doublesinglemapnode)-- (#1+0.5,#2-1)}
\def\flip(#1,#2)[#3]{\draw (#1,#2) -- (#1+1*#3,#2-1*#3); \draw
(#1+1*#3,#2) -- (#1,#2-1*#3);}
\def\laction(#1,#2)[#3,#4]{\draw (#1,#2) .. controls (#1,#2-0.555*#4/2)
and (#1+0.445*#4/2,#2-1*#4/2) .. (#1+1*#4/2,#2-1*#4/2) --
(#1+2*#4/2+#3*#4/2,#2-1*#4/2)
(#1+2*#4/2+#3*#4/2,#2)--(#1+2*#4/2+#3*#4/2,#2-2*#4/2)}
\def\lactionnamed(#1,#2)[#3,#4][#5]{\draw (#1+2*#4/2+#3*#4/2,#2-#4/2)
node[name=nodemap,inner sep=0pt,  minimum size=10pt,
shape=circle,draw]{$#5$}
(#1,#2) .. controls (#1,#2-0.555*#4/2) and (#1+0.445*#4/2,#2-1*#4/2) ..
(#1+1*#4/2,#2-1*#4/2) -- (nodemap)
(#1+2*#4/2+#3*#4/2,#2)  -- (nodemap)
(nodemap)--(#1+2*#4/2+#3*#4/2,#2-2*#4/2)
}
\def\map(#1,#2)[#3]{\draw (#1,#2-0.5)  node[name=nodemap,inner sep=0pt,
minimum size=10pt, shape=circle,draw]{$#3$} (#1,#2)-- (nodemap)
(nodemap)-- (#1,#2-1)}
\def\mult(#1,#2)[#3,#4]{\draw (#1,#2) arc (180:360:0.5*#3 and 0.5*#4)
(#1+0.5*#3, #2-0.5*#4) -- (#1+0.5*#3,#2-#4)}
\def\raction(#1,#2)[#3,#4]{\draw (#1,#2) -- (#1,#2-2*#4/2)
(#1,#2-1*#4/2)--(#1+1*#4/2+#3*#4/2,#2-1*#4/2) .. controls
(#1+1.555*#4/2+#3*#4/2,#2-1*#4/2) and (#1+2*#4/2+#3*#4/2,#2-0.555*#4/2) ..
(#1+2*#4/2+#3*#4/2,#2)}
\def\mult(#1,#2)[#3,#4]{\draw (#1,#2) arc (180:360:0.5*#3 and 0.5*#4)
(#1+0.5*#3, #2-0.5*#4) -- (#1+0.5*#3,#2-#4)}
\def\rcoaction(#1,#2)[#3,#4]{\draw (#1,#2)-- (#1,#2-2*#4/2) (#1,#2-1*#4/2)
-- (#1+1*#4/2+#3*#4/2,#2-1*#4/2).. controls
(#1+1.555*#4/2+#3*#4/2,#2-1*#4/2) and (#1+2*#4/2+#3*#4/2,#2-1.445*#4/2) ..
(#1+2*#4/2+#3*#4/2,#2-2*#4/2)}
\def\singledoublemap(#1,#2)[#3]{\draw (#1+0.5,#2-0.5) node
[name=doublemapnode,inner xsep=0pt, inner ysep=0pt, minimum height=10pt,
minimum width=22pt,shape=rectangle,draw,rounded corners] {$#3$}
(#1+0.5,#2)--(doublemapnode) (doublemapnode) .. controls
(#1,#2-0.925)..(#1,#2-1) (doublemapnode) .. controls (#1+1,#2-0.925)..
(#1+1,#2-1) }
\def\singletriplemap(#1,#2)[#3]{\draw (#1+0.5,#2-0.5) node
[name=triplemapnode,inner xsep=0pt, inner ysep=0pt, minimum height=9pt,
minimum width=27pt, shape=rectangle,draw,rounded corners] {$#3$}
(#1+0.5,#2)--(triplemapnode) (triplemapnode) .. controls
(#1-0.5,#2-0.925)..(#1-0.5,#2-1) (triplemapnode) .. controls
(#1+0.5,#2-0.925).. (#1+0.5,#2-1) (triplemapnode) .. controls
(#1+1.5,#2-0.925).. (#1+1.5,#2-1) }
\def\transposition(#1,#2)[#3]{\draw (#1+#3,#2) .. controls
(#1+#3,#2-0.05*#3) and (#1+0.96*#3,#2-0.15*#3).. (#1+0.9*#3,#2-0.2*#3)
(#1+0.1*#3,#2-0.8*#3)--(#1+0.9*#3,#2-0.2*#3) (#1+0.1*#3,#2-0.2*#3) ..
controls  (#1+0.3*#3,#2-0.2*#3) and (#1+0.46*#3,#2-0.31*#3) ..
(#1+0.5*#3,#2-0.34*#3) (#1,#2-1*#3) .. controls (#1,#2-0.95*#3) and
(#1+0.04*#3,#2-0.85*#3).. (#1+0.1*#3,#2-0.8*#3) (#1,#2) .. controls
(#1,#2-0.05*#3) and (#1+0.04*#3,#2-0.15*#3).. (#1+0.1*#3,#2-0.2*#3)
(#1+0.1*#3,#2-0.2*#3) .. controls  (#1+0.1*#3,#2-0.38*#3) and
(#1+0.256*#3,#2-0.49*#3) .. (#1+0.275*#3,#2-0.505*#3)
(#1+0.50*#3,#2-0.66*#3) .. controls (#1+0.548*#3,#2-0.686*#3) and
(#1+0.70*#3,#2-0.8*#3)..(#1+0.9*#3,#2-0.8*#3) (#1+0.72*#3,#2-0.50*#3) ..
controls (#1+0.80*#3,#2-0.56*#3) and
(#1+0.9*#3,#2-0.73*#3)..(#1+0.9*#3,#2-0.8*#3) (#1+#3,#2-#3) .. controls
(#1+#3,#2-0.95*#3) and (#1+0.96*#3,#2-0.85*#3).. (#1+0.9*#3,#2-0.8*#3)}
\def\tripledoublemap(#1,#2)[#3]{\draw (#1,#2-0.5) node
[name=doublesinglemapnode,inner xsep=0pt, inner ysep=0pt, minimum
height=10pt, minimum width=30pt, shape=rectangle, draw, rounded corners]
{$#3$} (#1-1,#2) .. controls (#1-1,#2-0.075) .. (doublesinglemapnode)
(#1,#2) .. controls (#1,#2-0.075) .. (doublesinglemapnode) (#1+1,#2) ..
controls (#1+1,#2-0.075) .. (doublesinglemapnode) (doublesinglemapnode)--
(#1-0.5,#2-1) (doublesinglemapnode)-- (#1+0.5,#2-1)
}
\def\triplesinglemap(#1,#2)[#3]{\draw (#1,#2-0.5) node
[name=doublesinglemapnode,inner xsep=0pt, inner ysep=0pt, minimum
height=10pt, minimum width=30pt, shape=rectangle, draw, rounded corners]
{$#3$} (#1-1,#2) .. controls (#1-1,#2-0.075) .. (doublesinglemapnode)
(#1,#2) .. controls (#1,#2-0.075) .. (doublesinglemapnode) (#1+1,#2) ..
controls (#1+1,#2-0.075) .. (doublesinglemapnode) (doublesinglemapnode)--
(#1,#2-1)
}
\def\twisting(#1,#2)[#3]{\draw (#1+#3,#2) .. controls (#1+#3,#2-0.05*#3)
and (#1+0.96*#3,#2-0.15*#3).. (#1+0.9*#3,#2-0.2*#3) (#1,#2-1*#3) ..
controls (#1,#2-0.95*#3) and (#1+0.04*#3,#2-0.85*#3)..
(#1+0.1*#3,#2-0.8*#3) (#1+0.1*#3,#2-0.8*#3) ..controls
(#1+0.25*#3,#2-0.8*#3) and (#1+0.45*#3,#2-0.69*#3) ..
(#1+0.50*#3,#2-0.66*#3) (#1+0.1*#3,#2-0.8*#3) ..controls
(#1+0.1*#3,#2-0.65*#3) and (#1+0.22*#3,#2-0.54*#3) ..
(#1+0.275*#3,#2-0.505*#3) (#1+0.72*#3,#2-0.50*#3) .. controls
(#1+0.75*#3,#2-0.47*#3) and (#1+0.9*#3,#2-0.4*#3).. (#1+0.9*#3,#2-0.2*#3)
(#1,#2) .. controls (#1,#2-0.05*#3) and (#1+0.04*#3,#2-0.15*#3)..
(#1+0.1*#3,#2-0.2*#3) (#1+0.5*#3,#2-0.34*#3) .. controls
(#1+0.6*#3,#2-0.27*#3) and (#1+0.65*#3,#2-0.2*#3).. (#1+0.9*#3,#2-0.2*#3)
(#1+#3,#2-#3) .. controls (#1+#3,#2-0.95*#3) and (#1+0.96*#3,#2-0.85*#3)..
(#1+0.9*#3,#2-0.8*#3) (#1+#3,#2) .. controls (#1+#3,#2-0.05*#3) and
(#1+0.96*#3,#2-0.15*#3).. (#1+0.9*#3,#2-0.2*#3) (#1+0.1*#3,#2-0.2*#3) ..
controls  (#1+0.3*#3,#2-0.2*#3) and (#1+0.46*#3,#2-0.31*#3) ..
(#1+0.5*#3,#2-0.34*#3) (#1+0.1*#3,#2-0.2*#3) .. controls
(#1+0.1*#3,#2-0.38*#3) and (#1+0.256*#3,#2-0.49*#3) ..
(#1+0.275*#3,#2-0.505*#3) (#1+0.50*#3,#2-0.66*#3) .. controls
(#1+0.548*#3,#2-0.686*#3) and
(#1+0.70*#3,#2-0.8*#3)..(#1+0.9*#3,#2-0.8*#3) (#1+#3,#2-1*#3) .. controls
(#1+#3,#2-0.95*#3) and (#1+0.96*#3,#2-0.85*#3).. (#1+0.9*#3,#2-0.8*#3)
(#1+0.72*#3,#2-0.50*#3) .. controls (#1+0.80*#3,#2-0.56*#3) and
(#1+0.9*#3,#2-0.73*#3)..(#1+0.9*#3,#2-0.8*#3)(#1+0.72*#3,#2-0.50*#3) --
(#1+0.50*#3,#2-0.66*#3) -- (#1+0.275*#3,#2-0.505*#3) --
(#1+0.5*#3,#2-0.34*#3) -- (#1+0.72*#3,#2-0.50*#3)}
\def\unit(#1,#2){\draw (#1,#2) circle[radius=2pt] (#1,#2-0.07) -- (#1,#2-1)}
\begin{scope}[xshift=0cm, yshift=0cm]
\flip(0,0)[1];
\end{scope}
\begin{scope}[xshift=1.3cm,yshift=0cm]
\node at (0,-0.6){,};
\end{scope}
\begin{scope}[xshift=3.2cm, yshift=0cm]
\twisting(0,0)[1];
\end{scope}
\begin{scope}[xshift=4.5cm,yshift=0cm]
\node at (0,-0.6){,};
\end{scope}
\begin{scope}[xshift=6.8cm, yshift=0cm]
\doublemap(0,0)[\scriptstyle \mathcal{F}];
\end{scope}
\begin{scope}[xshift=10.6cm,yshift=0cm]
\node at (0,-0.5){and};
\end{scope}
\begin{scope}[xshift=13.4cm, yshift=0.4cm]
\unit(0.5,0); \singledoublemap(0,-1)[\scriptstyle \nu];
\end{scope}

\end{tikzpicture}
\end{equation*}
to denote the flip, a twisting map (See Definition~\ref{no unitarytwisted space}), and the maps $\mathcal{F}$ and $\nu$ of a crossed product system with preunit (See Definition~\ref{def preunit}).

\subsection{Weak Brzezi\'nski's crossed products}\label{subsection: a generalization of Brzezinskis crossed prodcts }
In this subsection we recall a very general notion of crossed product developed in~\cites{AFGR, FGR}, and we review its basic properties. Also we compare this construction with the one given in~\cite{Br}.

\begin{definition}\label{no unitarytwisted space} A triple $(A,V,\chi)\index{aa@$(A,V,\chi)$|dotfillboldidx}$, consisting of an associative algebra $A$, a $k$-vector space $V$ and a map $\chi\colon V\ot_k A\longrightarrow A\ot_k V\index{zw@$\chi$|dotfillboldidx}$, is a {\em twisted space} if
\begin{equation}\label{eqtwistingcond}
\chi\xcirc (V\ot_k\mu_A)= (\mu_A\ot_k V)\xcirc (A\ot_k\chi)\xcirc (\chi\ot_k A).
\end{equation}
In such a case we say that $\chi$ is a {\em twisting map}.
\end{definition}

Throughout this paper we assume that $A$ is a unitary algebra and $(A,V,\chi)$ is a twisted space.~A~di\-rect computation shows that $A\ot_k V$ is a non unitary $A$-bimodule via
$$
a'\xcdot (a\ot_k v) = a'a\ot_k v\qquad\text{and}\qquad (a\ot_k v)\xcdot a' = a\xcdot \chi(v\ot_k a').
$$
Let $\nabla_{\!\chi}\index{zw@$\chi$!a@$\nabla_{\chi}$|dotfillboldidx}$ be the endomorphism of $A\ot_k V$ defined by $\nabla_{\!\chi}(a\ot_k v) \coloneqq  a\xcdot \chi(v\ot_k 1_A)$. It is easy to see that $\nabla_{\!\chi}$ is a left and right $A$-linear idempotent, and that $\chi(V\ot_k A)\subseteq A\times V$, where $A\times V\coloneqq \nabla_{\!\chi}(A\ot_k V)\index{ad@$A\times V$|dotfillboldidx}$. Let $p_{\chi}\index{zw@$\chi$!i@$p_{\chi}$|dotfillboldidx}$ and $\imath_{\chi}\index{zw@$\chi$!p@$\imath_{\chi}$|dotfillboldidx}$ be the corestriction of $\nabla_{\!\chi}$ to $A\times V$ and the canonical inclusion of $A\times V$ in $A\ot_k V$, respectively. By the above discussion $p_{\chi}\xcirc \imath_{\chi}= \ide_{A\times V}$. Moreover $A\times V$ is an unitary $A$-subbimodule of $A\ot_k V$ and both $p_{\chi}$ and $\imath_{\chi}$ are $A$-bi\-mod\-ule morphisms.

\begin{remark}\label{nabla es multiplicar por 1_A} Note that $\cramped{\nabla_{\!\chi}(\bx)=\bx\xcdot 1_A}$ for all $\cramped{\bx\in A\ot_k V}$. So, $A\times V$ is the set of all the $\bx$'s in $\cramped{A\ot_k V}$ such that $\bx\xcdot 1_A=\bx$. The group $X_{\chi}\coloneqq \bigl\{\bx\in A\ot_k V:\bx\xcdot 1_A=0\bigr\}\index{zw@$\chi$!Xa@$X_{\chi}$|dotfillboldidx}$. is an $A$-subbimodule of $A\ot_k V$ and $A\ot_k V = A\times V\oplus X_{\chi}$. Moreover the projection of $A\ot_k V$ onto $A\times V$ along $X_{\chi}$ coincides with $p_{\chi}$.
\end{remark}

Let $C$ be a $k$-vector space and let $\nabla_{\hs C}\colon C\to C$ be an idempotent map. An associative product $\mu_C\colon C\ot_k C\longrightarrow C$ is {\em normalized with respect to $\nabla_{\hs C}$} if $\nabla_{\hs C}(cc') = cc' = \nabla_{\hs C}(c)\nabla_{\hs C}(c')$, for all $c,c'\in C$.

\smallskip

Let $C$ be an algebra. A map $\nu\colon k\to C$ is a {\em preunit} of $\mu_C$ if $\nu(1)$ is a central idempotent of~$C$.

\begin{definition}\label{def preunit} We say that a tuple $(A,V,\chi,\mathcal{F},\nu)\index{ac@$(A,V,\chi,\mathcal{F},\nu)$|dotfillboldidx}$ is a {\em crossed product system with preunit} if

\begin{enumerate}[itemsep=0.7ex, topsep=1.0ex, label=(\arabic*)]

\item $(A,V,\chi)$ is a twisted space,

\item $\mathcal{F}\colon V\ot_k V\longrightarrow A\ot_k V$ is a map with $\ima(\mathcal{F})\subseteq A\times V$,

\item $\nu\colon k\longrightarrow A\ot_k V\index{zn@$\nu$|dotfillboldidx}$ is a map satisfying
\begin{align}
& (\mu_A\ot_k V)\xcirc(A\ot_k \mathcal{F})\xcirc(\chi\ot_k V)\xcirc(V\ot_k \nu) = \nabla_{\!\chi}\xcirc (\eta_A\ot_k V),\label{preunit1}\\
& (\mu_A\ot_k V)\xcirc (A\ot_k \mathcal{F})\xcirc(\nu\ot_k V) = \nabla_{\!\chi}\xcirc (\eta_A\ot_k V),\label{preunit2}\\
& (\mu_A\ot_k V)\xcirc (A\ot_k \chi)\xcirc(\nu\ot_k A) = (\mu_A\ot_k V)\xcirc (A\ot_k \nu).\label{preunit3}
\end{align}
\end{enumerate}

\end{definition}

\begin{notation}\label{mult A times V} Given a crossed product system with preunit $(A,V,\chi,\mathcal{F},\nu)$, we let $A\ot_{\chi}^{\mathcal{F}} V\index{af@$A\ot_{\chi}^{\mathcal{F}} V$|dotfillboldidx}$ denote $A\ot_k V$ endowed with the (non necessarily associative) multiplication map $\mu_{A\ot_{\chi}^{\mathcal{F}} V}$ defined by
$$
\mu_{A\ot_{\chi}^{\mathcal{F}} V}\coloneqq  (\mu_A\ot_k V)\xcirc (\mu_A\ot_k\mathcal{F})\xcirc (A\ot_k\chi\ot_k V),
$$
and we let $A\times_{\chi}^{\mathcal{F}} V\index{ae@$A\times_{\chi}^{\mathcal{F}} V$|dotfillboldidx}$ denote $A\times V$ endowed with the (non necessarily associative) multiplication map $\mu_{A\times_{\chi}^{\mathcal{F}} V}$ induced by $\mu_{A\ot_{\chi}^{\mathcal{F}} V}$ (this is correct since clearly $\ima\bigl(\mu_{A\ot_{\chi}^{\mathcal{F}} V}\bigr)\subseteq A\times V$). In the sequel for simplicity we will write $E\index{ea@$E$|dotfillboldidx}$ instead of $A\times_{\chi}^{\mathcal{F}} V$ and $\mathcal{E}\index{eb@$\mathcal{E}$|dotfillboldidx}$ instead of $A\ot_{\chi}^{\mathcal{F}} V$.
\end{notation}

\begin{definition}\label{twisted module and cociclo condiciones} Let $(A,V,\chi,\mathcal{F},\nu)$ be a crossed product system with preunit. We say that $\mathcal{F}$ is a {\em cocycle} that satisfies the {\em twisted module condition} if
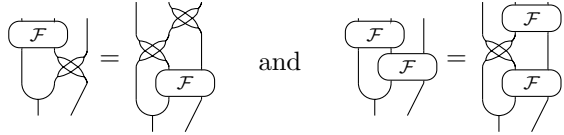
\begin{equation*}
\begin{tikzpicture}[scale=0.43]
\def\mult(#1,#2)[#3]{\draw (#1,#2) .. controls (#1,#2-0.555*#3/2) and (#1+0.445*#3/2,#2-#3/2) .. (#1+#3/2,#2-#3/2) .. controls (#1+1.555*#3/2,#2-#3/2) and (#1+2*#3/2,#2-0.555*#3/2) .. (#1+2*#3/2,#2) (#1+#3/2,#2-#3/2) -- (#1+#3/2,#2-2*#3/2)}
\def\doublemap(#1,#2)[#3]{\draw (#1+0.5,#2-0.5) node [name=doublemapnode,inner xsep=0pt, inner ysep=0pt, minimum height=10pt, minimum width=22pt,shape=rectangle,draw,rounded corners] {$#3$} (#1,#2) .. controls (#1,#2-0.075) .. (doublemapnode) (#1+1,#2) .. controls (#1+1,#2-0.075).. (doublemapnode) (doublemapnode) .. controls (#1,#2-0.925)..(#1,#2-1) (doublemapnode) .. controls (#1+1,#2-0.925).. (#1+1,#2-1)}
\def\doublesinglemap(#1,#2)[#3]{\draw (#1+0.5,#2-0.5) node [name=doublesinglemapnode,inner xsep=0pt, inner ysep=0pt, minimum height=11pt, minimum width=25pt,shape=rectangle,draw,rounded corners] {$#3$} (#1,#2) .. controls (#1,#2-0.075) .. (doublesinglemapnode) (#1+1,#2) .. controls (#1+1,#2-0.075).. (doublesinglemapnode) (doublesinglemapnode)-- (#1+0.5,#2-1)}
\def\twisting(#1,#2)[#3]{\draw (#1+#3,#2) .. controls (#1+#3,#2-0.05*#3) and (#1+0.96*#3,#2-0.15*#3).. (#1+0.9*#3,#2-0.2*#3) (#1,#2-1*#3) .. controls (#1,#2-0.95*#3) and (#1+0.04*#3,#2-0.85*#3).. (#1+0.1*#3,#2-0.8*#3) (#1+0.1*#3,#2-0.8*#3) ..controls (#1+0.25*#3,#2-0.8*#3) and (#1+0.45*#3,#2-0.69*#3) .. (#1+0.50*#3,#2-0.66*#3) (#1+0.1*#3,#2-0.8*#3) ..controls (#1+0.1*#3,#2-0.65*#3) and (#1+0.22*#3,#2-0.54*#3) .. (#1+0.275*#3,#2-0.505*#3) (#1+0.72*#3,#2-0.50*#3) .. controls (#1+0.75*#3,#2-0.47*#3) and (#1+0.9*#3,#2-0.4*#3).. (#1+0.9*#3,#2-0.2*#3) (#1,#2) .. controls (#1,#2-0.05*#3) and (#1+0.04*#3,#2-0.15*#3).. (#1+0.1*#3,#2-0.2*#3) (#1+0.5*#3,#2-0.34*#3) .. controls (#1+0.6*#3,#2-0.27*#3) and (#1+0.65*#3,#2-0.2*#3).. (#1+0.9*#3,#2-0.2*#3) (#1+#3,#2-#3) .. controls (#1+#3,#2-0.95*#3) and (#1+0.96*#3,#2-0.85*#3).. (#1+0.9*#3,#2-0.8*#3) (#1+#3,#2) .. controls (#1+#3,#2-0.05*#3) and (#1+0.96*#3,#2-0.15*#3).. (#1+0.9*#3,#2-0.2*#3) (#1+0.1*#3,#2-0.2*#3) .. controls (#1+0.3*#3,#2-0.2*#3) and (#1+0.46*#3,#2-0.31*#3) .. (#1+0.5*#3,#2-0.34*#3) (#1+0.1*#3,#2-0.2*#3) .. controls (#1+0.1*#3,#2-0.38*#3) and (#1+0.256*#3,#2-0.49*#3) .. (#1+0.275*#3,#2-0.505*#3) (#1+0.50*#3,#2-0.66*#3) .. controls (#1+0.548*#3,#2-0.686*#3) and (#1+0.70*#3,#2-0.8*#3)..(#1+0.9*#3,#2-0.8*#3) (#1+#3,#2-1*#3) .. controls (#1+#3,#2-0.95*#3) and (#1+0.96*#3,#2-0.85*#3).. (#1+0.9*#3,#2-0.8*#3) (#1+0.72*#3,#2-0.50*#3) .. controls (#1+0.80*#3,#2-0.56*#3) and (#1+0.9*#3,#2-0.73*#3)..(#1+0.9*#3,#2-0.8*#3)(#1+0.72*#3,#2-0.50*#3) -- (#1+0.50*#3,#2-0.66*#3) -- (#1+0.275*#3,#2-0.505*#3) -- (#1+0.5*#3,#2-0.34*#3) -- (#1+0.72*#3,#2-0.50*#3)}
\begin{scope}[yshift=-0.5cm]
\doublemap(0,0)[\scriptstyle\mathcal{F}];\draw (2,0) -- (2,-1);\draw (0,-1) -- (0,-2);\twisting(1,-1)[1];\mult(0,-2)[1];\draw (2,-2) -- (1.5,-3);
\end{scope}
\begin{scope}[xshift=0.2cm,yshift=0cm]
\node at (2.5,-1.8){=};
\end{scope}
\begin{scope}[xshift=0.5cm]
\draw (3,0) -- (3,-1);\twisting(4,0)[1];\twisting(3,-1)[1];\draw (5,-1) -- (5,-2);\draw (3,-2) -- (3,-3);\doublemap(4,-2)[\scriptstyle\mathcal{F}];\mult(3,-3)[1];\draw (5,-3) -- (4.5,-4);
\end{scope}
\begin{scope}[xshift=5.3cm,yshift=0cm]
\node at (2.6,-1.8){and};
\end{scope}
\begin{scope}[xshift=10.3cm, yshift=-0.5cm]
\doublemap(0,0)[\scriptstyle\mathcal{F}];\draw (2,0) -- (2,-1);\draw (0,-1) -- (0,-2);\doublemap(1,-1)[\scriptstyle\mathcal{F}];\mult(0,-2)[1];\draw (2,-2) -- (1.5,-3);
\end{scope}
\begin{scope}[xshift=10.8cm,yshift=0cm]
\node at (2.5,-1.8){=};
\end{scope}
\begin{scope}[xshift=11.1cm]
\draw (3,0) -- (3,-1);\doublemap(4,0)[\scriptstyle\mathcal{F}];\twisting(3,-1)[1];\draw (5,-1) -- (5,-2);\draw (3,-2) -- (3,-3);\doublemap(4,-2)[\scriptstyle\mathcal{F}];\mult(3,-3)[1];\draw (5,-3) -- (4.5,-4);
\end{scope}
\end{tikzpicture}
\end{equation*}
\noindent More precisely, the first equality says that $\mathcal{F}$ satisfies the twisted module condition, and the second one says that $\mathcal{F}$ is a cocycle.
\end{definition}

Let $(A,V,\chi,\mathcal{F},\nu)$ be a crossed product system with preunit and let
$$
\nabla_{\!\nu}\colon \mathcal{E}\longrightarrow \mathcal{E}\index{zn@$\nu$!zna@$\nabla_{\nu}$|dotfillboldidx},\quad \jmath'_{\nu}\colon A \to \mathcal{E}\index{zn@$\nu$!znb@$\jmath'_{\nu}$|dotfillboldidx},\quad \jmath_{\nu}\colon A\to E\index{zn@$\nu$!znc@$\jmath_{\nu}$|dotfillboldidx} \quad\text{and}\quad \gamma\colon V\to E\index{zc@$\gamma$|dotfillboldidx}
$$
be the arrows defined by
$$
\nabla_{\!\nu}(a\ot_k v)\coloneqq (a\ot_k v)\nu(1_k),\quad \jmath'_{\nu}(a)\coloneqq a\xcdot \nu(1_k),\quad \jmath_{\nu}(a)\coloneqq \nabla_{\! \chi}(\jmath'_{\nu}(a))\quad\text{and}\quad \gamma(v)\coloneqq \nabla_{\!\chi}(1_A \ot_k v).
$$

\begin{theorem}\label{prodcruz1} Let $(A,V,\chi,\mathcal{F},\nu)$ be a crossed product system with preunit. If $\mathcal{F}$ is a cocycle that satisfies the twisted module condition, then the following facts hold:

\begin{enumerate}[itemsep=0.7ex, topsep=1.0ex, label=\emph{(\arabic*)}]

\item $\mu_{\mathcal{E}}$ is a left and right $A$-linear associative product, that is normalized with respect to~$\nabla_{\!\chi}$.

\item $\nu$ is a preunit of $\mu_{\mathcal{E}}$, $\nabla_{\!\nu} = \nabla_{\!\chi}$ and $\nu(k)\subseteq E$.

\item $\mu_E$ is left and right $A$-linear, associative and has unit $1_E\coloneqq \nu(1_k)$.

\item The maps $\imath_{\chi}$ and $p_{\chi}$ are multiplicative.

\item $\jmath'_{\nu}$ is left and right $A$-linear, multiplicative, and $\jmath'_{\nu}(A) \subseteq E$.

\item $\jmath_{\nu}$ is left and right $A$-linear, multiplicative and unitary.

\item $\jmath_{\nu}(a)\bx = a\xcdot \bx$ and $\bx\jmath_{\nu}(a) = \bx\xcdot a$, for all $a\in A$ and $\bx\in E$.

\item $\chi(v\ot_k a) = (1_A\ot_k v)\jmath'_{\nu}(a)$ and $\mathcal{F}(v\ot_k w) = (1_A\ot_k v)(1_A\ot_k w)$.

\item $\chi(v\ot_k a) = \gamma(v)\jmath_{\nu}(a)$ and $\mathcal{F}(v\ot_k w) = \gamma(v)\gamma(w)$.

\end{enumerate}

\end{theorem}

\begin{proof} Except for items 4),~7) and the assertions about the right $A$-linearity in items 3), 5) and~6), whose proofs we leave to the reader, this follows immediately from \cite{FGR}*{Remark~3.10, one implication of Theorem~3.11 and Corollary~3.12}.
\end{proof}

\begin{remark}\label{jnu es igual a j'nu} By item~5) of the previous theorem, $\jmath_{\nu}(a) = \jmath'_{\nu}(a)$ for all $a\in A$.
\end{remark}

When the hypotheses of Theorem~\ref{prodcruz1} are fulfilled we say that $E$ is {\em the unitary crossed product of $A$ with $V$ associated with $\chi$ and $\mathcal{F}$}.

\begin{example}\label{crossed product system with unit} Let $(A,V,\chi)$ and $\mathcal{F}$ be as in items~1) and~2) of Definition~\ref{def preunit}, and let $1_V\in V$. If $\chi(1_V\ot_k a) = a\xcdot \chi(1_V\ot_k 1_A)$ for all $a\in A$, then we say that $(A,V,\chi,\mathcal{F},1_V)$ is a {\em a crossed product system with unit}; while if $\mathcal{F}(1_V\ot_k v) = \mathcal{F}(v\ot_k 1_V) = \chi(v\ot_k 1_A)$ for all $v\in V$, then we say that $\mathcal{F}$ is {\em normal}. If $(A,V,\chi,\mathcal{F},1_V)$ is a crossed product system with unit and $\mathcal{F}$ is a normal map that satisfies the twisted module condition, then the tuple $(A,V,\chi,\mathcal{F},\nu)$, where $\nu$ is the map defined by $\nu(1_k)\coloneqq \chi(1_V\ot_k 1_A)$, is a crossed product system with preunit.
\end{example}

\begin{example}\label{Productos cruzados de Brzezinki} The crossed product systems introduced by Brzezi\'nski in~\cite{Br} are the crossed product systems with unit $(A,V,\chi,\mathcal{F},1_V)$, such that $\chi(v\ot_k 1_A) = 1_A\ot_k v$ and $\mathcal{F}$ is a normal cocycle that satisfies the twisted module condition. The crossed product constructed from the datum $(A,V,\chi,\mathcal{F},1_V)$ is called {\em the Brzezi\'nski crossed product of $A$ with $V$ associated with $\chi$, $\mathcal{F}$ and $1_V$}. If $(A,V,\chi,\mathcal{F},1_V)$ is a Brzezi\'nski's crossed product system, then
$\chi(1_V\ot_k a) = a\ot_k 1_V$ and $\nabla = \ide_{A\ot V}$,
which implies $E = \mathcal{E}$. Suppose now that $V$ is an algebra with unit $1_V$. A Brzezi\'nski's crossed product in which $\mathcal{F}$ is the {\em trivial cocycle} given by $\mathcal{F}(v\ot_k w) \coloneqq  1_A\ot_k vw$, is called a {\em twisted tensor product}. In this case $\mathcal{F}$ is automatically a normal cocycle and the twisted module condition says that
$$
\chi\xcirc (\mu_V\ot_k A) = (A\ot_k \mu_V)\xcirc (\chi\ot_k V)\xcirc (V\ot_k \chi).
$$
When we deal with twisted tensor products we write $A\ot_{\chi} V\index{ag@$A\ot_{\chi} V$|dotfillboldidx}$ instead $A\ot_{\chi}^{\mathcal{F}} V$.
\end{example}

In the rest of the section we assume that $(A,V,\chi,\mathcal{F},\nu)$ is a crossed product system with preunit, in which $\mathcal{F}$ is a cocycle that satisfies the twisted module condition.

\begin{remark}\label{propiedad de gamma'} By definition $\nabla_{\!\chi}(a\ot_k v) = a\xcdot \gamma(v)$. Consequently, if $a\xcdot\gamma(v) = \sum_l a_{(l)}\ot_k v_{(l)}$, then
\begin{equation}
\jmath_{\nu}(a)\gamma(v) = a\xcdot\gamma(v) = \sum_l a_{(l)}\xcdot\gamma(v_{(l)}) = \sum_l \jmath_{\nu}(a_{(l)}) \gamma(v_{(l)});\label{eq4}
\end{equation}
if $\chi(v\ot_k a) = \sum_l a_{(l)}\ot_k v_{(l)}$, then, by the first equality in Theorem~\ref{prodcruz1}(9),
\begin{equation}
\gamma(v)\jmath_{\nu}(a) = \chi(v\ot_k a) = \nabla_{\!\chi}\bigl(\chi(v\ot_k a)\bigr) = \sum_l \jmath_{\nu}(a_{(l)})\gamma(v_{(l)});\label{eq12}
\end{equation}
and if $\mathcal{F}(v\ot_k w) = \sum_l a_{(l)}\ot_k u_{(l)}$, then, by the second equality in Theorem~\ref{prodcruz1}(9),
\begin{equation}
\gamma(v)\gamma(w) = \mathcal{F}(v\ot_k w) = \nabla_{\!\chi}\bigl(\mathcal{F}(v\ot_k w)\bigr) = \sum_l \jmath_{\nu}(a_{(l)}) \gamma(u_{(l)}).\label{eq12'}
\end{equation}
\end{remark}

\begin{remark}\label{caso jmath_{nu}(K) subseteq jmath_{nu}(K)gamma(V)} If $\nu(1_k) = \sum_l \lambda_{(l)}\ot_k v_{(l)}$, then $1_E = \nu(1) = \nabla_{\!\chi}(\nu(1)) = \sum_l \jmath_{\nu}(\lambda_{(l)})\gamma(v_{(l)})$. So,
$$
\jmath_{\nu}(a) = \sum_l \jmath_{\nu}(a\lambda_{(l)})\gamma(v_{(l)})\qquad\text{for all $a\in A$.}
$$
Consequently, if $\nu(1_k)\in K\ot_kV$, then $\jmath_{\nu}(K)\subseteq \jmath_{\nu}(K)\gamma(V)$.
\end{remark}

\begin{proposition}\label{R tensor V esta incluido} For each subalgebra $R$ of $A$, we have $(R\ot_k V)\cap E\subseteq R \xcdot \gamma(V) = \jmath_{\nu}(R)\gamma(V)$. Moreover, if $\gamma(V)\subseteq R\ot_k V$, then the equality holds.
\end{proposition}

\begin{proof} Since $\nabla_{\!\chi}$ is a left $A$-linear projection with image $E$, we have
$$
(R\ot_k V)\cap E = \nabla_{\!\chi}\bigl((R\ot_k V)\cap E\bigr) \subseteq \nabla_{\!\chi}(R\ot_k V) = R \xcdot \gamma(V) = \jmath_{\nu}(R)\gamma(V).
$$
The last assertion holds since $\gamma(V)\subseteq R\ot_k V$ implies $R\xcdot \gamma(V)\subseteq R\xcdot (R\ot_k V) = R\ot_k V$.
\end{proof}

\begin{definition}\label{estable bajo chi} We say that a subalgebra $R$ of $A$ is {\em stable under $\chi$} if $\chi(V\ot_k R)\subseteq R\ot_k V$.
\end{definition}

\begin{lemma}\label{gamma(V)iota(K) C= iota(K)gamma(V)} If $R$ is a stable under $\chi$ subalgebra of $A$, then $\gamma(V)\jmath_{\nu}(R)\subseteq (R\ot_k V)\cap E = \jmath_{\nu}(R)\gamma(V)$.
\end{lemma}

\begin{proof} Since $R$ is stable under $\chi$, we know that $\gamma(V)\subseteq R\ot_k V$. So,
$$
\gamma(V)\jmath_{\nu}(R) = \chi(V\ot_k R)\subseteq (R\ot_k V)\cap E =\jmath_{\nu}(R)\gamma(V),
$$
where the first equality holds by Theorem~\ref{prodcruz1}(9); and the last one, by Proposition~\ref{R tensor V esta incluido}.
\end{proof}

\begin{lemma}\label{lema gamma(V)gamma(V)} Let $R$ be a $k$-subalgebra of $A$. If $\mathcal{F}(V\ot_k V) \subseteq R\ot_k V$, then
$$
\gamma(V)\gamma(V) \subseteq (R\ot_k V)\cap E \subseteq \jmath_{\nu}(R)\gamma(V).
$$
\end{lemma}

\begin{proof} By Theorem~\ref{prodcruz1}(9), the fact that $\mathcal{F}(V\ot_k V)\subseteq (R\ot_k V)\cap E$ and Proposition~\ref{R tensor V esta incluido}.
\end{proof}

\begin{notation}\label{iteracion de chi} We let $\chi_{_{jl}}\colon V^{\ot_k j}\ot_k A^{\ot_k l}\longrightarrow A^{\ot_k l}\ot_k V^{\ot_k j}\index{zw@$\chi$!zwa@$\chi_{_{jl}}$|dotfillboldidx}$ denote the map defined by:
\begin{align*}
&\chi_{_{11}}\coloneqq \chi,\\
&\chi_{_{j+1,1}}\coloneqq \bigl(\chi_{_{j1}}\ot_k V\bigr)\xcirc\bigl(V^{\ot_k j}\ot_k\chi\bigr) &&\text{for $j\ge 1$},\\
&\chi_{_{j,l+1}}\coloneqq \bigl(A^{\ot_k l}\ot_k\chi_{_{j1}}\bigr)\xcirc\bigl(\chi_{_{jl}}\ot_k A\bigr) &&\text{for $j,l\ge 1$}.
\end{align*}
Furthermore, we set $\chi_{_{j0}}\coloneqq \ide_{V^{\ot_k j}}$ and $\chi_{_{0l}}\coloneqq \ide_{A^{\ot_k l}}$ for all $j,l\ge 1$.
\end{notation}

\begin{proposition}\label{twisting inducido} Let $K$ be a subalgebra of $A$ and let $\ov{A}\coloneqq A/K$. If $K$ is stable under $\chi$, then each~$\chi_{jl}$ induces maps
$$
\bar{\chi}_{jl}\colon V^{\ot_k j}\ot_k A^{\ot_{\hs K} l}\longrightarrow A^{\ot_{\hs K} l}\ot_k V^{\ot_k j} \index{zw@$\chi$!zwb@$\bar{\chi}_{_{jl}}$|dotfillboldidx} \qquad\text{and}\qquad \ov{\chi}_{jl}\colon V^{\ot_{\hs k} j}\ot_k \ov{A}^{\ot_{\hs K} l}\longrightarrow \ov{A}^{\ot_{\hs K} l}\ot_k V^{\ot_{\hs k} j} \index{zw@$\chi$!zwc@$\overline{\chi}_{_{jl}}$|dotfillboldidx}.
$$
\end{proposition}

\begin{proof} Straightforward.
\end{proof}

\subsection{Mixed complexes}\label{subsection: Mixed complexes}
In this subsection we recall briefly the notion of mixed complex. For more details about this concept we refer to~\cites{B, K}.

\smallskip

A {\em mixed complex} $\mathcal{X}\coloneqq  (X,b,B)\index{xa@$\mathcal{X}$|dotfillboldidx}$ is a graded $k$-module $(X_n)_{n\ge 0}$, endowed with morphisms
$$
b\colon X_n\longrightarrow X_{n-1}\qquad\text{and}\qquad B\colon X_n\longrightarrow X_{n+1},
$$
such that
$b\xcirc b = 0$, $B\xcirc B = 0$ and $B\xcirc b + b\xcirc B = 0$. A {\em morphism of mixed complexes} $g\colon (X,b,B)\longrightarrow (Y,d,D)$ is a family of maps $g\colon X_n\to Y_n$, such that $d\xcirc g = g\xcirc b$ and $D\xcirc g= g\xcirc B$. Let $u$ be a degree~$2$ variable. A mixed complex $\mathcal{X}\coloneqq  (X,b,B)$ determines a double complex
$$
\begin{tikzpicture}
\begin{scope}[yshift=-0.47cm,xshift=-6cm]
\draw (0.5,0.5) node {$\BP(\mathcal{X})=$};
\end{scope}
\begin{scope}
\matrix(BPcomplex) [matrix of math nodes,row sep=2.5em, text height=1.5ex, text
depth=0.25ex, column sep=2.5em, inner sep=0pt, minimum height=5mm,minimum width =9.5mm]
{&\vdots &\vdots &\vdots &\vdots\\
\cdots & X_3 u^{-1} & X_2 u^0 & X_1 u^{} & X_0 u^2\\
\cdots & X_2 u^{-1} & X_1 u^0 & X_0 u\\
\cdots & X_1 u^{-1} & X_0 u^0\\
\cdots & X_0 u^{-1}\\};
\draw[->] (BPcomplex-1-2) -- node[right=1pt,font=\scriptsize] {$b$} (BPcomplex-2-2);
\draw[->] (BPcomplex-1-3) -- node[right=1pt,font=\scriptsize] {$b$} (BPcomplex-2-3);
\draw[->] (BPcomplex-1-4) -- node[right=1pt, font=\scriptsize] {$b$} (BPcomplex-2-4);
\draw[->] (BPcomplex-1-5) -- node[right=1pt, font=\scriptsize] {$b$} (BPcomplex-2-5);
\draw[<-] (BPcomplex-2-1) -- node[above=1pt,font=\scriptsize] {$B$} (BPcomplex-2-2);
\draw[<-] (BPcomplex-2-2) -- node[above=1pt,font=\scriptsize] {$B$} (BPcomplex-2-3);
\draw[<-] (BPcomplex-2-3) -- node[above=1pt,font=\scriptsize] {$B$} (BPcomplex-2-4);
\draw[<-] (BPcomplex-2-4) -- node[above=1pt,font=\scriptsize] {$B$} (BPcomplex-2-5);
\draw[->] (BPcomplex-2-2) -- node[right=1pt,font=\scriptsize] {$b$} (BPcomplex-3-2);
\draw[->] (BPcomplex-2-3) -- node[right=1pt,font=\scriptsize] {$b$} (BPcomplex-3-3);
\draw[->] (BPcomplex-2-4) -- node[right=1pt, font=\scriptsize] {$b$} (BPcomplex-3-4);
\draw[<-] (BPcomplex-3-1) -- node[above=1pt,font=\scriptsize] {$B$} (BPcomplex-3-2);
\draw[<-] (BPcomplex-3-2) -- node[above=1pt,font=\scriptsize] {$B$} (BPcomplex-3-3);
\draw[<-] (BPcomplex-3-3) -- node[above=1pt,font=\scriptsize] {$B$} (BPcomplex-3-4);
\draw[->] (BPcomplex-3-2) -- node[right=1pt,font=\scriptsize] {$b$} (BPcomplex-4-2);
\draw[->] (BPcomplex-3-3) -- node[right=1pt,font=\scriptsize] {$b$} (BPcomplex-4-3);
\draw[<-] (BPcomplex-4-1) -- node[above=1pt,font=\scriptsize] {$B$} (BPcomplex-4-2);
\draw[<-] (BPcomplex-4-2) -- node[above=1pt,font=\scriptsize] {$B$} (BPcomplex-4-3);
\draw[->] (BPcomplex-4-2) -- node[right=1pt,font=\scriptsize] {$b$} (BPcomplex-5-2);
\draw[<-] (BPcomplex-5-1) -- node[above=1pt,font=\scriptsize] {$B$} (BPcomplex-5-2);
\end{scope}
\begin{scope}[yshift=-0.47cm,xshift=4cm]
\draw (0.5,0.5) node {,};
\end{scope}
\end{tikzpicture}
\index{xa@$\mathcal{X}$!bp@$\BP(\mathcal{X})$|dotfillboldidx}
$$
where $b(\bx u^i)\coloneqq  b(\bx)u^i$ and $B(\bx u^i)\coloneqq  B(\bx)u^{i-1}$. By deleting the positively numbered columns we obtain a subcomplex $\BN(\mathcal{X})\index{xa@$\mathcal{X}$!bn@$\BN(\mathcal{X})$|dotfillboldidx}$ of $\BP(\mathcal{X})$. Let $\BN'(\mathcal{X})\index{xa@$\mathcal{X}$!bn'@$\BN'(\mathcal{X})$|dotfillboldidx}$ be the kernel of the canonical surjection from $\BN(\mathcal{X})$ to $(X,b)$. The quotient double complex $\BP(\mathcal{X})/\BN'(\mathcal{X})$ is denoted by $\BC(\mathcal{X})\index{xa@$\mathcal{X}$!bc@$\BC(\mathcal{X})$|dotfillboldidx}$. The homology groups $\HC_*(\mathcal{X})\index{xa@$\mathcal{X}$!hc@$\HC_*(\mathcal{X})$|dotfillboldidx}$, $\HN_*(\mathcal{X})\index{xa@$\mathcal{X}$!hn@$\HN_*(\mathcal{X})$|dotfillboldidx}$ and $\HP_*(\mathcal{X})\index{xa@$\mathcal{X}$!hb@$\HP_*(\mathcal{X})$|dotfillboldidx}$, of the total complexes of $\BC(\mathcal{X})$, $\BN(\mathcal{X})$ and $\BP(\mathcal{X})$ respectively, are called the {\em cyclic}, {\em negative} and {\em periodic homology groups} of $\mathcal{X}$. The homology $\HH_*(\mathcal{X})\index{xa@$\mathcal{X}$!hh@$\HH_*(\mathcal{X})$|dotfillboldidx}$, of $(X,b)$, is called the {\em Hochschild homology} of $\mathcal{X}$. Finally, it is clear that a morphism $f\colon\mathcal{X}\to\mathcal{Y}$ of mixed complexes induces a morphism from the double complex $\BP(\mathcal{X})$ to the double complex $\BP(\mathcal{Y})$.

\smallskip

Let $C$ be an algebra. If $K$ is a subalgebra of $C$ we will say that $C$ is a $K$-algebra. Given a $K$-bi\-module $M$, we let $M\ot_{\hs K}\index{ma@$M\ot$|dotfillboldidx}$ denote the quotient $M/[M,K]$, where $[M,K]$ is the $k$-submodule of $M$ generated by all the commutators $m\lambda -\lambda m$, with $m\in M$ and $\lambda\in K$. Moreover, for $m\in M$, we let $[m]\index{mb@$[m]$|dotfillboldidx}$ denote the class of $m$ in $M\ot_{\hs K}$.

By definition, the {\em normalized mixed complex of the $K$-algebra $C$} is $(C\ot_{\hs K}\ov{C}^{\ot_{\hs K} *}\ot_{\hs K},b_*,B_*)$, where $\ov{C}\coloneqq C/K$, $b_*\index{ba@$b_*$|dotfillboldidx}$ is the canonical Hochschild boundary map and the Connes operator $B_*\index{bb@$B_*$|dotfillboldidx}$ is given by
$$
B\bigl([c_0\ot_{\hs K}\cdots \ot_{\hs K} c_r]\bigr)\coloneqq \sum_{i=0}^r (-1)^{ir} [1\ot_{\hs K} c_i\ot_{\hs K}\cdots \ot_{\hs K} c_r\ot_{\hs K} c_0\ot_{\hs K} c_1\ot_{\hs K} \cdots \ot_{\hs K} c_{i-1}].
$$
The {\em cyclic}, {\em negative}, {\em periodic} and {\em Hochschild homology groups} $\HC^K_*(C)\index{hcc@$\HC^K_*(C)$|dotfillboldidx}$, $\HN^K_*(C)\index{hd@$\HN^K_*(C)$|dotfillboldidx}$, $\HP^K_*(C)\index{he@$\HP^K_*(C)$|dotfillboldidx}$ and $\HH^K_*(C)\index{hf$@$\HH^K_*(C)$|dotfillboldidx}$, of $C$, are the respective homology groups of $(C\ot_{\hs K}\ov{C}^{\ot_{\hs K} *}\ot_{\hs K},b_*,B_*)$.

\subsection{The perturbation lemma}
Next, we recall the perturbation lemma. We present the version given in~\cite{C}.

\smallskip

A {\em homotopy equivalence data}
\begin{equation}\label{(a)}
\begin{tikzpicture}[baseline=(current bounding box.center)]
\draw (0,0) node[] {$(Y,\partial)$};\draw (2.5,0) node {$(X,d)$};\draw[<-] (0.6,0.12) -- node[above=-2pt,font=\scriptsize] {$p$}(1.9,0.12);\draw[->] (0.6,-0.12) -- node[below=-2pt,font=\scriptsize] {$i$}(1.9,-0.12);\draw (4,0) node {$X_*$};\draw (6.5,0) node {$X_{*+1}$};\draw[->] (4.4,0) -- node[above=-2pt,font=\scriptsize] {$h$}(5.9,0);
\end{tikzpicture}
\end{equation}
consists of the following:

\begin{enumerate}

\smallskip

\item Chain complexes $(Y,\partial)$, $(X,d)$ and quasi-isomorphisms $i$, $p$ between them,

\smallskip

\item A homotopy $h$ from $i\xcirc p$ to $\ide$.
\end{enumerate}

\smallskip

A {\em perturbation} of~\eqref{(a)} is a map $\delta\colon X_*\to X_{*-1}$ such that $(d+\delta)^2 = 0$. We call it {\em small} if $\ide -\delta\xcirc h$ is invertible. In this case we write $A \coloneqq  (\ide -\delta\xcirc h)^{-1}\xcirc\delta$ and we consider the diagram
\begin{equation}\label{(b)}
\begin{tikzpicture}[baseline=(current bounding box.center)]
\draw (0,0) node[] {$(Y,\partial^1)$};\draw (2.5,0) node {$(X,d)$};\draw[<-] (0.6,0.12) -- node[above=-2pt,font=\scriptsize] {$p^1$}(1.9,0.12);\draw[->] (0.6,-0.12) -- node[below=-2pt,font=\scriptsize] {$i^1$}(1.9,-0.12);\draw (4,0) node {$X_*$};\draw (6.5,0) node {$X_{*+1}$,};\draw[->] (4.4,0) -- node[above=-2pt,font=\scriptsize] {$h^1$}(5.9,0);
\end{tikzpicture}
\end{equation}
where $\partial^1\coloneqq \partial + p\xcirc A\xcirc i$, $i^1\coloneqq  i + h\xcirc A\xcirc i$, $p^1\coloneqq  p + p\xcirc A\xcirc h$ and $h^1\coloneqq  h + h\xcirc A\xcirc h$.

\smallskip

A {\em deformation retract} is a homotopy equivalence data such that $p\xcirc i =\ide$. A deformation retract is called {\em special} if $h\xcirc i = 0$, $p\xcirc h = 0$ and $h\xcirc h = 0$.

\smallskip

In all the cases considered in this paper the morphism $\delta\xcirc h$ is locally nilpotent (in other for all $x\in X_*$ there exists $n\in \mathds{N}$ such that $(\delta\xcirc h)^n(x)=0$). Consequently, $(\ide -\delta\xcirc h)^{-1} =\sum_{n=0}^{\infty} (\delta\xcirc h)^n$.

\begin{theorem}[{\cite{C}}]\label{lema de perturbacion} If $\delta$ is a small perturbation of~\eqref{(a)}, then the diagram~\eqref{(b)} is an homotopy equivalence data. Furthermore, if~\eqref{(a)} is a special deformation retract, then so it is~\eqref{(b)}.
\end{theorem}

\section{A resolution for a weak Brzezi\'nski's crossed product}\label{section: A resolution for a general crossed product}
Let $A$ be an algebra, $V$ a $k$-vector space, $K$ a sub\-algebra of $A$ and $(A,V,\chi,\mathcal{F},\nu)$ a crossed product system with preunit. Assume that $\mathcal{F}$ is a cocycle that satisfies the twisted module condition, that $K$ is stable under $\chi$ and that $1_E\in K\ot_k V$, where $E$ is as in Notation~\ref{mult A times V}. Let $\jmath_{\nu}\colon A\to E$ and $\gamma\colon V\to E$ be as above of Theorem~\ref{prodcruz1}. By Remark~\ref{jnu es igual a j'nu} we know that $\jmath_{\nu}(a) = a\cdot\nu(1)$, for all $a\in A$; while, by the discussions above Remark~\ref{nabla es multiplicar por 1_A} we know that $E$ is an $A$-bimodule, which implies that it is also a $K$-bimodule. Moreover, by Theorem~\ref{prodcruz1}(7) the left and right actions of $A$ on $E$ coincide with those obtained through the morphism $\jmath_{\nu}$. Let $\Upsilon\index{zu@$\Upsilon$|dotfillboldidx}$ denote the family of all the epimorphisms of $E$-bimodules which split as $(E,K)$-bimodule maps. In this section we construct a $\Upsilon$-relative projective resolution $(X_*,d_*)$, of $E$ as an $E$-bimodule, simpler than the normalized bar resolution of $E$. Also we will construct comparison maps between both resolutions. We let $\ot$ denote $\ot_{\hs K}$ and we set
$$
\ov{A}\coloneqq \frac{A}{K}\index{ak@$\ov{A}$|dotfillboldidx},\qquad\ov{E}\coloneqq \frac{E}{\jmath_{\nu}(K)}\index{ec@$\ov{E}$|dotfillboldidx} \qquad\text{and}\qquad \wt{E}\coloneqq \frac{E}{\jmath_{\nu}(A)}\index{ed@$\wt{E}$|dotfillboldidx}.
$$

\begin{notations}\label{notaciones basicas} We will use the following notations:

\begin{enumerate}[itemsep=0.7ex, topsep=1.0ex, label=(\arabic*)]

\item For each $x\in E$, we let $\ov{x}\index{eg@$\ov{x}$|dotfillboldidx}$ and $\wt{x}\index{eh@$\wt{x}$|dotfillboldidx}$ denote its class in $\ov{E}$ and $\wt{E}$, respectively. Similarly, for $a\in A$, we let $\ov{a}\index{aq@$\ov{a}$|dotfillboldidx}$ denote its class in $\ov{A}$.

\item Given $x_1,\dots, x_s\in E$ and $1\le i\le j\le s$, we set $\ov{\bx}_{ij}\coloneqq \ov{x}_i \ot \cdots \ot \ov{x}_j\index{xu@$\ov{\bx}_{ij}$|dotfillboldidx}$ and $\wt{\bx}_{ij}\coloneqq \wt{x}_i \ot_{\hs A}\cdots \ot_{\hs A} \wt{x}_j\index{xv@$\wt{\bx}_{ij}$|dotfillboldidx}$.

\item Given $v_1,\dots,v_s\in V$ and $1\le i\le j\le s$, we set $\bv_{ij}\coloneqq v_i \ot_{\hs k} \cdots \ot_{\hs k} v_j\index{vb@$\bv_{ij}$|dotfillboldidx}$.

\item Given $a_1,\dots, a_r\in A$ and $1\le i\le j\le r$, we set $\ba_{ij}\coloneqq  a_i\ot\cdots\ot a_j\index{as@$\ba_{ij}$|dotfillboldidx}$, and we let $\ov{\ba}_{ij}\index{at@$\ov{\ba}_{ij}$|dotfillboldidx}$ denote the class of $\ba_{ij}$ in $\ov{A}^{\ot {j-i+1}}$.

\item We let $\ov{\gamma}\colon V\to\ov{E}\index{zc@$\gamma$!zc1@$\ov{\gamma}$|dotfillboldidx}$ and $\wt{\gamma}\colon V\to \wt{E}\index{zc@$\gamma$!zc2@$\wt{\gamma}$|dotfillboldidx}$ denote the maps induced by $\gamma$.

\item Given $v_1,\dots,v_s\in V$ and $1\le i\le j\le s$, we write $\ov{\gamma}(\bv_{ij})\index{zc@$\gamma$!zc3@$\ov{\gamma}(\bv_{ij})$|dotfillboldidx}$, $\gamma_{\hs A}(\bv_{ij})\index{zc@$\gamma$!zc4@$\gamma_{A}(\bv_{ij})$|dotfillboldidx}$ and $\wt{\gamma}_{\hs A}(\bv_{ij})\index{zc@$\gamma$!zc5@$\wt{\gamma}_{\hs A}(\bv_{ij})$|dotfillboldidx}$ to mean
$$
\qquad \ov{\gamma}(v_i)\ot \cdots\ot \ov{\gamma}(v_j),\quad \gamma(v_i)\ot_{\hs A}\cdots\ot_{\hs A}\gamma(v_j) \quad \text{and}\quad \wt{\gamma}(v_i)\ot_{\hs A}\cdots \ot_{\hs A}\wt{\gamma}(v_j),
$$
respectively.

\item We let $\ov{\jmath}_{\nu}\colon \ov{A}\to \ov{E}\index{zn@$\nu$!znd@$\ov{\jmath}_{\nu}$|dotfillboldidx}$
 denote the map induced by $\jmath_{\nu}$.

\item  Given $a_1,\dots,a_r\in A$ and $1\le i<j\le r$, we set $\ov{\jmath}_{\nu}(\ba_{ij})\coloneqq \ov{\jmath}_{\nu}(\ov{a}_i)\ot\cdots\ot \ov{\jmath}_{\nu}(\ov{a}_j)\index{zn@$\nu$!zne@$\ov{\jmath}_{\nu|}(\ba_{ij})$|dotfillboldidx}$.

\item We let $\wt{\mu}_E\colon E\ot_{\hs A} E\to E\index{zma@$\wt{\mu}_E$|dotfillboldidx}$, $\ov{\mu}_E\colon E\ot E\to E\index{zmb@$\ov{\mu}_E$|dotfillboldidx}$ and $\ov{\mu}_A\colon A\ot A\to A\index{zmac@$\ov{\mu}_A$|dotfillboldidx}$ denote the maps induced by the multiplication maps $\mu_E$ and $\mu_A$.

\end{enumerate}

\end{notations}

For all $s\ge 0$, we set
$Y'_s\coloneqq  E\ot_{\hs A} E^{\ot_{\hs A} s}\ot_{\hs A} E$\index{ya@$Y'_s$|dotfillboldidx} and $Y_s\coloneqq  E\ot_{\hs A} \wt{E}^{\ot_{\hs A} s}\ot_{\hs A} E$\index{yb@$Y_s$|dotfillboldidx},
while, for all $r,s\ge 0$, we set
$X'_{rs}\coloneqq  E\ot_{\hs A} E^{\ot_{\hs A} s} \ot \ov{A}^{\ot r}\ot E$ and $X_{rs}\coloneqq  E\ot_{\hs A} \wt{E}^{\ot_{\hs A} s} \ot \ov{A}^{\ot r}\ot E\index{xb@$X'_{rs}$|dotfillboldidx}$\index{xc@$X_{rs}$|dotfillboldidx}.

\begin{remark} For each $s\ge 0$, we consider $E\ot_{\hs k} V^{\ot_{\hs k} s}$ as a right $A$-module via
$
(x\ot_{\hs k} \bv_{1s})\xcdot a\coloneqq \sum x\xcdot a^{(l)}\ot_{\hs k} \bv_{1s}^{(l)},
$
where $\sum_l a^{(l)} \ot_k \bv_{1s}^{(l)} \coloneqq  \chi(\bv_{1s}\ot_k a)$.
\end{remark}

\begin{proposition}\label{algu ident} The following assertions hold:

\begin{enumerate}[itemsep=0.7ex, topsep=1.0ex, label=\emph{(\arabic*)}]

\item For each $s\ge 0$, the map $\theta_s\colon \bigl(E\ot_{\hs k} V^{\ot_{\hs k} s}\bigr) \ot_{\hs A} E\longrightarrow Y'_s$, given by
$$
\theta_s\bigl((1_E\ot_{\hs k} \bv_{1s})\ot_{\hs A} 1_E \bigr)\coloneqq  1_E \ot_{\hs A} \gamma_{\hs A}(\bv_{1s})\ot_{\hs A} 1_E,
$$
is an isomorphism of $E$-bimodules.

\item For each $r,s\ge 0$, the map $\theta_{rs}\colon \bigl(E\ot_{\hs k} V^{\ot_{\hs k} s}\bigr) \ot \ov{A}^{\ot r}\ot E \longrightarrow X'_{rs}$, given by
$$
\theta_{rs}\bigl((1_E\ot_{\hs k}\bv_{1s}) \ot \ov{\ba}_{1r}\ot 1_E\bigr)\coloneqq  1_E \ot_{\hs A} \gamma_{\hs A}(\bv_{1s}) \ot\ov{\ba}_{1r}\ot 1_E,
$$
is an isomorphism of $E$-bimodules.

\end{enumerate}
\end{proposition}

\begin{proof} We prove item~(2) and leave the proof of item~(1), which is similar, to the reader. By~Re\-mark~\ref{nabla es multiplicar por 1_A} we know that $X'_{rs} = E\ot_{\hs A}(A\ot_k V)^{\ot_{\hs A} s}\ot \ov{A}^{\ot r}\ot E$. So, we have a canonical isomorphism
\begin{equation}\label{iso}
X'_{rs} \simeq\bigl(E\ot_{\hs k} V^{\ot_{\hs k} s}\bigr) \ot \ov{A}^{\ot r}\ot E.
\end{equation}
Let $a_1,\dots,a_r\in A$ and $v_1,\dots,v_s\in V$. Since, $\gamma(v) = (1\ot_k v)\xcdot 1_A \equiv 1\ot_k v\pmod{X_{\chi}}$ for all $v\in V$, the iso\-morphism~\eqref{iso} maps $1_E \ot_{\hs A} \gamma_{\hs A}(\bv_{1s})\ot \ov{\ba}_{1r}\ot 1_E$ to $(1_E\ot_{\hs k} \bv_{1s})\ot \ov{\ba}_{1r}\ot 1_E$, and so it is the inverse of the map $\theta_{rs}$.
\end{proof}

\begin{remark} The following assertions hold:

\begin{enumerate}[itemsep=0.7ex, topsep=1.0ex, label=(\arabic*)]

\item $Y'_0=Y_0\simeq E\ot_{\hs A} E$ and $X'_{r0}=X_{r0}\simeq E\ot\ov{A}^{\ot r}\ot E$ as $E$-bimodules.

\item $X'_{0s}\simeq E\ot_{\hs A} E^{\ot_{\hs A} s}\ot E$  and $X_{0s}\simeq E\ot_{\hs A}\wt{E}^{\ot_{\hs A} s}\ot E$ as $E$-bimodules.

\end{enumerate}
\end{remark}

\medskip

Consider the diagram of $E$-bimodules and $E$-bimodule maps
$$
\begin{tikzpicture}
\begin{scope}[yshift=0cm,xshift=0cm]
\matrix(BPcomplex) [matrix of math nodes,row sep=2.5em, text height=1.5ex, text
depth=0.25ex, column sep=2.5em, inner sep=0pt, minimum height=5mm,minimum width =8mm]
{\vdots\\
Y_2 & X_{02} & X_{12} &\cdots\\
Y_1 & X_{01} & X_{11} &\cdots\\
Y_0 & X_{00} & X_{10} &\cdots,\\};
\draw[->] (BPcomplex-1-1) -- node[right=1pt,font=\scriptsize] {$-\partial_3$} (BPcomplex-2-1);
\draw[->] (BPcomplex-2-1) -- node[right=1pt,font=\scriptsize] {$-\partial_2$} (BPcomplex-3-1);
\draw[->] (BPcomplex-3-1) -- node[right=1pt, font=\scriptsize] {$-\partial_1$} (BPcomplex-4-1);
\draw[<-] (BPcomplex-2-1) -- node[above=1pt,font=\scriptsize] {$\upsilon_2$} (BPcomplex-2-2);
\draw[<-] (BPcomplex-2-2) -- node[above=1pt,font=\scriptsize] {$d^0_{12}$} (BPcomplex-2-3);
\draw[<-] (BPcomplex-2-3) -- node[above=1pt,font=\scriptsize] {$d^0_{22}$} (BPcomplex-2-4);
\draw[<-] (BPcomplex-3-1) -- node[above=1pt,font=\scriptsize] {$\upsilon_1$} (BPcomplex-3-2);
\draw[<-] (BPcomplex-3-2) -- node[above=1pt,font=\scriptsize] {$d^0_{11}$} (BPcomplex-3-3);
\draw[<-] (BPcomplex-3-3) -- node[above=1pt,font=\scriptsize] {$d^0_{21}$} (BPcomplex-3-4);
\draw[<-] (BPcomplex-4-1) -- node[above=1pt,font=\scriptsize] {$\upsilon_0$} (BPcomplex-4-2);
\draw[<-] (BPcomplex-4-2) -- node[above=1pt,font=\scriptsize] {$d^0_{10}$} (BPcomplex-4-3);
\draw[<-] (BPcomplex-4-3) -- node[above=1pt,font=\scriptsize] {$d^0_{20}$} (BPcomplex-4-4);
\end{scope}
\end{tikzpicture}
$$
where $(Y_*,\partial_*)\index{zzf@$\partial_n$|dotfillboldidx}$ is the normalized bar resolution of the $A$-algebra $E$, introduced in~\cite{GS}; for each $s\ge 0$, the complex $(X_{*s},d^0_{*s})\index{da@$d^0_{rs}$|dotfillboldidx}$ is $(-1)^s$-times the normalized bar resolution of the $K$-algebra $A$, tensored on the left over $A$ with $E\ot_{\hs A}\wt{E}^{\ot_{\hs A} s}$, and on the right over $A$ with $E$; and for each $s\ge 0$, the map $\upsilon_s\index{zu@$\upsilon_s$|dotfillboldidx}$ is the canonical surjection.

\begin{proposition}\label{las filas son contractiles} Each one of the rows of the above diagram is contractible as an $(E,K)$-bi\-mo\-du\-le complex. A contracting homotopy $\sigma^0_{0s}\colon Y_s\to X_{0s}$ and $\sigma^0_{r+1,s}\colon X_{rs}\to X_{r+1,s}$\index{zsa@$\sigma^0_{rs}$|dotfillboldidx} (for $r\ge 0$), of the $s$-th row, is given by
\begin{align*}
&\sigma^0_{0s}(x_0\ot_{\hs A} \wt{\bx}_{1s}\ot_{\hs A} x_{s+1})\coloneqq \sum_l x_0\ot_{\hs A}\wt{\bx}_{1s}\xcdot a_{(l)} \ot \gamma(v_{(l)})\\
\shortintertext{and}
&\sigma^0_{r+1,s}(x_0\ot_{\hs A}\wt{\bx}_{1s}\ot \ov{\ba}_{1r} \ot x_{s+1}) \coloneqq (-1)^{r+s+1} \sum_l x_0\ot_{\hs A}\wt{\bx}_{1s}\ot \ov{\ba}_{1r} \ot \ov{a_{(l)}}\ot \gamma(v_{(l)}),
\end{align*}
where $\sum_l a_{(l)} \ot_k v_{(l)} = x_{s+1}$.
\end{proposition}

\begin{proof} To begin note that the morphism $\sigma^0_{0s}$ is well defined, since
$$
\tilde{\sigma}^0_{0s}(x_0\ot_{\hs A} \wt{\bx}_{1s}\ot_k a\xcdot x_{s+1}) = \sum_l x_0\ot_{\hs A} \wt{\bx}_{1s}\xcdot aa_{(l)}\ot \gamma(v_{(l)}) = \tilde{\sigma}^0_{0s}(x_0\ot_{\hs A} \wt{\bx}_{1s}\xcdot a \ot_k x_{s+1}).
$$
In order to finish the proof of the proposition we must check that
\begin{align}
&\upsilon_s\xcirc\sigma^0_{0s}=\ide_{Y_s},\label{eqhomot1}\\
&\sigma^0_{0s}\xcirc\upsilon_s + d^0_{1s}\xcirc\sigma^0_{1s}=\ide_{X_{0s}} \label{eqhomot2}
\shortintertext{and}
&\sigma^0_{rs}\xcirc d^0_{rs} + d^0_{r+1,s}\xcirc\sigma^0_{r+1,s}=\ide_{X_{rs}} \qquad \text{for $r>0$.}\label{eqhomot3}
\end{align}
A direct computation shows that
\begin{align*}
&\upsilon\bigl(\sigma^0(x_0\ot_{\hs A}\wt{\bx}_{1s}\ot_{\hs A} x_{s+1})\bigr) = \sum_l x_0\ot_{\hs A}\wt{\bx}_{1s} \ot_{\hs A} a_{(l)}\xcdot \gamma(v_{(l)}),\\
&\sigma^0\bigl(\upsilon(x_0\ot_{\hs A}\wt{\bx}_{1s}\ot x_{s+1})\bigr)=\sum_l x_0\ot_{\hs A}\wt{\bx}_{1s}\xcdot a_{(l)}\ot\gamma(v_{(l)}),\\
&\sigma^0\bigl(d^0(x_0\!\ot_{\hs A}\wt{\bx}_{1s}\!\ot\ov{\ba}_{1r}\!\ot x_{s+1})\bigr) = \sum_l (-1)^r x_0\ot_{\hs A} \wt{\bx}_{1s}\!\ot_{\hs A} b'\bigl(1_{\hs A}\ot  \ov{\ba}_{1r}\!\ot 1_{\hs A}\bigr)\ot_{\hs A} \ov{a_{(l)}}\ot \gamma(v_{(l)})
\shortintertext{and}
& d^0\bigl(\sigma^0(x_0\!\ot_{\hs A}\wt{\bx}_{1s}\!\ot\ov{\ba}_{1r}\!\ot x_{s+1})\bigr)\! =\! \sum_l (-1)^{r+1} x_0\!\ot_{\hs A} \wt{\bx}_{1s}\!\ot_{\hs A} b'\bigl(1_{\hs A}\ot \ov{\ba}_{1r}\! \ot \ov{a_{(l)}}\ot 1_{\hs A}\bigr) \ot_{\hs A} \gamma(v_{(l)}),
\end{align*}
where $b'_*$ is the boundary map of the normalized Hochschild resolution of the $K$-algebra $A$. So, equalities~\eqref{eqhomot1}, \eqref{eqhomot2} and~\eqref{eqhomot3} follow immediately from equality~\eqref{eq4}.
\end{proof}

\begin{remark}\label{previo} Using equality~\eqref{eq4} it is easy to see that if $x_{s+1}\!\in\! (K\ot_k V)\cap E$, then, for all $a\!\in\! A$,
\begin{align}
& \sigma^0_{0s}(x_0\ot_{\hs A} \wt{\bx}_{1s}\ot_{\hs A} a\xcdot x_{s+1}) = x_0\ot_{\hs A}\wt{\bx}_{1s}\xcdot a \ot x_{s+1}\label{primera cond}
\shortintertext{and}
&\sigma^0_{r+1,s}(x_0\ot_{\hs A}\wt{\bx}_{1s}\ot \ov{\ba}_{1r} \ot a\xcdot x_{s+1})=(-1)^{r+s+1}  x_0\ot_{\hs A}\wt{\bx}_{1s}\ot \ov{\ba}_{1r}\ot \ov{a} \ot x_{s+1}.\label{segunda cond}
\end{align}
\end{remark}

\begin{remark} The complex of $E$-bimodules
$$
\begin{tikzpicture}
\begin{scope}[yshift=0cm,xshift=0cm]
\matrix(BPcomplex) [matrix of math nodes,row sep=0em, text height=1.5ex, text
depth=0.25ex, column sep=2.5em, inner sep=0pt, minimum height=5mm,minimum width =6mm]
{E & Y_0 & Y_1 & Y_2 & Y_3 & Y_4 & Y_5 &\cdots\\};
\draw[<-] (BPcomplex-1-1) -- node[above=1pt,font=\scriptsize] {$-\wt{\mu}_E$} (BPcomplex-1-2);
\draw[<-] (BPcomplex-1-2) -- node[above=1pt,font=\scriptsize] {$-\partial_1$} (BPcomplex-1-3);
\draw[<-] (BPcomplex-1-3) -- node[above=1pt,font=\scriptsize] {$-\partial_2$} (BPcomplex-1-4);
\draw[<-] (BPcomplex-1-4) -- node[above=1pt,font=\scriptsize] {$-\partial_3$} (BPcomplex-1-5);
\draw[<-] (BPcomplex-1-5) -- node[above=1pt,font=\scriptsize] {$-\partial_4$} (BPcomplex-1-6);
\draw[<-] (BPcomplex-1-6) -- node[above=1pt,font=\scriptsize] {$-\partial_5$} (BPcomplex-1-7);
\draw[<-] (BPcomplex-1-7) -- node[above=1pt,font=\scriptsize] {$-\partial_6$} (BPcomplex-1-8);
\end{scope}
\end{tikzpicture}
$$
is contractible as a complex of $(E,A)$-bimodules. A chain contracting homotopy
$$
\sigma_0^{-1}\colon E\to Y_0\qquad\text{and}\qquad\sigma^{-1}_{s+1}\colon Y_s\to Y_{s+1}\index{zsb@$\sigma^{-1}_s$|dotfillboldidx}
\quad\text{for $s\ge 0$,}
$$
is given by $\sigma^{-1}_{s+1}(x_0\ot_{\hs A}\wt{\bx}_{1s} \ot_{\hs A} x_{s+1})\coloneqq  (-1)^sx_0\ot_{\hs A}\wt{\bx}_{1,s+1} \ot_{\hs A} 1_E$.
\end{remark}

\begin{notation}\label{notacion X supra u sub rs (R)} In the sequel we will use the following notations:

\begin{enumerate}[itemsep=0.7ex, topsep=1.0ex, label=(\arabic*)]

\item Let $r,s\ge 0$ and $0\le u\le r$. For each $k$-subalgebra $R$ of $A$, we let $L^u_{rs}(R)\index{lb@$L^u_{rs}(R)$|dotfillboldidx}$ denote the $(A,K)$-sub\-bimodule of $X_{rs}$ generated by the simple tensors of the form $1_E\ot_{\hs A}\wt{\gamma}_{\hs A}(\bv_{1s})\ot\ov{\ba}_{1r}\ot 1_E$, where $v_1,\dots,v_s\in V$, $a_1,\dots,a_r\in A$ and at least $u$ of the $a_j$'s are in $R$.

\item Let $r,s\ge 0$ and $0\le u\le r$. For each $k$-subalgebra $R$ of $A$, we let $\mathfrak{L}^u_{rs}(R)\index{lc@$\mathfrak{L}^u_{rs}(R)$|dotfillboldidx}$ denote the $(A,K)$-sub\-bimodule of $X'_{rs}$ generated by the simple tensors of the form $1_E\ot_{\hs A}\gamma_{\hs A}(\bv_{1s})\ot\ov{\ba}_{1r}\ot 1_E$, where $v_1,\dots,v_s\in V$, $a_1,\dots,a_r\in A$ and at least $u$ of the $a_j$'s are in $R$.

\item Let $r,s\ge 0$ and $0\le u\le r$. For each $k$-subalgebra $R$ of $A$ we set
$$
\qquad U^u_{rs}(R)\coloneqq  L^u_{rs}(R)\xcdot \gamma(V)\index{ub@$U^u_{rs}(R)$|dotfillboldidx},\quad W^u_{rs}(R)\coloneqq L^u_{rs}(R)\xcdot A \index{wa@$W^u_{rs}(R)$|dotfillboldidx}\quad\text{and}\quad\mathfrak{U}^u_{rs}(R)\coloneqq \mathfrak{L}^u_{rs}(R)\xcdot \gamma(V)\index{uc@$\mathfrak{U}^u_{rs}(R)$|dotfillboldidx}.
$$
\end{enumerate}

\end{notation}

\begin{remark}\label{L u {rs}(A)} Note that for $0\le u\le r$ and all $R\subseteq A$, we have $L^u_{rs}(A) = L^0_{rs}(R)$, $U^u_{rs}(A) = U^0_{rs}(R)$ and $W^u_{rs}(A) = W^0_{rs}(R)$. To abbreviate expressions we will write $L_{rs}\index{la@$L_{rs}$|dotfillboldidx}$, $U_{rs}\index{ua@$U_{rs}$|dotfillboldidx}$ and $W_{rs}\index{wa@$W_{rs}$|dotfillboldidx}$, instead of $L^0_{rs}(A)$, $U^0_{rs}(A)$ and $W^0_{rs}(A)$, respectively. Clearly $L_{rs}$, $U_{rs}$ and $W_{rs}$ are the image of the canonical maps
\begin{align*}
&\jmath_{\nu}(A)\ot_{\hs A} \wt{E}^{\ot_{\hs A} s} \ot \ov{A}^{\ot r}\ot \jmath_{\nu}(K)\longrightarrow E\ot_{\hs A} \wt{E}^{\ot_{\hs A} s} \ot \ov{A}^{\ot r}\ot E,\\
&\jmath_{\nu}(A)\ot_{\hs A} \wt{E}^{\ot_{\hs A} s} \ot \ov{A}^{\ot r}\ot \jmath_{\nu}(K)\gamma(V)\longrightarrow E\ot_{\hs A} \wt{E}^{\ot_{\hs A} s} \ot \ov{A}^{\ot r}\ot E
\shortintertext{and}
&\jmath_{\nu}(A)\ot_{\hs A} \wt{E}^{\ot_{\hs A} s} \ot \ov{A}^{\ot r}\ot \jmath_{\nu}(A)\longrightarrow E\ot_{\hs A} \wt{E}^{\ot_{\hs A} s} \ot \ov{A}^{\ot r}\ot E,
\end{align*}
respectively. By Remark~\ref{caso jmath_{nu}(K) subseteq jmath_{nu}(K)gamma(V)} if $1_E\in K\ot_k V$, then $L_{rs}\subseteq U_{rs}$.
\end{remark}

For $r\ge 0$ and $1\le l\le s$, we define $E$-bimodule maps $d^l_{rs}\colon X_{rs}\to X_{r+l-1,s-l}\index{db@$d^l_{rs}$|dotfillboldidx}$ recursively by:
$$
d^l(\bz)\coloneqq \begin{cases}
\phantom{-}\sigma^0\xcirc\partial\xcirc\upsilon(\bz) &\text{if $l=1$ and $r=0$,}\\
-\sigma^0\xcirc d^1\xcirc d^0(\bz) &\text{if $l=1$ and $r>0$,}\\
-\sum_{j=1}^{l-1}\sigma^0\xcirc d^{l-j}\xcirc d^j(\bz) &\text{if $1<l$ and $r=0$,}\\
-\sum_{j=0}^{l-1}\sigma^0\xcirc d^{l-j}\xcirc d^j(\bz) &\text{if $1<l$ and $r>0$,}
\end{cases}
$$
for $\bz\in E\xcdot L_{rs}$.

\begin{theorem}\label{res nuestra} The chain complex
$$
\begin{tikzpicture}
\begin{scope}[yshift=0cm,xshift=0cm]
\matrix(BPcomplex) [matrix of math nodes, row sep=0em, text height=1.5ex, text
depth=0.25ex, column sep=2.5em, inner sep=0pt, minimum height=5mm, minimum width =6mm]
{E & X_0 & X_1 & X_2 & X_3 & X_4 &\cdots,\\};
\draw[<-] (BPcomplex-1-1) -- node[above=1pt,font=\scriptsize] {$\ov{\mu}_E$} (BPcomplex-1-2);
\draw[<-] (BPcomplex-1-2) -- node[above=1pt,font=\scriptsize] {$d_1$} (BPcomplex-1-3);
\draw[<-] (BPcomplex-1-3) -- node[above=1pt,font=\scriptsize] {$d_2$} (BPcomplex-1-4);
\draw[<-] (BPcomplex-1-4) -- node[above=1pt,font=\scriptsize] {$d_3$} (BPcomplex-1-5);
\draw[<-] (BPcomplex-1-5) -- node[above=1pt,font=\scriptsize] {$d_4$} (BPcomplex-1-6);
\draw[<-] (BPcomplex-1-6) -- node[above=1pt,font=\scriptsize] {$d_5$} (BPcomplex-1-7);
\end{scope}
\end{tikzpicture}
$$
where
$$
X_n\coloneqq \bigoplus_{r+s=n} X_{rs}\index{xd@$X_n$|dotfillboldidx}\qquad\text{and}\qquad d_n(\bx)\coloneqq \begin{cases} \displaystyle{\sum_{l=1}^n d^l_{0n}(\bx)} &\text{if $\bx\in X_{0n}$,}\\ \displaystyle{\sum^{n-r}_{l=0} d^l_{r,n-r}(\bx)} &\text{if $\bx\in X_{r,n-r}$ with $r>0$.}\end{cases}\index{dc@$d_n$|dotfillboldidx}
$$
is a $\Upsilon$-relative projective resolution of $E$.
\end{theorem}

\begin{proof} This is an immediate consequence of~\cite{GG}*{Corollary~A2}.
\end{proof}

\begin{remark} In Theorem~\ref{res nuestra} and in the rest of this work we identify $X_{rs}$ with its image inside $X_{r+s}$.
\end{remark}

In order to carry out our computations we also need to give an explicit contracting homotopy of this resolution. For this we define $(E,K)$-bimodule maps
$$
\sigma^l_{l,s-l}\colon Y_s\longrightarrow X_{l,s-l}\qquad\text{and}\qquad\sigma^l_{r+l+1,s-l}\colon X_{rs}\longrightarrow X_{r+l+1,s-l}
\index{zsc@$\sigma^l_{rs}$|dotfillboldidx}
$$
recursively on $l$, by $\sigma^l_{r+l+1,s-l}\coloneqq  -\sum_{i=0}^{l-1}\sigma^0\xcirc d^{l-i}\xcirc\sigma^i$  for $0<l\le s$ and $r\ge -1$.

\begin{proposition}\label{cont nuestra} A contracting homotopy $\ov{\sigma}_0\colon E\to X_0$ and $\ov{\sigma}_{n+1}\colon X_n\to X_{n+1}$ for $n\ge 0$,\index{zsd@$\ov{\sigma}_n$|dotfillboldidx} of the~res\-olution introduced in Theorem~\ref{res nuestra}, is given by
\begin{align*}
& \ov{\sigma}_0(\bx)\coloneqq -\sigma_{00}^0\xcirc\sigma_0^{-1}(\bx)
\shortintertext{and}
& \ov{\sigma}_{n+1}(\bx)\coloneqq \begin{cases} \displaystyle{-\sum_{l=0}^{n+1}\sigma_{l,n-l+1}^l\xcirc\sigma_{n+1}^{-1}\xcirc \upsilon_n(\bx) +  \sum_{l=0}^n\sigma_{l+1,n-l}^l(\bx)} &\text{if $\bx\in X_{0n}$,}\\ \displaystyle{\phantom{-} \sum_{l=0}^{n-r}\sigma_{r+l+1,n-r-l}^l(\bx)} &\text{if $\bx\in X_{r,n-r}$ with $r>0$.}\end{cases}
\end{align*}
\end{proposition}

\begin{proof} This is a direct consequence of~\cite{GG}*{Corollary~A2}.
\end{proof}

\begin{notation} For a right $A$-submodule $E_1$ of $E$ and a left $A$-submodule $E_2$ of $E$, we let $E_1 \bar{\ot}_{\hs A} \wt{E}^{\ot_{\hs A} s} \bar{\ot}_{\hs A} E_2\index{fdc@$E_1 \bar{\ot}_{\hs A}\wt{E}^{\ot_{\hs A} s} \bar{\ot}_{\hs A} E_2$|dotfillboldidx}$ denote the image of the canonical map $E_1 \ot_{\hs A}\wt{E}^{\ot_{\hs A} s} \ot_{\hs A} E_2\longrightarrow E\ot_{\hs A}\wt{E}^{\ot_{\hs A} s} \ot_{\hs A} E$.
\end{notation}

\begin{remark}\label{cond de sigma} By its very definition,
\begin{equation}
\sigma^0(\jmath_{\nu}(A)\bar{\ot}_{\hs A}\wt{E}^{\ot_{\hs A} s}\bar{\ot}_{\hs A} E)\subseteq U_{0s} \quad\text{and}\quad \sigma^0(L_{rs}\xcdot E)\subseteq U_{r+1,s}\qquad\text{for all $r,s\ge 0$.}\label{ec5p}
\end{equation}
Since $\sigma^0$ is left $E$-linear, from this and the definition of $d^l$ it follows that,
\begin{equation}
d^l(E\xcdot L_{rs})\subseteq  E\xcdot U_{r+l-1,s-l}\qquad\text{for all $r\ge 0$ and $1\le l\le s$.}\label{ec5}
\end{equation}

\end{remark}

\begin{remark}\label{ima of sigma^l} By inclusions~\eqref{ec5p} and the definition of $\sigma^l$, we have
$$
\sigma^l(Y_s)\subseteq E\xcdot U_{l,s-l}\quad\text{and}\quad \sigma^l(X_{rs})\subseteq E\xcdot U_{r+l+1,s-l}\qquad\text{for all $r\ge 0$ and $0\le l\le s$.}
$$
\end{remark}

\smallskip

Since $K$ is stable under $\chi$ we have $\gamma(V)\subseteq K\ot_k V$. So, by equality~\eqref{segunda cond},
\begin{equation}
\sigma^0(E\xcdot U_{rs}) = 0\qquad\text{for all $r,s\ge 0$.}\label{ec2}
\end{equation}
On the other hand, using that $1_E\in K\ot_k V$ and equalities~\eqref{primera cond} and~\eqref{segunda cond}, we get that
\begin{align}
&\sigma^0\bigl(\jmath_{\nu}(A)\bar{\ot}_{\hs A}\wt{E}^{\ot_{\hs A} s}\bar{\ot}_{\hs A} \jmath_{\nu}(A)\bigr)\subseteq L_{0s}&&\text{for all $s\ge 0$,}\label{ec3}
\shortintertext{and}
&\sigma^0(W_{rs}) \subseteq L_{r+1,s}&&\text{for all $r,s\ge 0$.}\label{ec4}
\end{align}

\begin{remark}\label{propiedades de sigma^0} By Remark~\ref{ima of sigma^l} and equality~\eqref{ec2}, we have $\sigma^0\xcirc\sigma^0=0$.
\end{remark}

\begin{proposition}\label{prim prop of ov sigma} The homotopy $\ov{\sigma}$ found in Proposition~\ref{cont nuestra} satisfies
$$
\ov{\sigma}_{n+1}(\bx) = -\sigma_{0,n+1}^0\xcirc\sigma_{n+1}^{-1}\xcirc\upsilon_n(\bx) + \sum_{l=0}^n \sigma_{l+1,n-l}^l(\bx)\qquad\text{for all $\bx\in X_{0n}$.}
$$
\end{proposition}

\begin{proof} By the definitions of $\ov{\sigma}$, $\upsilon$ and $\sigma^{-1}$, it suffices to prove that $\sigma^l\bigl(E\bar{\ot}_{\hs A}\wt{E}^{\ot_{\hs A} {n+1}}\bar{\ot}_{\hs A} \jmath_{\nu}(A)\bigr) = 0$ for $l\ge 1$. Suppose this is false, and let $l\ge 1$ be the least upper index for which the above equality is wrong. Hence
$$
\sigma^l(\bx) = -\sum_{i=0}^{l-1}\sigma^0\xcirc d^{l-i}\xcirc\sigma^i(\bx) = -\sigma^0\xcirc d^l\xcirc\sigma^0(\bx)\qquad\text{for each $\bx\in E\bar{\ot}_{\hs A}\wt{E}^{\ot_{\hs A} {n+1}}\bar{\ot}_{\hs A}\jmath_{\nu}(A)$.}
$$
But, by the inclusions in~\eqref{ec5} and~\eqref{ec3} we have $\sigma^0\xcirc d^l\xcirc\sigma^0(\bx)\in \sigma^0\xcirc d^l(E\xcdot L_{0n})\subseteq \sigma^0\bigl(E\xcdot U_{l+1,n-l})$, which, by equality~\eqref{ec2} is zero. This contradiction shows that the proposition is true.
\end{proof}

\begin{corollary}\label{ima of ov sigma caso p} For all $n\ge 0$ and $0\le r\le n$, we have
$$
\ov{\sigma}(X_{r,n-r})\subseteq \begin{cases} E\xcdot L_{0,n+1}\oplus \bigoplus_{l=0}^n E\xcdot U_{l+1,n-l}& \text{if $r=0$,}\\\bigoplus_{l=0}^{n-r} E\xcdot U_{r+l+1,n-r-l} & \text{if $r>0$.}
\end{cases}
$$
\end{corollary}

\begin{proof} By inclusion~\eqref{ec3} and the definitions of $\upsilon$ and $\sigma^{-1}$, we have
$$
\sigma_{0,n+1}^0\xcirc\sigma_{n+1}^{-1}\xcirc\upsilon_n(X_{0n})\subseteq \sigma^0\bigl(E\bar{\ot}_{\hs A}\wt{E}^{\ot_{\hs A} {n+1}}\bar{\ot}_{\hs A} \jmath_{\nu}(A)\bigr) \subseteq E\xcdot L_{0,n+1}.
$$
By the definition of $\ov{\sigma}$, Proposition~\ref{prim prop of ov sigma} and Remark~\ref{ima of sigma^l} this finishes the proof.
\end{proof}

\begin{proposition}\label{sigma barra al cuadrado es cero} The contracting homotopy $\ov{\sigma}$ satisfies $\ov{\sigma}\xcirc\ov{\sigma} = 0$.
\end{proposition}

\begin{proof} By the definition of $\ov{\sigma}$ and Proposition~\ref{prim prop of ov sigma} it will be sufficient to see that
$$
\sigma^0 \xcirc \sigma^{-1}\xcirc \upsilon\xcirc \sigma^0\xcirc \sigma^{-1}\xcirc \upsilon = 0\qquad\text{and}\qquad \sigma^l \xcirc \sigma^{l'} = 0\quad\text{for all $l,l'\ge 0$.}
$$
The first equality is true because $\upsilon\xcirc \sigma^0 = \ide$ and $\sigma^{-1}\xcirc \sigma^{-1} = 0$. Since $\sigma^0\xcirc\sigma^0 = 0$, in order to prove the second one it suffices to show that there exists a map $\varsigma^l$ such that $\sigma^l = \sigma^0\xcirc \varsigma^l \xcirc\sigma^0$ for all $l\ge 1$. But this follows immediately from the definition of $\sigma^l$ and an easy inductive argument.
\end{proof}

\begin{proposition}\label{sigma l=0} For all $r\ge 0$ and $1\le l \le s$, we have $\sigma^l\bigl(E\xcdot W_{rs}) = 0$.
\end{proposition}

\begin{proof} We proceed by induction on $l$. Suppose that this assertion is true for all $l<l_0$. By the definition of $\sigma^{l_0}$ and the inductive hypothesis,
$$
\sigma^{l_0}(\bx)=-\sum_{i=0}^{l_0-1}\sigma^0\xcirc d^{l_0-i}\xcirc\sigma^i(\bx) = \sigma^0\xcirc d^{l_0}\xcirc\sigma^0(\bx)\qquad\text{for each $\bx\in E\xcdot W_{rs}$.}
$$
But, by the inclusions in~\eqref{ec5} and~\eqref{ec4}
$$
\sigma^0\xcirc d^{l_0}\xcirc\sigma^0(\bx)\in \sigma^0\xcirc d^{l_0}(E\xcdot L_{r+1,s})\subseteq \sigma^0(E\xcdot U_{r+l_0,s-l_0}),
$$
which, by equality~\eqref{ec2}, is zero.
\end{proof}

\begin{remark}\label{ov sigma reducido} Clearly $\sigma^{-1}\xcirc\upsilon(E\xcdot W_{0n})=0$. So, by the definition of $\ov{\sigma}$ and Proposition~\ref{sigma l=0},
$$
\ov{\sigma}(\bx) = \sigma^0(\bx) \qquad\text{for all $0\le r\le n$ and all $\bx\in E\xcdot W_{r,n-r}$}.
$$
Consequently, by the inclusion in~\eqref{ec4},
\begin{equation}
\overline{\sigma}(W_{r,n-r})\subseteq L_{r+1,n-r}\qquad\text{for all $0\le r\le n$.}\label{ec6}
\end{equation}
\end{remark}

\begin{definition} For $r\ge 0$ and $s>0$, we define $E$-bimodule maps $\mu'_i\colon X'_{rs}\to X'_{r,s-1}$\index{zmc@$\mu'_i$|dotfillboldidx} for $0\le i\le s$, by
$$
\mu'_i(\bx_{0s}\ot\ov{\ba}_{1r}\ot 1_E) \coloneqq  \begin{cases} x_0x_1\ot_{\hs A} \bx_{2s}\ot\ov{\ba}_{1r}\ot 1_E &\text{if $i=0$,}\\
\bx_{0,i-1}\ot_{\hs A} x_ix_{i+1} \ot_{\hs A} \bx_{i+2,s}\ot \ov{\ba}_{1r}\ot 1_E & \text{if $0<i<s$,}\\
\sum_l \bx_{0,s-2}\ot_{\hs A}x_{s-1}\jmath_{\nu}(a) \ot \ov{\ba}_{1r}^{(l)}\ot \gamma(v^{(l)}) & \text{if $i=s$,}
\end{cases}
$$
where $x_0,\dots,x_{s-1}\in E$, $a_1,\dots,a_r\in A$, $x_s\coloneqq \jmath_{\nu}(a)\gamma(v)$ with $a\in A$ and $v\in V$, and
$$
\sum_l \ov{\ba}_{1r}^{(l)}\ot_k v^{(l)} \coloneqq  \ov{\chi}\bigl(v\ot_k \ov{\ba}_{1r}\bigr).
$$
\end{definition}

\begin{theorem}\label{formula para d^1} The map $d^1\colon X_{rs}\to X_{r,s-1}$ is induced by $\sum_{i=0}^s (-1)^i \mu'_i$.
\end{theorem}

\begin{proof} Assume $r=0$ and let $x_0,\dots,x_{s-1}\in E$, $a_s\in A$, $v_s\in V$ and $x_s\coloneqq \jmath_{\nu}(a_s)\gamma(v_s)$.
Since $1_E\in K\ot_k V$ and $\gamma(v_s)\subseteq K\ot_k V$, from equality~\eqref{primera cond} it follows that
\begin{align*}
d^1(x_0\ot_{\hs A}\wt{\bx}_{1s}\ot 1_E) & = \sigma^0\xcirc\partial\xcirc \upsilon(x_0\ot_{\hs A}\wt{\bx}_{1s}\ot 1_E)\\
%
%
& = x_0x_1 \ot_{\hs A} \wt{\bx}_{2s}\ot 1_E\\
& + \sum_{i=1}^{s-1} (-1)^i x_0 \ot_{\hs A} \wt{\bx}_{1,i-1}\ot_{\hs A} \stackon[-8pt]{$x_ix_{i+1}$}{\vstretch{1.5}{\hstretch{2.6} {\widetilde{\phantom{\;\;\;\;}}}}} \ot_{\hs A} \wt{\bx}_{i+2,s} \ot 1_E\\
& + (-1)^s x_0\ot_{\hs A}\wt{\bx}_{1,s-2} \ot_{\hs A} \stackon[-8pt]{$x_{s-1}\jmath_{\nu}(a_s)$}{\vstretch{1.5}{\hstretch{3.9} {\widetilde{\phantom{\;\;\;\;}}}}}\ot \gamma(v_s).
\end{align*}
Assume now that $r>0$ and that the theorem holds for $r-1$ and $s$. Let $T\in E\ot_{\hs A} \wt{E}^{\ot_{\hs A} s} \ot \ov{ A}^{\ot {r-1}}\ot E$ be defined by
\begin{align*}
T &\coloneqq  (-1)^s x_0\ot_{\hs A}\wt{\bx}_{1,s-1}\ot_{\hs A} \stackon[-8pt]{$x_s\jmath_{\nu}(a_{s+1})$}{\vstretch{1.5}{\hstretch{3.9} {\widetilde{\phantom{\;\;\;\;}}}}} \ot\ov{\ba}_{s+2,s+r}\ot 1_E\\
& + \sum_{i=1}^{r-1} (-1)^{s+i} x_0\ot_{\hs A} \wt{\bx}_{1s} \ot \ov{\ba}_{s+1,s+i-1}\ot \ov{a_{s+i}a_{s+i+1}}\ot\ov{\ba}_{i+2,r}\ot 1_E,
\end{align*}
where $x_0,\dots,x_{s-1}\in E$, $a_s,\dots,a_{s+r}\in A$, $v_s\in V$ and $x_s\coloneqq \jmath_{\nu}(a_s)\gamma(v_s)$. By the definition of $d^1$,
$$
\sigma^0\xcirc d^1(T) = \begin{cases} \phantom{-} \sigma^0\xcirc\sigma^0\xcirc \partial\xcirc \upsilon(T) &\text{if $r=1$,}\\ -\sigma^0\xcirc \sigma^0\xcirc d^1\xcirc d^0(T) &\text{if $r>1$,}\end{cases}
$$
which, by Remark~\ref{propiedades de sigma^0} equals zero. So, by the definition of $d^1$, we have
\begin{equation}\label{qqq}
d^1\bigl(x_0\ot_{\hs A} \wt{\bx}_{1s}\ot\ov{\ba}_{s+1,s+r}\ot 1_E\bigr) = (-1)^{s+r+1} \sigma^0\xcirc d^1\bigl(x_0\ot_{\hs A} \wt{\bx}_{1s}\ot\ov{\ba}_{s+1,s+r-1}\ot \jmath_{\nu}(a_{s+r})\bigr).
\end{equation}
Now by the inductive hypothesis
\begin{align*}
d^1\bigl(x_0\ot_{\hs A} \wt{\bx}_{1s}&\ot\ov{\ba}_{s+1,s+r-1}\ot \jmath_{\nu}(a_{s+r})\bigr) = x_0x_1\ot_{\hs A} \wt{\bx}_{2s}\ot\ov{\ba}_{s+1,s+r-1}\ot \jmath_{\nu}(a_{s+r})\\
& + \sum_{i=1}^{s-1}(-1)^i x_0\ot_{\hs A} \wt{\bx}_{0,i-1}\ot_{\hs A} \stackon[-8pt]{$x_ix_{i+1}$}{\vstretch{1.5}{\hstretch{2.6} {\widetilde{\phantom{\;\;\;\;}}}}} \ot_{\hs A} \wt{\bx}_{i+2,s}\ot \ov{\ba}_{s+1,s+r-1}\ot \jmath_{\nu}(a_{s+r})\\
& + \sum_l x_0\ot_{\hs A} \wt{\bx}_{1,s-2}\ot_{\hs A} \stackon[-8pt]{$x_{s-1}\jmath_{\nu}(a_s)$}{\vstretch{1.5}{\hstretch{3.9} {\widetilde{\phantom{\;\;\;\;}}}}} \ot \ov{\ba}_{s+1,s+r-1}^{(l)}\ot \gamma(v_s^{(l)})\jmath_{\nu}(a_{s+r}),
\end{align*}
where $\sum_l \ov{\ba}_{s+1,s+r-1}^{(l)}\ot_k v_s^{(l)} \coloneqq  \ov{\chi}\bigl(v_s\ot_k \ov{\ba}_{s+1,s+r-1}\bigr)$. The result for $d_{rs}^1$ follows combining this with~\eqref{qqq}, using equalities~\eqref{eq4} and~\eqref{segunda cond}, and that $\jmath_{\nu}(a_{s+r}) = a_{s+r}\xcdot 1_E$, $1_E\in K\ot_kV$ and $\gamma(V)\subseteq K\ot_kV$.
\end{proof}

In the following theorem we assume that $V$ is a non unitary associative algebra and a non~co\-unitary coassociative coalgebra such that
$(vw)^{(1)}\ot_k (vw)^{(2)} = v^{(1)}w^{(1)}\ot_k v^{(2)}w^{(2)}$ for all $v,w\in V$, where we are using the Sweedler notation. We assume also that there exist maps $\rho\colon V\ot_k A\to A$ and $f\colon V\ot_k V\to A$ such that
\begin{equation}\label{eqq}
\gamma(v)\jmath_{\nu}(a) = v^{(1)}\xcdot a\ot_k v^{(2)}\quad\text{and}\quad \gamma(v)\gamma(w) = f(v^{(1)}\ot_k w^{(1)})\ot_k v^{(2)}w^{(2)},
\end{equation}
where $v\xcdot a\coloneqq \rho(v\ot_k a)$. Moreover, given $v\in V$ and $a_1,\dots,a_r\in A$, we~set
$$
v\xcdot {\ov{\ba}}_{1r}\coloneqq  \ov{v^{(1)}\xcdot a_1}\ot\cdots\ot \ov{v^{(r)}\xcdot a_r}\in \ov{A}^{\ot r}\index{vk@$v\cdot {\ov{\ba}}_{1r}$|dotfillboldidx}.
$$
These conditions are satisfied for the crossed products of algebras by weak bialgebras. So Theorems~\ref{formula para wh{d}^2} and~\ref{formula para wh{d}_2} (which are immediate consequences of the following result) apply in this context. We will use these results in~\cite{GGV1}.

\begin{theorem}\label{formula para d^2} The map $d^2\colon X_{rs}\to X_{r+1,s-2}$ is given by
$$
d^2(\bz_{rs}) = (-1)^{s+1} 1_E\ot_{\hs A}\wt{\gamma}_{\hs A}(\bv_{1,s-2})\ot \mathfrak{T}(v^{(1)}_{s-1},v^{(1)}_s,\ov{\ba}_{1r}) \ot \gamma(v_{s-1}^{(2)}v_s^{(2)}),
$$
where $\bz_{rs}\coloneqq 1_E\ot_{\hs A} \wt{\gamma}_{\hs A}(\bv_{1s})\ot \ov{\ba}_{1r}\ot 1_E$, with $v_1,\dots,v_s\in V$ and $a_1,\dots,a_r\in A$, and
$$
\mathfrak{T}(v_{s-1},v_s,\ov{\ba}_{1r})\coloneqq \sum_{i=0}^r (-1)^i  v_{s-1}^{(1)}\xcdot (v_s^{(1)}\xcdot \ov{\ba}_{1i})\ot f(v_{s-1}^{(2)}\ot_k v_s^{(2)}) \ot v_{s-1}^{(3)}v_s^{(3)}\xcdot \ov{\ba}_{i+1,r}.
$$
\end{theorem}

\begin{proof} We proceed by induction on $r$. Assume $r=0$. By definition $d^2(\bz_{0s}) = - \sigma^0\xcirc d^1\xcirc d^1(\bz_{0s})$. Since $1_E\in K\ot_k V$ and $\gamma(V)\subseteq K\ot_k V$, from Theorem~\ref{formula para d^1} and equality~\eqref{segunda cond} it follows that
$$
d^2(\bz_{0s}) = \sigma^0\bigl(1_E\ot_{\hs A} \wt{\gamma}_{\hs A}(\bv_{1,s-2})\ot \gamma(v_{s-1})\gamma(v_s)\bigr).
$$
Consequently, by equality~\eqref{segunda cond} and the second equality in~\eqref{eqq},
$$
d^2(\bz_{0s}) = (-1)^{s+1} 1_E\ot_{\hs A} \wt{\gamma}_{\hs A}(\bv_{1,s-2})\ot f(v_{s-1}^{(1)}\ot_k v_s^{(1)}) \ot \gamma\bigl(v_{s-1}^{(2)}v_s^{(2)}\bigr),
$$
as desired. Suppose now that the result is true for $r$. By definition
\begin{equation}\label{ec8}
d^2(\bz_{r+1,s}) = - \sigma^0\xcirc (d^2\xcirc d^0+d^1\xcirc d^1)(\bz_{r+1,s}).
\end{equation}
A similar computation as above (using the inductive hypothesis) shows that
\begin{align*}
& \sigma^0\xcirc d^2\xcirc d^0 (\bz_{r+1,s}) = (-1)^r \sigma^0\bigl(1_E\ot_{\hs A}\wt{\gamma}_{\hs A}(\bv_{1,s-2})\ot \mathfrak{T}(v^{(1)}_{s-1},v^{(2)}_s,\ov{\ba}_{1r}) \ot \gamma(v_{s-1}^{(2)}v_s^{(2)})\jmath_{\nu}(a_{r+1})\bigr)\\
\shortintertext{and}
&\sigma^0\xcirc d^1\xcirc d^1 (\bz_{r+1,s}) = -\sigma^0\bigl(1_E\ot_{\hs A} \wt{\gamma}_{\hs A}(\bv_{1,s-2})\ot v_{s-1}^{(1)}\xcdot (v_s^{(1)}\xcdot \ov{\ba}_{1,r+1})\ot \gamma(v_{s-1}^{(1)})\gamma(v_s^{(1)})\bigr).
\end{align*}
The formula for $d^2(\bz_{r+1,s})$ follows now from equalities~\eqref{segunda cond},~\eqref{eqq} and~\eqref{ec8}.
\end{proof}

\begin{proposition}\label{imagen de flat dag} Let $R$ be a $k$-subalgebra of $A$. If $R$ is stable under $\chi$, then
$$
\quad \mu'_i(\mathfrak{L}^u_{rs}(R))\subseteq \begin{cases} E\xcdot \mathfrak{L}^u_{r,s-1}(R) &\text{if $i=0$,}\\  \mathfrak{L}^u_{r,s-1}(R) &\text{if $0<i<s$,}\\ \mathfrak{U}^u_{r,s-1}(R) &\text{if $i=s$.}\end{cases}
$$
\end{proposition}

\begin{proof} For $i=0$ this is trivial, while for $0<i<s$ it follows using equality~\eqref{eq12} repeatedly. Finally for $i=s$, it follows from the very definition of $\mu'_s$ using that $R$ is stable under $\chi$.
\end{proof}

\begin{theorem}\label{formula para d^1b} Let $R$ be a $k$-subalgebra of $A$. If $R$ is stable under $\chi$ and $\mathcal{F}(V\ot_k V) \subseteq R\ot_k V$, then~the following assertions hold:

\begin{enumerate}[itemsep=0.7ex, topsep=1.0ex, label=\emph{(\arabic*)}]

\item $d^0(L^u_{rs}(R))\subseteq L^{u-1}_{r-1,s}(R) \xcdot R + W^u_{r-1,s}(R)$, for each $r\ge 1$, $s\ge 0$ and $0<u\le r$.

\item $d^1(L^u_{rs}(R))\subseteq E\xcdot L^u_{r,s-1}(R) + U^u_{r,s-1}(R)$, for each $r\ge 0$, $s\ge 1$ and $0\le u\le r$.

\item $d^l(L^u_{rs}(R))\subseteq U^{u+l-1}_{r+l-1,s-l}(R)$, for each $r\ge 0$, $s\ge 2$, $0\le u\le r$ and $2\le l\le s$.
\end{enumerate}

\end{theorem}

\begin{proof} Item~(1) is trivial and item~(2) is an immediate consequence of Theorem~\ref{formula para d^1} and Proposition~\ref{imagen de flat dag}. It remains to prove item~(3). Assume that $s\ge l>1$, $r\ge 0$ and that the result is true for every $d^j_{r's'}$ with $j\le \min(l-1,s')$, or with $j=l\le s'$ and $r'<r$, or with $j=l\le s'<s$ and $r'=r$. We claim that
\begin{equation*}
\sum_{j=1}^{l-1}\sigma^0\xcirc d^{l-j}\xcirc d^j\bigl(L^u_{rs}(R)\bigr)\subseteq U^{u+l-1}_{r+l-1,s-l}(R).
\end{equation*}
By item~(2) and the inductive hypothesis
\begin{equation*}
\sum_{j=1}^{l-1}\sigma^0\xcirc d^{l-j}\xcirc d^j\bigl(L^u_{rs}(R)\bigr) \subseteq \sigma^0\xcirc d^{l-1}\bigl(E\xcdot L^u_{r,s-1}(R)\bigr) +\sum_{j=1}^{l-1} \sigma^0\xcirc d^{l-j}\bigl(U^{u+j-1}_{r+j-1,s-j}(R)\bigr).
\end{equation*}
So, by the inclusion in~\eqref{ec5} and equality~\eqref{ec2},
\begin{equation}
\sum_{j=1}^{l-1}\sigma^0\xcirc d^{l-j}\xcirc d^j\bigl(L^u_{rs}(R)\bigr) \subseteq  \sum_{j=1}^{l-1}\sigma^0\xcirc d^{l-j}\bigl(U^{u+j-1}_{r+j-1,s-j}(R)\bigr).\label{etiqueta}
\end{equation}
Since, by item~(2),
$$
d^1\bigl(U^{u+l-2}_{r+l-2,s-l+1}(R)\bigr)\subseteq E\xcdot U^{u+l-2}_{r+l-2, s-l}(R) + L^{u+l-2}_{r+l-2,s-l}(R)\xcdot \gamma(V)\gamma(V),
$$
and, by the inductive hypothesis,
$$
d^{l-j}\bigl(U^{u+j-1}_{r+j-1,s-j}(R)\bigr)\subseteq L^{u+l-2}_{r+l-2,s-l}(R)\xcdot \gamma(V)\gamma(V)\quad\text{for $l-j>1$},
$$
from the inclusion~\eqref{etiqueta} we obtain that
$$
\sum_{j=1}^{l-1}\sigma^0\xcirc d^{l-j}\xcirc d^j\bigl(L^u_{rs}(R)\bigr)\subseteq \sigma^0\bigl(E\xcdot U^{u+l-2}_{r+l-2, s-l}(R)\bigr) + \sigma^0\bigl(L^{u+l-2}_{r+l-2,s-l}(R)\xcdot \gamma(V)\gamma(V)\bigr).
$$
Thus, the claim follows using equality~\eqref{ec2} and the fact that, by the second equality in Theorem~\ref{prodcruz1}(9), Proposition~\ref{R tensor V esta incluido} and equality~\eqref{segunda cond},
\begin{equation*}
\sigma^0(L^{u+l-2}_{r+l-2,s-l}(R)\xcdot \gamma(V)\gamma(V)) \subseteq \sigma^0(L^{u+l-2}_{r+l-2,s-l}(R)\xcdot \jmath_{\nu}(R)\gamma(V)) \subseteq U^{u+l-1}_{r+l-1,s-l}(R).
\end{equation*}
Consequently, by the very definition of $d^l$ we have
\begin{equation*}
d^l(L^u_{rs}(R)) \subseteq \begin{cases} U^{u+l-1}_{r+l-1,s-l}(R) & \text{if $r=0$}\\ \sigma^0\xcirc d^l\xcirc d^0(L^u_{rs}(R)) + U^{u+l-1}_{r+l-1,s-l}(R) &\text{if $r>0$.} \end{cases}
\end{equation*}
Hence, we are reduced to prove that $\sigma^0\xcirc d^l\xcirc d^0(L^u_{rs}(R)) \subseteq U^{u+l-1}_{r+l-1,s-l}(R)$ for each $r>0$. But, by item~(1) and the inductive hypothesis,
\begin{align*}
\sigma^0\xcirc d^l\xcirc d^0(L^u_{rs}(R))&\subseteq\sigma^0\Bigl(U^{u+l-2}_{r+l-2,s-l}(R)\xcdot R+ U^{u+l-1}_ {r+l-2,s-l}(R)\xcdot A\Bigr)\\
& \subseteq \sigma^0 \Bigl(L^{u+l-2}_{r+l-2,s-l}(R)\xcdot \gamma(V)\jmath_{\nu}(R) + L^{u+l-1}_{r+l-2,s-l}(R)\xcdot E\Bigr),
\end{align*}
and therefore, by the definition of $\sigma^0$, Lemma~\ref{gamma(V)iota(K) C= iota(K)gamma(V)}, the fact that $\gamma(V)\subseteq K\ot_k V$ and equality~\eqref{segunda cond}, we have $\sigma^0\xcirc d^l\xcirc d^0(L^u_{rs}(R))\subseteq U^{u+l-1}_{r+l-1,s-l}(R)$, as we want.
\end{proof}

\begin{corollary}\label{caso complej doble} If $\mathcal{F}$ takes its values in $K\ot_k V$, then $(X_*,d_*)$ is the total complex of $(X_{**},d^0_{**},d^1_{**})$.
\end{corollary}

\subsection{Comparison with the normalized bar resolution}\label{comparison maps} Let $(E\ot\ov{E}^{\ot *} \ot E, b'_*)\index{bc@$b'_*$|dotfillboldidx}$ be the normalized bar resolution of the $K$-algebra $E$. The complex
$$
\begin{tikzpicture}
\begin{scope}[yshift=0cm,xshift=0cm]
\matrix(BPcomplex) [matrix of math nodes,row sep=0em, column sep=2.5em, text height=1.5ex, text
depth=0.25ex , inner sep=3pt, minimum height=5mm, minimum width =6mm]
{E & E\ot E & E\ot\ov{E}\ot E & E\ot\ov{E}^{\ot 2}\ot E &\cdots,\\};
\draw[<-] (BPcomplex-1-1) -- node[above=1pt,font=\scriptsize] {$\ov{\mu}_E$} (BPcomplex-1-2);
\draw[<-] (BPcomplex-1-2) -- node[above=1pt,font=\scriptsize] {$b'_1$} (BPcomplex-1-3);
\draw[<-] (BPcomplex-1-3) -- node[above=1pt,font=\scriptsize] {$b'_2$} (BPcomplex-1-4);
\draw[<-] (BPcomplex-1-4) -- node[above=1pt,font=\scriptsize] {$b'_3$} (BPcomplex-1-5);
\end{scope}
\end{tikzpicture}
$$
is contractible as a complex of $(E,K)$-bimodules, with contracting homotopy
$$
\xi_0\colon E\longrightarrow E\ot E\qquad\text{and}\qquad \xi_{n+1}\colon E\ot \ov{E}^{\ot n}\ot E\longrightarrow E\ot\ov{E}^{\ot {n+1}}\ot E\index{zo@$\xi_n$|dotfillboldidx}\quad\text{for $n\ge 0$,}
$$
given by $\xi_{n+1}(\bx) \coloneqq (-1)^{n+1}\bx\ot 1_E$. Let
$$
\phi_*\colon (X_*,d_*)\longrightarrow (E\ot\ov{E}^{\ot *}\ot E,b'_*)\index{zv@$\phi_n$|dotfillboldidx}\qquad\text{and} \qquad\psi_*\colon (E\ot\ov{E}^{\ot *}\ot E,b'_*)\longrightarrow (X_*,d_*)\index{zy@$\psi_n$|dotfillboldidx}
$$
be the morphisms of $E$-bimodule chain complexes, recursively defined by
\begin{align*}
&\phi_0 \coloneqq \ide,\qquad\psi_0 \coloneqq \ide,\\
&\phi_n(\bx\ot 1_E) \coloneqq \xi_n\xcirc\phi_{n-1}\xcirc d_n(\bx\ot 1_E)\quad\text{for $n>0$}
\shortintertext{and}
&\psi_n(\byy\ot 1_E) \coloneqq \ov{\sigma}_n\xcirc\psi_{n-1}\xcirc b'(\byy\ot 1_E)\quad\text{for $n>0$.}
\end{align*}

\begin{notation} Given a right $K$-submodule $E_1$ of $E$, we let $E\ot\ov{E}^{\ot n}\ovb{\ot} E_1\index{fdd@$E\ot\ov{E}^{\ot n}\ovb{\ot} E_1$|dotfillboldidx}$ denote the image of the canon\-i\-cal map $E\ot\ov{E}^{\ot n}\ot E_1\longrightarrow E\ot\ov{E}^{\ot n}\ot E$.
\end{notation}

\begin{proposition}\label{homotopia} $\psi_*\xcirc\phi_* =\ide_*$ and $\phi_*\xcirc \psi_*$ is homotopically equivalent to the identity map. A homotopy $\omega_{*+1}\colon \phi_*\xcirc\psi_* \to \ide_*$ is the one degree map, recursively defined by
\begin{equation}
\omega_1\coloneqq  0\quad\text{and}\quad\omega_{n+1}(\byy\ot 1_E)\coloneqq \xi_{n+1} \xcirc(\phi_n \xcirc\psi_n -\ide -\omega_n\xcirc b'_n)(\byy\ot 1_E)\quad\text{for $n\ge 0$.}\label{aaa}
\end{equation}

\end{proposition}

\begin{proof} We prove both assertions by induction. Clearly $\psi_0\xcirc\phi_0 =\ide$. Assume that $\psi_n\xcirc\phi_n =\ide$. Since
$$
\phi_{n+1}(E\xcdot L_{n+1-s,s}) = \xi_{n+1}\xcirc\phi_n\xcirc d_{n+1}(E\xcdot L_{n+1-s,s}) \subseteq E\ot\ov{E}^{\ot {n+1}} \ovb{\ot}\jmath_{\nu}(K);
$$
on $E\xcdot L_{n+1-s,s}$, we have
\begin{align*}
\psi_{n+1}\xcirc\phi_{n+1} &=\ov{\sigma}_{n+1}\xcirc\psi_n\xcirc b'_{n+1}\xcirc\xi_{n+1}\xcirc\phi_n\xcirc d_{n+1}\\
& =\ov{\sigma}_{n+1}\xcirc\psi_n\xcirc\phi_n\xcirc d_{n+1} -\ov{\sigma}_{n+1}\xcirc \psi_n\xcirc\xi_n\xcirc b'_n\xcirc\phi_n\xcirc d_{n+1}\\
& =\ov{\sigma}_{n+1}\xcirc d_{n+1}\\
& =\ide_{n+1} - d_{n+2}\xcirc\ov{\sigma}_{n+2}.
\end{align*}
So, to conclude that $\psi_{n+1}\xcirc\phi_{n+1} =\ide$ it suffices to check that $\ov{\sigma}_{n+2}(E\xcdot L_{n+1-s,s}) = 0$. But by the definition of $\ov{\sigma}_{n+2}$ and the arguments in the proof of Proposition~\ref{sigma barra al cuadrado es cero}, for this it will be sufficient to prove that $\sigma^{-1}\xcirc\upsilon(E\xcdot L_{0,n+1}) = 0$ and $\sigma^0(E\xcdot L_{n+1-s,s}) = 0$. The first equality is clear and the second one is a particular case of equality~\eqref{ec2}.

\smallskip

Next we prove that $\omega_{*+1}\colon \phi_*\xcirc\psi_* \to \ide_*$ is an homotopy. Clearly $\phi_0\xcirc\psi_0 - \ide = 0 = b'_1\xcirc w_1$. Let
$
U_n \coloneqq \phi_n\xcirc\psi_n-\ide$ and $T_n \coloneqq  U_n -\omega_n\xcirc b'_n.
$
Assuming that $b'_n\xcirc\omega_n +\omega_{n-1}\xcirc b'_{n-1} = U_{n-1}$, we get that, on $E\ot\ov{E}^{\ot n}\ovb{\ot}\jmath_{\nu}(K)$,
\begin{align*}
b'_{n+1}\xcirc\omega_{n+1}+\omega_n\xcirc b'_n &=b'_{n+1}\xcirc\xi_{n+1}\xcirc T_n+\omega_n\xcirc b'_n && \text{by equality~\eqref{aaa}}\\
&= T_n -\xi_n\xcirc b'_n\xcirc T_n +\omega_n\xcirc b'_n\\
& = U_n -\xi_n\xcirc b'_n\xcirc U_n +\xi_n\xcirc b'_n\xcirc\omega_n\xcirc b'_n\\
& = U_n -\xi_n\xcirc U_{n-1}\xcirc b'_n +\xi_n\xcirc b'_n\xcirc\omega_n\xcirc b'_n\\
%
%
& = U_n &&\text{by the assumption}.
\end{align*}
Hence, $b'_{n+1}\xcirc\omega_{n+1} + \omega_n\xcirc b'_n = U_n$ on $E\ot\ov{E}^{\ot n}\ot E$.
\end{proof}

\begin{proposition}\label{formula para psi_n(y ot 1)} We have $\psi(x_0\ot \ov{\bx}_{1n}\ot 1_E) = (-1)^n \ov{\sigma}\xcirc\psi(x_0\ot \ov{\bx}_{1,n-1}\ot x_n)$ for all $n\ge 1$.
\end{proposition}

\begin{proof} By definition $\psi(E\ot \ov{E}^{\ot {n-1}}\ovb{\ot}\jmath_{\nu}(K)) \subseteq \ima(\ov{\sigma})$. So, by Proposition~\ref{sigma barra al cuadrado es cero},
$$
\psi(x_0\ot \ov{\bx}_{1n}\ot 1_E) = \ov{\sigma} \xcirc \psi \xcirc b'(x_0\ot \ov{\bx}_{1n}\ot 1_E) = (-1)^n \ov{\sigma}\xcirc\psi(x_0\ot \ov{\bx}_{1,n-1}\ot x_n),
$$
as desired.
\end{proof}

\begin{remark}\label{simp formula w} Let $x_0,\dots,x_n\in E$ and let $\bx\coloneqq x_0\ot\ov{\bx}_{1n}\ot 1_E$. Since
$$
\omega\bigl(E\ot\ov{E}^{\ot {n-1}}\ovb{\ot}\jmath_{\nu}(K)\bigr)\subseteq E\ot\ov{E}^{\ot n} \ovb{\ot}\jmath_{\nu}(K)\qquad\text{and}\qquad \xi(E\ot\ov{E}^{\ot n}\ovb{\ot} \jmath_{\nu}(K))=0,
$$
we have
$$
\omega(\bx) = \xi\bigl(\phi(\psi(\bx)) - \bx - \omega(b'(\bx))\bigr) = \xi \bigl(\phi(\psi(\bx)) - (-1)^n\omega(x_0\ot \ov{\bx}_{1,n-1}\ot x_n)\bigr).
$$
\end{remark}

\subsection{The filtrations of the resolutions}

\begin{notation}\label{notation F supra i} For each $n\ge 0$ and $i\ge -1$ we write
$
F^i(X_n)\coloneqq \bigoplus_{0\le s\le i} X_{n-s,s},\index{fe@$F^i(X_n)$|dotfillboldidx}
$
and we let $F^i\bigl(\ov{E}^{\ot n}\bigr)\index{fee@$F^i\bigl(\ov{E}^{\ot n}\bigr)$|dotfillboldidx}$ denote the $K$-subbimodule of $\ov{E}^{\ot n}$ generated by the tensors $\ov{\bx}_{1n}$ such that at least $n-i$ of the $x_j$'s belong to $\jmath_{\nu}(A)$. Furthermore, given a right $K$-submodule $E_1$ of $E$ and a left $K$-submodule $E_2$ of $E$, we let $E_1\bar{\ot} F^i(\ov{E}^{\ot n})\bar{\ot} E_2\index{ff@$E_1\bar{\ot} F^i(\ov{E}^{\ot n})\bar{\ot} E_2$|dotfillboldidx}$ denote the image of the canonical map $E_1\ot F^i\bigl(\ov{E}^{\ot n}\bigr)\ot E_2\longrightarrow E\ot \ov{E}^{\ot n}\ot E$. Moreover we write $F^i\bigl(E\ot \ov{E}^{\ot n}\ot E\bigr)\coloneqq E\bar{\ot} F^i(\ov{E}^{\ot n})\bar{\ot} E\index{fg@$F^i\bigl(E\ot \ov{E}^{\ot n}\ot E\bigr)$|dotfillboldidx}$.
\end{notation}

The normalized bar resolution $\bigl(E\ot\ov{E}^{\ot *}\ot E,b'_*\bigr)$ and the resolution $(X_*, d_*)$ are filtered by
\begin{align*}
& 0 = F^{-1}\bigl(E\ot\ov{E}^{\ot *}\ot E\bigr)\subseteq  F^0\bigl(E\ot\ov{E}^{\ot *}\ot E\bigr)\subseteq F^1\bigl(E\ot\ov{E}^{\ot *}\ot E\bigr)\subseteq\dots
\shortintertext{and}
& 0 = F^{-1}(X_*)\subseteq F^0(X_*)\subseteq F^1(X_*)\subseteq F^2(X_*)\subseteq F^3(X_*)\subseteq F^4(X_*) \subseteq \dots,
\end{align*}
respectively.

\begin{notation}\label{notation F sub R} Let $n,u\ge 0$ and $i\ge -1$. For each $k$-subalgebra $R$ of $A$, we let $\cramped{F_{\! R,u}^i\bigl(\ov{E}^{\ot n} \bigr)}\index{fg@$F_{R,u}^i(\ov{E}^{\ot n})$|dotfillboldidx}$ de\-note the $K$-subbimodule of $\cramped{\ov{E}^{\ot n}}$ generated by all the simple tensors $\ov{\bx}_{1n}$ such that
$$
\#\{j:x_j\!\notin\! \jmath_{\nu}(A)\cup\gamma(V)\} = 0,\quad \# \{j:x_j\!\notin\!\jmath_{\nu} (A)\}\le i \quad \text{and} \quad \# \{j:x_j\!\in\!\jmath_{\nu}(R)\}\ge u.
$$
Moreover, given a right $K$-submodule $E_1$ of $E$ and a left $K$-submodule $E_2$ of $E$, we let $\cramped{E_1 \bar{\ot} F_{\! R,u}^i(\ov{E}^{\ot n})\bar{\ot} E_2}\index{fh@$E_1\bar{\ot} F_{R,u}^i(\ov{E}^{\ot n})\bar{\ot} E_2$|dotfillboldidx}$ denote the image of the canonical map $E_1 \ot F_{\! R,u}^i(\ov{E}^{\ot n})\ot E_2\longrightarrow E\ot \ov{E}^{\ot n}\ot E$.
\end{notation}

\begin{remark} Note that $\cramped{F_{\! R,u}^{-1}\bigl(\ov{E}^{\ot n}\bigl)}=0$ and that $\cramped{F_{\! R,0}^i\bigl(\ov{E}^{\ot n}\bigl)}$ does not depend on $R$.
\end{remark}

\begin{definition} For $s,r\ge 0$, we let $\Sh_{sr}\colon  V^{\ot_k s}\ot_k\ov{A}^{\ot r}\to \ov{E}^{\ot {r+s}}\index{sd@$\Sh_{sr}$|dotfillboldidx}$ denote the map recursively defined~by:
\begin{align*}
&\Sh_{0r}(\ov{\ba}_{1r})\coloneqq  \ov{\jmath}_{\nu}(\ba_{1r})\\
\shortintertext{and}
& \Sh_{sr}(\bv_{1s}\ot_k \ov{\ba}_{1r}) \coloneqq \sum_{i=0}^r (-1)^i\sum_{l_i} \Sh_{s-1,i}(\bv_{1,s-1}\ot_k \ov{\ba}^{(l_i)}_{1i}) \ot \ov{\gamma} \bigl(v_s^{(l_i)}\bigr)\ot \ov{\jmath}_{\nu}(\ba_{i+1,r}) \quad\text{for $s\ge 1$,}
\end{align*}
where $\sum_{l_i} \ov{\ba}_{1i}^{(l_i)}\ot_k v_s^{(l_i)}\coloneqq  \ov{\chi}(v_s\ot_k \ov{\ba}_{1i})$.
\end{definition}

\begin{remark} Note that $\Sh_{s0}(\bv_{1s}) = \ov{\gamma}(\bv_{1s})$.
\end{remark}

\begin{remark}\label{imagen de Sh} Let $R$ be a $k$-subalgebra of $A$ and let $v_1,\dots,v_s\in V$ and $a_1,\dots,a_r\in A$. It is evident that if $R$ is stable under $\chi$, then $\cramped{\Sh_{sr}(\bv_{1s}\ot_k \ov{\ba}_{1r})\in F_{\! R,u}^s \bigl(\ov{E}^{\ot {r+s}} \bigr)}$, where $u=\#\{i:a_i\in R\}$.
\end{remark}

\begin{proposition}\label{propiedad de phi} Let $R$ be a subalgebra of $A$. If $R$ is stable under $\chi$ and $\mathcal{F}$ takes its values in $R\ot_k V$, then
$$
\phi\bigl(1_E\ot_{\hs A}\wt{\gamma}_{\hs A}(\bv_{1i})\ot\ov{\ba}_{1,n-i}\ot 1_E\!\bigr)\!\equiv\! 1_E \ot\Sh(\bv_{1i}\ot_k \ov{\ba}_{1,n-i})\ot 1_E \quad\!\! \mod{\jmath_{\nu}(A) \bar{\ot} F_{\! R,u+1}^{i-1}(\ov{E}^{\ot n})\bar{\ot} \jmath_{\nu}(K),}
$$
for all $v_1,\dots, v_i\in V$ and $a_1,\dots,a_{n-i}\in A$ such that $\#\{j:a_j\in R\}\ge u$.
\end{proposition}

\begin{proof} We proceed by induction on $n$. For $n=0$ the proposition is trivial. Assume that $n>0$ and that the proposition is true for $n-1$. Write $\bx \coloneqq  1_E\ot_{\hs A}\wt{\gamma}_{\hs A}(\bv_{1i})\ot\ov{\ba}_{1,n-i}\ot 1_E$. By item~(3) of Theorem~\ref{formula para d^1b}, Remark~\ref{imagen de Sh}, the inductive hypothesis and the definition of $\xi$,
$$
\xi\xcirc\phi\xcirc d^l(\bx)\in \jmath_{\nu}(A) \bar{\ot} F_{\! R,u+l-1}^{i-l+1}(\ov{E}^{\ot n})\bar{\ot} \jmath_{\nu}(K)  \qquad\text{for all $l>1$.}
$$
So,
$$
\phi(\bx) = \xi\xcirc\phi\xcirc d^0(\bx) +\xi\xcirc\phi\xcirc d^1(\bx) \quad \mod{\jmath_{\nu}(A) \bar{\ot} F_{\! R,u+1}^{i-1}(\ov{E}^{\ot n})\bar{\ot} \jmath_{\nu}(K)}.
$$
Note now that, since by the inductive hypothesis, $\phi(E\xcdot L_{i',n-i'-1})\subseteq \ima(\xi)\subseteq \ker(\xi)$, from the de\-finition of $d^0$ it follows that
$$
\xi\xcirc\phi\xcirc d^0(\bx)=(-1)^n\xi\xcirc\phi\bigl(1_E\ot_{\hs A} \wt{\gamma}_{\hs A}(\bv_{1i}) \ot\ov{\ba}_{1,n-i-1} \ot \jmath_{\nu}(a_{n-i})\bigr),
$$
while, by Theorem~\ref{formula para d^1} and Proposition~\ref{imagen de flat dag}, we have
$$
\xi\xcirc\phi\xcirc d^1(\bx) = \sum (-1)^i\xi\xcirc\phi\bigl(1_E\ot_{\hs A} \wt{\gamma}_{\hs A}(\bv_{1,i-1})\ot \ov{\ba}^{(l)}_{1,n-i}\ot\gamma(v^{(l)}_i)\bigr),
$$
where $\sum_l \ov{\ba}^{(l)}_{1,n-i}\ot_k v^{(l)}_i\coloneqq \ov{\chi}(v_i\ot_k \ov{\ba}_{1,n-i})$. Now we finish using the in\-ductive hypothesis.
\end{proof}

\begin{remark}\label{nota de propiedad de phi} If the hypothesis of Proposition~\ref{propiedad de phi} are satisfied, then  there exist maps
$$
\phi'_{in}\colon V^{\ot_k i}\ot_k\ov{A}^{\ot_{\hs K} {n-i}}\longrightarrow \jmath_{\nu}(A)\bar{\ot} F_{\! R,1}^{i-1}\bigl(\ov{E}^{\ot n}\bigr) \bar{\ot} \jmath_{\nu}(K) \qquad \text{for $0\le i\le n$},
$$
such that
$$
\phi_n\bigl(1_E\ot_{\hs A}\wt{\gamma}_{\hs A}(\bv_{1i})\ot\ov{\ba}_{1,n-i}\ot 1_E\bigr) = 1_E \ot\Sh(\bv_{1i}\ot_k \ov{\ba}_{1,n-i})\ot 1_E + \phi'_{in}(\bv_{1i}\ot_k \ov{\ba}_{1,n-i}),
$$
for all $v_1,\dots,v_i\in V$ and $a_1,\dots,a_{n-i}\in A$ (of course $\phi'_{0n} = 0$ for all $n$).
\end{remark}

\begin{notation}\label{complemento} Given $v_1,\dots,v_i\in V$ and $a_1,\dots,a_{n-i}\in A$ we will write $\phi'_{in}(\bv_{1i}\ot_k \ov{\ba}_{1,n-i})$ in the form $\phi''_{in}(\bv_{1i}\ot_k \ov{\ba}_{1,n-i})\ot 1_E$.
\end{notation}

\begin{theorem}\label{las func phi psi y omega preservan filtraciones} The maps $\phi$, $\psi$ and $\omega$ preserve filtrations.
\end{theorem}

\begin{proof} Since for $\phi$ this follows from Proposition~\ref{propiedad de phi}, we only must check it for $\psi$ and $\omega$. In order to abbreviate expressions we set
$
F_Q^i(X_n)\coloneqq  \bigoplus_{s=0}^i E\xcdot U_{n-s,s}$ for $-1\le i\le n$.
Note that $F_Q^{-1}(X_n) =0$. Let $0\le i\le n$. By Corollary~\ref{ima of ov sigma caso p}
\begin{equation}
\overline{\sigma}\bigl(F^i(X_n)\bigr)\subseteq \begin{cases} F_Q^i(X_{n+1}) & \text{if $i<n$,}\\ L_{0,n+1} + F_Q^n(X_{n+1}) & \text{if $i=n$.}\end{cases} \label{eq5}
\end{equation}
Furthermore, by the inclusion~\eqref{ec6},
\begin{equation}
\overline{\sigma}(E\xcdot W_{n-i,i})\subseteq E\xcdot L_{n+1-i,i}\qquad\text{for $0\le i\le n$.}\label{eq7}
\end{equation}
We claim that
\begin{equation}
\psi\bigl(E\bar{\ot} F^i\bigl(\ov{E}^{\ot n}\bigr)\bar{\ot} \jmath_{\nu}(K)\bigr)\subseteq E\xcdot L_{n-i,i} + F_Q^{i-1}(X_n).\label{eq8}
\end{equation}
For $n=0$ this is trivial. Suppose~\eqref{eq8} is valid for $n$. Let $x_0\ot \ov{\bx}_{1,n+1}\ot 1_E \in E\bar{\ot} F^i\bigl(\ov{E}^{\ot {n+1}}\bigr)\bar{\ot}\jmath_{\nu}(K)$, where $0\le i\le n+1$. By Proposition~\ref{formula para psi_n(y ot 1)}, to prove~\eqref{eq8} for $n+1$, we only must check that
$$
\ov{\sigma}(\psi(x_0\ot\ov{\bx}_{1n}\ot x_{n+1})) \subseteq E\xcdot L_{n+1-i,i} + F_Q^{i-1}(X_{n+1}).
$$
If $x_{n+1}\in \jmath_{\nu}(A)$, then $i\le n$ and using~\eqref{eq5}, \eqref{eq7} and the inductive hypothesis, we get
$$
\ov{\sigma}\bigl(\psi(x_0\ot \ov{\bx}_{1n}\ot x_{n+1})\bigr) \subseteq \ov{\sigma}\bigl(E\xcdot W_{{n-i},i} + F^{i-1}(X_n)\bigr) \subseteq E\xcdot L_{n+1-i,i} + F_Q^{i-1}(X_{n+1});
$$
%
%
while if $x_{n+1}\notin \jmath_{\nu}(A)$, then $\cramped{x_0\ot \ov{\bx}_{1n}\ot x_{n+1}\in F^{i-1}\bigl(E\ot \ov{E}^{\ot n}\ot E\bigr)}$, which together with~\eqref{eq5} and the inductive hypothesis, implies that
$$
\ov{\sigma}\bigl(\psi(x_0\ot\ov{\bx}_{1n}\ot x_{n+1})\bigr) \subseteq \ov{\sigma}\bigl(F^{i-1}(X_n) \bigr) \subseteq E\xcdot L_{n+1-i,i} + F_Q^{i-1}(X_{n+1}),
$$
which finishes the proof of the claim. From~\eqref{eq8} it follows immediately that $\psi$ preserves fil\-trations. Next, we prove that $\omega$ also does it. More concretely we claim that
$$
w_{n+1}\bigl(E\bar{\ot} F^i\bigl(\ov{E}^{\ot n}\bigr)\bar{\ot} \jmath_{\nu}(K)\bigr)\subseteq E\bar{\ot} F^i\bigl(\ov{E}^{\ot {n+1}}\bigr)\bar{\ot} \jmath_{\nu}(K).
$$
We proceed by induction on $n$. The case $n=0$ is trivial, since $\omega_1=0$. Assume that the previous inclusion is satisfied for $n$ and let
$\bx \coloneqq  x_0\ot\ov{\bx}_{1n}\ot 1_E\in E\bar{\ot} F^i\bigl(\ov{E}^{\ot n}\bigr)\bar{\ot} \jmath_{\nu}(K)$. By Remark~\ref{simp formula w}, $$
\omega(\bx) = \xi\xcirc \phi\xcirc \psi(\bx) + (-1)^n \xi\xcirc\omega(x_0\ot\ov{\bx}_{1,n-1}\ot x_n).
$$
Since, by inclusion~\eqref{eq8}, Proposition~\ref{propiedad de phi} and the definition of $\xi$,
\begin{align*}
\xi\xcirc\phi\xcirc \psi(\bx)& \in \xi\xcirc \phi\bigl(E\xcdot L_{n-i,i} + F_Q^{i-1}(X_n)\bigr)\\
& \subseteq \xi\Bigl(E\bar{\ot} F_{A,0}^i(\ov{E}^{\ot n})\bar{\ot} \jmath_{\nu}(K) + E \bar{\ot} F_{A,0}^{i-1}(\ov{E}^{\ot n})\bar{\ot} \jmath_{\nu}(K)\gamma(V)\Bigr)\\
& \subseteq E\bar{\ot} F_{A,1}^i(\ov{E}^{\ot {n+1}})\bar{\ot} \jmath_{\nu}(K),
\end{align*}
to finish the proof of the claim it remains to check that $\xi\bigl(\omega(x_0\ot\ov{\bx}_{1,n-1}\ot x_n)\bigr) \in E\bar{\ot} F^i\bigl(\ov{E}^{\ot {n+1}}\bigr)\bar{\ot} \jmath_{\nu}(K)$. If $x_n\in \jmath_{\nu}(A)$, then, by the inductive hypothesis,
$$
\xi\bigl(\omega(x_0\ot\ov{\bx}_{1,n-1}\ot x_n)\bigr) \subseteq \xi\bigl(E\bar{\ot} F^i\bigl(\ov{E}^{\ot n}\bigr)\bar{\ot} \jmath_{\nu}(A)\bigr) \subseteq E\bar{\ot} F^i\bigl(\ov{E}^{\ot {n+1}}\bigr)\bar{\ot} \jmath_{\nu}(K);
$$
%
%
while if $x_n\notin \jmath_{\nu}(A)$, then $x_0\ot\ov{\bx}_{1,n-1}\ot x_n\in F^{i-1}\bigl(E\ot\ov{E}^{\ot {n-1}}\ot E\bigr)$, and so
$$
\xi\bigl(\omega(x_0\ot\ov{\bx}_{1,n-1}\ot x_n)\bigr) \in \xi\bigl(F^{i-1}\bigl(E\ot\ov{E}^{\ot {n-1}}\ot E\bigr)\bigr) \subseteq E\bar{\ot} F^i\bigl(\ov{E}^{\ot {n+1}}\bigr)\bar{\ot} \jmath_{\nu}(K),
$$
as we want.
\end{proof}

\section{Hochschild homology of general crossed products}\label{section: Hochschild homology of general crossed products}
In this section we use freely the notations introduced in Section~\ref{section: A resolution for a general crossed product}. We assume that all the prop\-er\-ties required in that section are fulfilled (for the convenience of the reader we recall that all these prop\-er\-ties were established at the beginning of Section~\ref{section: A resolution for a general crossed product}). Let~$M$ be an $E$-bimodule. By definition the Hochschild homology $\Ho^{\hs K}_*(E,M)\index{hg@$\Ho^{\hs K}_*(E,M)$|dotfillboldidx}$, of the $K$-algebra $E$ with coefficients in $M$, is the homology of the normalized Hochschild chain complex $\cramped{\bigl(M\ot\ov{E}^{\ot *}\ot,b_*\bigr)}$, where $b_*$ is the canonical Hochschild boundary map. It is well known that $\Ho^{\hs K}_*(E,M)$ is the $\Tor$ functor relative to $\Upsilon$. Since $(X_*,d_*)$ is a $\Upsilon$-relative projective resolution of~$E$, the Hochschild homology of the $K$-algebra $E$ with coefficients in $M$ is the homology of $M\ot_{E^e}(X_*,d_*)$. It is well known that when $K$ is separable, $\Ho^{\hs K}_*(E,M)$ is the absolute Hochschild homology of $E$ with coefficients in $M$. We consider $M$ as an $A$-bimodule through the map $\jmath_{\nu}\colon A\to E$. For $r,s\ge 0$, write
$$
X_{rs}(M)\coloneqq  M\ot_{\hs A} E^{\ot_{\hs A} s}\ot \ov{A}^{\ot r} \ot\index{xi@$X_{rs}(M)$|dotfillboldidx}\quad\text{and}\quad \wh{X}_{rs}(M)\coloneqq  M\ot_{\hs A}\wt{E}^{\ot_{\hs A} s}\ot\ov{A}^{\ot r}\ot\index{xe@$\wh{X}_{rs}(M)$|dotfillboldidx}.
$$
Since $\wt{E}^{\ot_{\hs A} 0}= \jmath_{\nu}(A)$ and $\ov{A}^{\ot 0}=K$, we have
\begin{equation}\label{ec10*}
\wh{X}_{r0}(M)\simeq M\ot\ov{A}^{\ot r}\ot\qquad\text{and}\qquad \wh{X}_{0s}(M)\simeq M\ot_{\hs A}\wt{E}^{\ot_{\hs A} s}\ot.
\end{equation}
Throughout this section we let $[m\ot_{\hs A}\wt{\bx}_{1s}\ot\ov{\ba}_{1r}]$ denote the class of $m\ot_{\hs A}\wt{\bx}_{1s}\ot\ov{\ba}_{1r}$ in $\wh{X}_{rs}(M)$. It is easy to check that the map
$$
\begin{tikzpicture}
\draw [->] node[left=0pt]{$\wh{X}_{rs}(M)$} (0,0) -- (2.5,0) node[right=0pt]{$M\ot_{E^e} X_{rs}$};
\draw[|->] (0.5,-0.5) node[left=0pt]{$[m\ot_{\hs A}\wt{\bx}_{1s}\ot\ov{\ba}_{1r}]$} -- (1.7,-0.5) node[right=0pt]{$m\ot_{E^e}(1_E\ot_{\hs A} \wt{\bx}_{1s}\ot\ov{\ba}_{1r}\ot 1_E)$};
\end{tikzpicture}
$$
is an isomorphism. Let $\wh{d}^l_{rs}\colon\wh{X}_{rs}(M)\longrightarrow \wh{X}_{r+l-1,s-l}(M)\index{dd@$\wh{d}^l_{rs}$|dotfillboldidx}$ be the map induced by $\ide_M\ot_{E^e} d^l_{rs}$. Via the above identifications, $M\ot_{E^e}(X_*,d_*)$ becomes the complex $(\wh{X}_*(M),\wh{d}_*)$, where
$$
\wh{X}_n(M)\coloneqq \bigoplus_{r+s = n}\wh{X}_{rs}(M)\index{xf@$\wh{X}_n(M)$|dotfillboldidx}\qquad\text{and}\qquad \wh{d}_n(\bx)\coloneqq \begin{cases} \displaystyle{\sum_{l=1}^n \wh{d}^l_{0n}(\bx)} &\text{if $\bx\in \wh{X}_{0n}$,}\\ \displaystyle{\sum^{n-r}_{l=0} \wh{d}^l_{r,n-r}(\bx)} &\text{if $\bx\in \wh{X}_{r,n-r}$ with $r>0$.}\end{cases}\index{de@$\wh{d}_n$|dotfillboldidx}
$$
Consequently, we have the following result:

\begin{theorem}\label{Hochschild homology} The Hochschild homology $\Ho^{\hs K}_*(E,M)$, of the $K$-algebra $E$ with coefficients in~$M$, is~the~ho\-mol\-o\-gy of $(\wh{X}_*(M),\wh{d}_*)$.
\end{theorem}

\begin{remark}\label{remark: caso F toma valores en K} It follows from Corollary~\ref{caso complej doble} that, if $\mathcal{F}$ takes its values in $K\ot_k V$, then $(\wh{X}_*(M),\wh{d}_*)$ is the total complex of the double complex $(\wh{X}_{**}(M),\wh{d}^0_{**},\wh{d}^1_{**})$.
\end{remark}

\begin{remark}\label{remark: caso K=A} If $K=A$, then $(\wh{X}_*(M),\wh{d}_*) = (\wh{X}_{0*}(M),\wh{d}^1_{0*})$.
\end{remark}

Likewise for $\wh{X}_{rs}(M)$, there are canonical identifications $X_{rs}(M)\simeq M\ot_{E^e} X'_{rs}$. For $r\ge0$, $s>0$ and $0\le i\le s$, let $\mu_i\colon X_{rs}(M)\longrightarrow X_{r,s-1}(M)\index{zmca@$\mu_i$|dotfillboldidx}$ be the map induced by $\mu'_i$. Clearly
$$
\mu_i([m\ot_{\hs A} \bx_{1s}\ot\ov{\ba}_{1r}])\! = \! \begin{cases} [m\xcdot x_1\ot_{\hs A} \bx_{2s}\ot\ov{\ba}_{1r}] &\text{if $i=0$,}\\
[m\ot_{\hs A}\bx_{1,i-1}\ot_{\hs A} x_ix_{i+1} \ot_{\hs A} \bx_{i+2,s}\ot \ov{\ba}_{1r}] & \text{if $0<i<s$,}\\
\sum_l [\gamma(v^{(l)})\xcdot m\ot_{\hs A}\bx_{1,s-2}\ot_{\hs A} x_{s-1}\jmath_{\nu}(a) \ot \ov{\ba}_{1r}^{(l)}] & \text{if $i=s$,}
\end{cases}
$$
where $m\in M$, $x_1,\dots,x_{s-1}\in E$, $a_1,\dots,a_r\in A$, $x_s\coloneqq \jmath_{\nu}(a)\gamma(v)$ with $a\in A$ and $v\in V$, and
$$
\sum_l \ov{\ba}_{1r}^{(l)}\ot_k v^{(l)} \coloneqq  \ov{\chi}\bigl(v\ot_k \ov{\ba}_{1r}\bigr).
$$

\begin{notation} Given a $k$-subalgebra $R$ of $A$ and $0\!\le\! u\!\le\! r$, we let $\wh{X}^u_{rs}(R,M)\index{xj@$\wh{X}^u_{rs}(R,M)$|dotfillboldidx}$ denote the $k$-sub\-module of $\wh{X}_{rs}(M)$ generated by all the circular simple tensors $\bigl[m\ot_{\hs A}\wt{\gamma}_{\hs A}(\bv_{1s}) \ot \ov{\ba}_{1r}\bigr]$, with $m\in M$, $v_1,\dots,v_s\in V$, $a_1,\dots,a_r\in A$, and at least $u$ of the $a_j$'s in $R$.
\end{notation}

\begin{theorem}\label{teorema wh{d^O}, wh{d^1} y wh{d^l} con l>=2} The following assertions hold:

\begin{enumerate}[itemsep=0.7ex, topsep=1.0ex, label=\emph{(\arabic*)}]

\item The morphism $\wh{d}^0\colon\wh{X}_{rs}(M)\to\wh{X}_{r-1,s}(M)$ is $(-1)^s$-times the boundary map of the normalized chain Hochschild complex of the $K$-algebra $A$ with coefficients in $M\ot_{\hs A} \wt{E}^{ \ot_{\hs A} s}$, considered as an $A$-bimodule via the left and right canonical actions.

\item The morphism $\wh{d}^1\colon\wh{X}_{rs}(M)\to\wh{X}_{r,s-1}(M)$ is induced by $\sum_{i=0}^s (-1)^i \mu_i$.

\item Let $R$ be a $k$-subalgebra of $A$. If $R$ is stable under $\chi$ and $\mathcal{F}$ takes its values in $R\ot_k V$, then $\wh{d}^l\bigl(\wh{X}_{rs}(M)\bigr)\subseteq\wh{X}^{l-1}_{r+l-1,s-l}(R,M)$, for each $r\ge 0$ and $1<l\le s$.
\end{enumerate}
\end{theorem}

\begin{proof} Item~(1) follows from the definition of $d^0$; item~(2), from Theorem~\ref{formula para d^1}; and item~(3), from The\-o\-rem~\ref{formula para d^1b}(3).
\end{proof}

\begin{theorem}\label{formula para wh{d}^2} Under the hypothesis of Theorem~\ref{formula para d^2} the map $\wh{d}^2\colon \wh{X}_{rs}(M)\to \wh{X}_{r+1,s-2}(M)$ is given by
$$
\wh{d}^2\bigl([m\ot_{\hs A} \wt{\gamma}_{\hs A}(\bv_{1s})\ot \ov{\ba}_{1r}]\bigr) = (-1)^{s+1} \bigl[\gamma(v_{s-1}^{(2)}v_s^{(2)})\xcdot m\ot_{\hs A} \wt{\gamma}_{\hs A}(\bv_{1,s-2})\ot \mathfrak{T}(v^{(1)}_{s-1},v^{(1)}_s,\ov{\ba}_{1r})\bigr],
$$
where $m\in M$, $v_1,\dots,v_s\in V$, $a_1,\dots,a_r\in A$ and $\mathfrak{T}(v_{s-1},v_s,\ov{\ba}_{1r})$ is as in Theorem~\ref{formula para d^2}.
\end{theorem}

\begin{proof} This follows immediately from Theorem~\ref{formula para d^2}.
\end{proof}

\begin{proposition} If $\mathcal{F}$ takes its values in $K\ot_k V$, then $A\ot \ov{A}^{\ot r}\ot M$ is an $E$-bimodule via
$$
\jmath_{\nu}(a)\gamma(v) \xcdot T\xcdot \jmath_{\nu}(a')\gamma(v') \coloneqq \sum a a_0^{(l)} \ot \ov{\ba}_{1r}^{(l)}\ot \gamma(v^{(l)})\xcdot m\xcdot \jmath_{\nu}(a')\gamma(v'),
$$
where $T\coloneqq a_0\ot \ov{\ba}_{1r}\ot m$ and $\sum_l a_0^{(l)}\ot \ov{\ba}_{1r}^{(l)}\ot_k v^{(l)} = (A\ot \ov{\chi})\xcirc(\chi\ot \ov{A}^{\ot_K r})(v\ot_k a_0 \ot \ov{\ba}_{1r})$.
\end{proposition}

\begin{proof} It is clear that $A\ot \ov{A}^{\ot r} \ot M$ is a right $E$-module. We next prove it is also a left $E$-module. For this it suffices to show that the left action is unitary and that
\begin{equation}\label{epa}
\gamma(v)\jmath_{\nu}(a)\xcdot T = \gamma(v)\cdot \bigl(\jmath_{\nu}(a)\xcdot T\bigr)\quad\text{and}\quad \gamma(v_1)\gamma(v_2)\xcdot T = \gamma(v_1)\cdot \bigl(\gamma(v_2)\xcdot T\bigr).
\end{equation}
The first equality in~\eqref{epa} follows easily using equality~\eqref{eq12}; while, by equality~\eqref{eq12'}, for the second one it suffices to note that, by the twisted module condition, the fact that $\mathcal{F}$ takes its values in $K\ot_k V$ and an inductive argument, we have
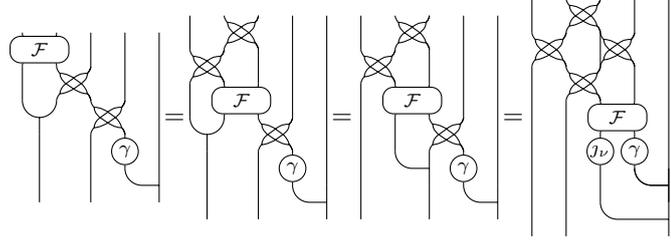
\begin{equation*}
\begin{tikzpicture}[scale=0.45]
\def\braid(#1,#2)[#3]{\draw (#1+1*#3,#2) .. controls (#1+1*#3,#2-0.05*#3) and (#1+0.96*#3,#2-0.15*#3).. (#1+0.9*#3,#2-0.2*#3) (#1+0.1*#3,#2-0.8*#3)--(#1+0.9*#3,#2-0.2*#3) (#1,#2-1*#3) .. controls (#1,#2-0.95*#3) and (#1+0.04*#3,#2-0.85*#3).. (#1+0.1*#3,#2-0.8*#3) (#1,#2) .. controls (#1,#2-0.05*#3) and (#1+0.04*#3,#2-0.15*#3).. (#1+0.1*#3,#2-0.2*#3) (#1+0.1*#3,#2-0.2*#3) -- (#1+0.37*#3,#2-0.41*#3) (#1+0.62*#3,#2-0.59*#3)-- (#1+0.9*#3,#2-0.8*#3) (#1+1*#3,#2-1*#3) .. controls (#1+1*#3,#2-0.95*#3) and (#1+0.96*#3,#2-0.85*#3).. (#1+0.9*#3,#2-0.8*#3)}
\def\braidmamed(#1,#2)[#3]{\draw (#1,#2-0.5)  node[name=nodemap,inner sep=0pt,  minimum size=10pt, shape=circle,draw]{$#3$}
(#1-0.5,#2) .. controls (#1-0.5,#2-0.15) and (#1-0.4,#2-0.2) .. (#1-0.3,#2-0.3) (#1-0.3,#2-0.3) -- (nodemap)
(#1+0.5,#2) .. controls (#1+0.5,#2-0.15) and (#1+0.4,#2-0.2) .. (#1+0.3,#2-0.3) (#1+0.3,#2-0.3) -- (nodemap)
(#1+0.5,#2-1) .. controls (#1+0.5,#2-0.85) and (#1+0.4,#2-0.8) .. (#1+0.3,#2-0.7) (#1+0.3,#2-0.7) -- (nodemap)
(#1-0.5,#2-1) .. controls (#1-0.5,#2-0.85) and (#1-0.4,#2-0.8) .. (#1-0.3,#2-0.7) (#1-0.3,#2-0.7) -- (nodemap)
}
\def\cocycle(#1,#2)[#3]{\draw (#1,#2) .. controls (#1,#2-0.555*#3/2) and (#1+0.445*#3/2,#2-#3/2) .. (#1+#3/2,#2-#3/2) .. controls (#1+1.555*#3/2,#2-#3/2) and (#1+2*#3/2,#2-0.555*#3/2) .. (#1+2*#3/2,#2) (#1+#3/2,#2-#3/2) -- (#1+#3/2,#2-2*#3/2) (#1+#3/2,#2-#3/2)  node [inner sep=0pt,minimum size=3pt,shape=circle,fill] {}}
\def\cocyclenamed(#1,#2)[#3][#4]{\draw (#1+#3/2,#2-1*#3/2) node[name=nodemap,inner sep=0pt,  minimum size=9pt, shape=circle,draw]{$#4$}
(#1,#2) .. controls (#1,#2-0.555*#3/2) and (#1+0.445*#3/2,#2-#3/2) .. (nodemap) .. controls (#1+1.555*#3/2,#2-#3/2) and (#1+2*#3/2,#2-0.555*#3/2) .. (#1+2*#3/2,#2) (nodemap)-- (#1+#3/2,#2-2*#3/2)
}
\def\cocycleTR90(#1,#2)[#3]{\draw (#1,#2) .. controls (#1,#2-0.555*#3/2) and (#1+0.445*#3/2,#2-#3/2) .. (#1+#3/2,#2-#3/2) .. controls (#1+1.555*#3/2,#2-#3/2) and (#1+2*#3/2,#2-0.555*#3/2) .. (#1+2*#3/2,#2) (#1+#3/2,#2-#3/2) -- (#1+#3/2,#2-2*#3/2) (#1+#3/2,#2-#3/2)  node [inner sep=0pt,minimum size=3pt,shape=isosceles triangle, rotate=90,fill] {}}
\def\comult(#1,#2)[#3,#4]{\draw (#1,#2) -- (#1,#2-0.5*#4) arc (90:0:0.5*#3 and 0.5*#4) (#1,#2-0.5*#4) arc (90:180:0.5*#3 and 0.5*#4)}
\def\counit(#1,#2){\draw (#1,#2) -- (#1,#2-0.93) (#1,#2-1) circle[radius=2pt]}
\def\cuadrupledoublemap(#1,#2)[#3]{\draw (#1+1.5,#2-0.5) node [name=doublesinglemapnode,inner xsep=0pt, inner ysep=0pt, minimum height=9pt, minimum width=40pt,shape=rectangle,draw,rounded corners] {$#3$} (#1+1,#2) .. controls (#1+1,#2-0.075) .. (doublesinglemapnode)
(#1,#2) .. controls (#1,#2-0.075) .. (doublesinglemapnode) (#1+2,#2) .. controls (#1+2,#2-0.075) .. (doublesinglemapnode) (#1+3,#2) .. controls (#1+3,#2-0.075) .. (doublesinglemapnode) (doublesinglemapnode)-- (#1+1,#2-1) (doublesinglemapnode)-- (#1+2,#2-1)
}
\def\doublemap(#1,#2)[#3]{\draw (#1+0.5,#2-0.5) node [name=doublemapnode,inner xsep=0pt, inner ysep=0pt, minimum height=10pt, minimum width=22pt,shape=rectangle,draw,rounded corners] {$#3$} (#1,#2) .. controls (#1,#2-0.075) .. (doublemapnode) (#1+1,#2) .. controls (#1+1,#2-0.075).. (doublemapnode) (doublemapnode) .. controls (#1,#2-0.925)..(#1,#2-1) (doublemapnode) .. controls (#1+1,#2-0.925).. (#1+1,#2-1)}
\def\doublesinglemap(#1,#2)[#3]{\draw (#1+0.5,#2-0.5) node [name=doublesinglemapnode,inner xsep=0pt, inner ysep=0pt, minimum height=10pt, minimum width=2pt,shape=rectangle,draw,rounded corners] {$#3$} (#1,#2) .. controls (#1,#2-0.075) .. (doublesinglemapnode) (#1+1,#2) .. controls (#1+1,#2-0.075).. (doublesinglemapnode) (doublesinglemapnode)-- (#1+0.5,#2-1)}
\def\flip(#1,#2)[#3]{\draw (
#1+1*#3,#2) .. controls (#1+1*#3,#2-0.05*#3) and (#1+0.96*#3,#2-0.15*#3).. (#1+0.9*#3,#2-0.2*#3)
(#1+0.1*#3,#2-0.8*#3)--(#1+0.9*#3,#2-0.2*#3)
(#1,#2-1*#3) .. controls (#1,#2-0.95*#3) and (#1+0.04*#3,#2-0.85*#3).. (#1+0.1*#3,#2-0.8*#3)
(#1,#2) .. controls (#1,#2-0.05*#3) and (#1+0.04*#3,#2-0.15*#3).. (#1+0.1*#3,#2-0.2*#3)
(#1+0.1*#3,#2-0.2*#3) -- (#1+0.9*#3,#2-0.8*#3)
(#1+1*#3,#2-1*#3) .. controls (#1+1*#3,#2-0.95*#3) and (#1+0.96*#3,#2-0.85*#3).. (#1+0.9*#3,#2-0.8*#3)
}
\def\laction(#1,#2)[#3,#4]{\draw (#1,#2) .. controls (#1,#2-0.555*#4/2) and (#1+0.445*#4/2,#2-1*#4/2) .. (#1+1*#4/2,#2-1*#4/2) -- (#1+2*#4/2+#3*#4/2,#2-1*#4/2) (#1+2*#4/2+#3*#4/2,#2)--(#1+2*#4/2+#3*#4/2,#2-2*#4/2)}
\def\lactionnamed(#1,#2)[#3,#4][#5]{\draw (#1+2*#4/2+#3*#4/2,#2-#4/2)  node[name=nodemap,inner sep=0pt,  minimum size=10pt, shape=circle,draw]{$#5$}
(#1,#2) .. controls (#1,#2-0.555*#4/2) and (#1+0.445*#4/2,#2-1*#4/2) .. (#1+1*#4/2,#2-1*#4/2) -- (nodemap)
(#1+2*#4/2+#3*#4/2,#2)  -- (nodemap)
(nodemap)--(#1+2*#4/2+#3*#4/2,#2-2*#4/2)
}
\def\map(#1,#2)[#3]{\draw (#1,#2-0.5)  node[name=nodemap,inner sep=0pt,  minimum size=10pt, shape=circle,draw]{$#3$} (#1,#2)-- (nodemap)  (nodemap)-- (#1,#2-1)}
\def\mult(#1,#2)[#3,#4]{\draw (#1,#2) arc (180:360:0.5*#3 and 0.5*#4) (#1+0.5*#3, #2-0.5*#4) -- (#1+0.5*#3,#2-#4)}
\def\raction(#1,#2)[#3,#4]{\draw (#1,#2) -- (#1,#2-2*#4/2)  (#1,#2-1*#4/2)--(#1+1*#4/2+#3*#4/2,#2-1*#4/2) .. controls (#1+1.555*#4/2+#3*#4/2,#2-1*#4/2) and (#1+2*#4/2+#3*#4/2,#2-0.555*#4/2) .. (#1+2*#4/2+#3*#4/2,#2)}
\def\mult(#1,#2)[#3,#4]{\draw (#1,#2) arc (180:360:0.5*#3 and 0.5*#4) (#1+0.5*#3, #2-0.5*#4) -- (#1+0.5*#3,#2-#4)}
\def\rcoaction(#1,#2)[#3,#4]{\draw (#1,#2)-- (#1,#2-2*#4/2) (#1,#2-1*#4/2) -- (#1+1*#4/2+#3*#4/2,#2-1*#4/2).. controls (#1+1.555*#4/2+#3*#4/2,#2-1*#4/2) and (#1+2*#4/2+#3*#4/2,#2-1.445*#4/2) .. (#1+2*#4/2+#3*#4/2,#2-2*#4/2)}
\def\singledoublemap(#1,#2)[#3]{\draw (#1+0.5,#2-0.5) node [name=doublemapnode,inner xsep=0pt, inner ysep=0pt, minimum height=10pt, minimum width=22pt,shape=rectangle,draw,rounded corners] {$#3$} (#1+0.5,#2)--(doublemapnode) (doublemapnode) .. controls (#1,#2-0.925)..(#1,#2-1) (doublemapnode) .. controls (#1+1,#2-0.925).. (#1+1,#2-1) }
\def\singletriplemap(#1,#2)[#3]{\draw (#1+0.5,#2-0.5) node [name=triplemapnode,inner xsep=0pt, inner ysep=0pt, minimum height=9pt, minimum width=27pt, shape=rectangle,draw,rounded corners] {$#3$} (#1+0.5,#2)--(triplemapnode) (triplemapnode) .. controls (#1-0.5,#2-0.925)..(#1-0.5,#2-1) (triplemapnode) .. controls (#1+0.5,#2-0.925).. (#1+0.5,#2-1) (triplemapnode) .. controls (#1+1.5,#2-0.925).. (#1+1.5,#2-1) }
\def\transposition(#1,#2)[#3]{\draw (#1+#3,#2) .. controls (#1+#3,#2-0.05*#3) and (#1+0.96*#3,#2-0.15*#3).. (#1+0.9*#3,#2-0.2*#3) (#1+0.1*#3,#2-0.8*#3)--(#1+0.9*#3,#2-0.2*#3) (#1+0.1*#3,#2-0.2*#3) .. controls  (#1+0.3*#3,#2-0.2*#3) and (#1+0.46*#3,#2-0.31*#3) .. (#1+0.5*#3,#2-0.34*#3) (#1,#2-1*#3) .. controls (#1,#2-0.95*#3) and (#1+0.04*#3,#2-0.85*#3).. (#1+0.1*#3,#2-0.8*#3) (#1,#2) .. controls (#1,#2-0.05*#3) and (#1+0.04*#3,#2-0.15*#3).. (#1+0.1*#3,#2-0.2*#3) (#1+0.1*#3,#2-0.2*#3) .. controls  (#1+0.1*#3,#2-0.38*#3) and (#1+0.256*#3,#2-0.49*#3) .. (#1+0.275*#3,#2-0.505*#3) (#1+0.50*#3,#2-0.66*#3) .. controls (#1+0.548*#3,#2-0.686*#3) and (#1+0.70*#3,#2-0.8*#3)..(#1+0.9*#3,#2-0.8*#3) (#1+0.72*#3,#2-0.50*#3) .. controls (#1+0.80*#3,#2-0.56*#3) and (#1+0.9*#3,#2-0.73*#3)..(#1+0.9*#3,#2-0.8*#3) (#1+#3,#2-#3) .. controls (#1+#3,#2-0.95*#3) and (#1+0.96*#3,#2-0.85*#3).. (#1+0.9*#3,#2-0.8*#3)}
\def\tripledoublemap(#1,#2)[#3]{\draw (#1,#2-0.5) node [name=doublesinglemapnode,inner xsep=0pt, inner ysep=0pt, minimum height=10pt, minimum width=30pt, shape=rectangle, draw, rounded corners] {$#3$} (#1-1,#2) .. controls (#1-1,#2-0.075) .. (doublesinglemapnode)
(#1,#2) .. controls (#1,#2-0.075) .. (doublesinglemapnode) (#1+1,#2) .. controls (#1+1,#2-0.075) .. (doublesinglemapnode) (doublesinglemapnode)-- (#1-0.5,#2-1) (doublesinglemapnode)-- (#1+0.5,#2-1)
}
\def\triplesinglemap(#1,#2)[#3]{\draw (#1,#2-0.5) node [name=doublesinglemapnode,inner xsep=0pt, inner ysep=0pt, minimum height=10pt, minimum width=30pt, shape=rectangle, draw, rounded corners] {$#3$} (#1-1,#2) .. controls (#1-1,#2-0.075) .. (doublesinglemapnode)
(#1,#2) .. controls (#1,#2-0.075) .. (doublesinglemapnode) (#1+1,#2) .. controls (#1+1,#2-0.075) .. (doublesinglemapnode) (doublesinglemapnode)-- (#1,#2-1)
}
\def\twisting(#1,#2)[#3]{\draw (#1+#3,#2) .. controls (#1+#3,#2-0.05*#3) and (#1+0.96*#3,#2-0.15*#3).. (#1+0.9*#3,#2-0.2*#3) (#1,#2-1*#3) .. controls (#1,#2-0.95*#3) and (#1+0.04*#3,#2-0.85*#3).. (#1+0.1*#3,#2-0.8*#3) (#1+0.1*#3,#2-0.8*#3) ..controls (#1+0.25*#3,#2-0.8*#3) and (#1+0.45*#3,#2-0.69*#3) .. (#1+0.50*#3,#2-0.66*#3) (#1+0.1*#3,#2-0.8*#3) ..controls (#1+0.1*#3,#2-0.65*#3) and (#1+0.22*#3,#2-0.54*#3) .. (#1+0.275*#3,#2-0.505*#3) (#1+0.72*#3,#2-0.50*#3) .. controls (#1+0.75*#3,#2-0.47*#3) and (#1+0.9*#3,#2-0.4*#3).. (#1+0.9*#3,#2-0.2*#3) (#1,#2) .. controls (#1,#2-0.05*#3) and (#1+0.04*#3,#2-0.15*#3).. (#1+0.1*#3,#2-0.2*#3) (#1+0.5*#3,#2-0.34*#3) .. controls (#1+0.6*#3,#2-0.27*#3) and (#1+0.65*#3,#2-0.2*#3).. (#1+0.9*#3,#2-0.2*#3) (#1+#3,#2-#3) .. controls (#1+#3,#2-0.95*#3) and (#1+0.96*#3,#2-0.85*#3).. (#1+0.9*#3,#2-0.8*#3) (#1+#3,#2) .. controls (#1+#3,#2-0.05*#3) and (#1+0.96*#3,#2-0.15*#3).. (#1+0.9*#3,#2-0.2*#3) (#1+0.1*#3,#2-0.2*#3) .. controls  (#1+0.3*#3,#2-0.2*#3) and (#1+0.46*#3,#2-0.31*#3) .. (#1+0.5*#3,#2-0.34*#3) (#1+0.1*#3,#2-0.2*#3) .. controls  (#1+0.1*#3,#2-0.38*#3) and (#1+0.256*#3,#2-0.49*#3) .. (#1+0.275*#3,#2-0.505*#3) (#1+0.50*#3,#2-0.66*#3) .. controls (#1+0.548*#3,#2-0.686*#3) and (#1+0.70*#3,#2-0.8*#3)..(#1+0.9*#3,#2-0.8*#3) (#1+#3,#2-1*#3) .. controls (#1+#3,#2-0.95*#3) and (#1+0.96*#3,#2-0.85*#3).. (#1+0.9*#3,#2-0.8*#3) (#1+0.72*#3,#2-0.50*#3) .. controls (#1+0.80*#3,#2-0.56*#3) and (#1+0.9*#3,#2-0.73*#3)..(#1+0.9*#3,#2-0.8*#3)(#1+0.72*#3,#2-0.50*#3) -- (#1+0.50*#3,#2-0.66*#3) -- (#1+0.275*#3,#2-0.505*#3) -- (#1+0.5*#3,#2-0.34*#3) -- (#1+0.72*#3,#2-0.50*#3)}
\def\unit(#1,#2){\draw (#1,#2) circle[radius=2pt] (#1,#2-0.07) -- (#1,#2-1)}
\begin{scope}[xshift=0cm, yshift=-1cm]
\doublemap(0,0)[\scriptstyle \mathcal{F}]; \draw (2,0)-- (2,-1); \draw (3,0)-- (3,-2); \draw (4,0)-- (4,-4); \draw (0,-1)-- (0,-2); \mult(0,-2)[1,1]; \twisting(1,-1)[1]; \twisting(2,-2)[1]; \map(3,-3)[\scriptstyle \gamma]; \draw (0.5,-3)-- (0.5,-5); \draw (2,-3)-- (2,-5); \laction(3,-4)[0,1];
\end{scope}
\begin{scope}[xshift=4.45cm,yshift=-2.55cm]
\node at (0,-1){=};
\end{scope}
\begin{scope}[xshift=4.9cm, yshift=-0.5cm]
\doublemap(1,-2)[\scriptstyle \mathcal{F}]; \draw (0,-2)-- (0,-3); \draw (3,0)-- (3,-3); \draw (4,0)-- (4,-5); \draw (0,0)-- (0,-1); \mult(0,-3)[1,1]; \twisting(1,0)[1]; \twisting(0,-1)[1]; \twisting(2,-3)[1]; \map(3,-4)[\scriptstyle \gamma]; \draw (0.5,-4)-- (0.5,-6); \draw (2,-1)-- (2,-2); \laction(3,-5)[0,1];  \draw (2,-4)-- (2,-6);
\end{scope}
\begin{scope}[xshift=9.35cm,yshift=-2.55cm]
\node at (0,-1){=};
\end{scope}
\begin{scope}[xshift=9.9cm, yshift=-0.5cm]
\doublemap(1,-2)[\scriptstyle \mathcal{F}]; \draw (0,-2)-- (0,-6); \draw (3,0)-- (3,-3); \draw (4,0)-- (4,-5); \draw (0,0)-- (0,-1); \draw (1,-3)-- (1,-4); \laction(1,-4)[0,1]; \twisting(1,0)[1]; \twisting(0,-1)[1]; \twisting(2,-3)[1]; \map(3,-4)[\scriptstyle \gamma]; \draw (2,-1)-- (2,-2); \laction(3,-5)[0,1];  \draw (2,-5)-- (2,-6);
\end{scope}
\begin{scope}[xshift=14.35cm,yshift=-2.55cm]
\node at (0,-1){=};
\end{scope}
\begin{scope}[xshift=14.9cm, yshift=0cm]
\doublemap(2,-3)[\scriptstyle \mathcal{F}]; \draw (0,-2)-- (0,-7); \draw (3,0)-- (3,-1); \draw (3,-2)-- (3,-3); \draw (4,0)-- (4,-5); \draw (0,0)-- (0,-1); \draw (1,-3)-- (1,-7); \laction(3,-5)[0,1]; \twisting(1,0)[1]; \twisting(0,-1)[1]; \twisting(2,-1)[1]; \twisting(1,-2)[1]; \map(3,-4)[\scriptstyle \gamma]; \map(2,-4)[\scriptstyle \jmath_{\nu}]; \laction(3,-5)[0,1]; \laction(2,-6)[2,1]; \draw (2,-5)-- (2,-6);
\end{scope}
\begin{scope}[xshift=19.1cm,yshift=-2.55cm]
\node at (0,-1){,};
\end{scope}
\end{tikzpicture}
\end{equation*}
where the domain is $V\ot_k V \ot_k A\ot A^{\ot r}\ot M$ and the codomain is $A\ot A^{\ot r}\ot M$. The proof that the left action is unitary follows the same pattern, but using equality~\eqref{preunit3} instead of the twisted module condition.
\end{proof}

\begin{remark}\label{homologia de las filas} Note that $\Ho_r\bigl(\wh{X}_{*s}(M),\wh{d}^0_{*s}\bigr) = \Ho^{\hs K}_r\bigl(A,M \ot_{\hs A} \wt{E}^{\ot_{\hs A}^s}\bigr)$ and that if $\mathcal{F}$ takes its values in $K\ot_k V$, then $\Ho_s\bigl(\wh{X}_{r*}(M),\wh{d}^1_{r*}\bigr) = \Ho^{\hs A}_s\bigl(E,A\ot \ov{A}^{\ot r} \ot M\bigr)$.
\end{remark}

\subsection{Comparison maps}\label{morfismos de comparacion en homologia}
Let $\cramped{\bigl(M\ot\ov{E}^{\ot *}\ot,b_*\bigr)}$ be the normalized Hochschild chain complex of the $K$-algebra $E$ with coe\-fficients in $M$. Recall that there is a canonical identification $\bigl(M\ot\ov{E}^{\ot *}\ot,b_*\bigr)\simeq M\ot_{E^e}\bigl(E\ot\ov{E}^{\ot *}\ot E,b'_*\bigr)$. From now on we let $[m\ot\ov{\bx}_{1n}]$ denote the class of $m\ot\ov{\bx}_{1n}$ in $M\ot\ov{E}^{\ot n}\ot$. Let
$$
\wh{\phi}_*\colon (\wh{X}_*(M),\wh{d}_*)\longrightarrow\bigl(M\ot\ov{E}^{\ot *} \ot,b_*\bigr)\index{zva@$\wh{\phi}_n$|dotfillboldidx} \quad\text{and} \quad \wh{\psi}_*\colon \bigl(M\ot\ov{E}^{\ot *}\ot,b_*\bigr)\longrightarrow (\wh{X}_*(M),\wh{d}_*)\index{zya@$\wh{\psi}_n$|dotfillboldidx}
$$
be the morphisms of complexes induced by $\phi_*$ and $\psi_*$, respectively. By Proposition~\ref{homotopia} we know that $\wh{\psi}_*\xcirc\wh{\phi}_* =\ide_*$ and $\wh{\phi}_*\xcirc\wh{\psi}_*$ is ho\-motopically equivalent to the identity map. More precisely a homotopy $\wh{\omega}_{*+1}\index{zyya@$\wh{\omega}_n$|dotfillboldidx}$, from $\wh{\phi}_*\xcirc\wh{\psi}_*$ to $\ide_*$, is the family of maps
$\bigl(\wh{\omega}_{n+1}\colon M\ot\ov{E}^{\ot n}\ot\longrightarrow M\ot \ov{E}^{\ot {n+1}}\ot\bigr)_{n\ge 0}$, induced by $\cramped{\bigl(\omega_{n+1}\colon E\ot\ov{E}^{\ot n}\ot E\longrightarrow E\ot\ov{E}^{\ot {n+1}}\ot E\bigr)_{n\ge 0}}$.

\subsection[The filtrations of \texorpdfstring{$(\wh{X}_*(M),\wh{d}_*)$}{(X(M),d)} and \texorpdfstring{$\bigl(M\ot\ov{E}^{\ot *}\ot,b_*\bigr)$}{(ME,b)}]{The filtrations of \texorpdfstring{$\bm{(\wh{X}_*(M),\wh{d}_*)}$}{(X(M),d)} and \texorpdfstring{$\bm{\bigl(M\ot\ov{E}^{\ot *}\ot,b_*\bigr)}$}{(ME,b)}}\label{filtraciones en homologia}
The complex $(\wh{X}_*(M),\wh{d}_*)$ is filtered by
$$
0=F^{-1}(\wh{X}_*(M))\subseteq F^0(\wh{X}_*(M))\subseteq F^1(\wh{X}_*(M))\subseteq F^2(\wh{X}_*(M))\subseteq F^3(\wh{X}_*(M))\subseteq \dots,
$$
where $F^i(\wh{X}_n(M)) \coloneqq \bigoplus_{s = 0}^i\wh{X}_{n-s,s}(M)\index{fi@$F^i(\wh{X}_n(M))$|dotfillboldidx}$. Hence, by Remark~\ref{homologia de las filas} there is a spectral sequence
\begin{equation}
E^l_{rs}\Longrightarrow\Ho^{\hs K}_{r+s}(E,M),\qquad\text{with $E^1_{rs} =\Ho^{\hs K}_r\bigr(A,M \ot_{\hs A} \wt{E}^{\ot_{\hs A} s}\bigl)$.}\label{eq9}
\end{equation}
Let $F^i\bigl(M\ot\ov{E}^{\ot n}\ot\bigr)\index{fj@$F^i\bigl(M\ot\ov{E}^{\ot n}\ot\bigr)$|dotfillboldidx}$ be the image of the canonical map $M\ot F^i\bigl(\ov{E}^{\ot n}\bigr)\ot \longrightarrow M\ot\ov{E}^{\ot n}\ot$, where $F^i\bigl(\ov{E}^{\ot n}\bigr)$ is as in Notation~\ref{notation F supra i}. The normalized Hochschild complex $\bigl(M\ot\ov{E}^{\ot *}\ot,b_*\bigr)$ is filtered~by
$$
F^{-1}\bigl(M\ot\ov{E}^{\ot *}\ot\bigr)\subseteq F^0\bigl(M\ot\ov{E}^{\ot *}\ot\bigr)\subseteq F^1\bigl(M\ot\ov{E}^{\ot *}\ot\bigr) \subseteq F^2\bigl(M\ot \ov{E}^{\ot *}\ot\bigr)\subseteq\dots.
$$
The spectral sequence associated with this filtration is called the {\em homological Hochschild-Serre spectral sequence of $E$ with coefficients in $M$}.

\begin{theorem}\label{phi, psi y omega preservan filtraciones en homologia} The maps $\wh{\phi}_*$, $\wh{\psi}_*$ and $\wh{\omega}_*$ preserve filtrations.
\end{theorem}

\begin{proof} This follows immediately from Theorem~\ref{las func phi psi y omega preservan filtraciones}.
\end{proof}

\begin{corollary} The homological Hochschild-Serre spectral sequence of $E$ with coefficients in~$M$ is iso\-morphic to the spectral sequence~\eqref{eq9}.
\end{corollary}

\begin{notation}\label{notation F sub R(M)} For $n,u\ge 0$, $i\ge -1$ and each $k$-subalgebra $R$ of $A$, we let $M\bar{\ot}F_{\! R,u}^i(\ov{E}^{\ot n})\bar{\ot}\index{fk@$M\bar{\ot} F_{R,u}^i(\ov{E}^{\ot n})\bar{\ot}$|dotfillboldidx}$ denote the image of the canonical map $M\ot F_{\! R,u}^i\bigl(\ov{E}^{\ot n} \bigr)\ot \longrightarrow M\ot \ov{E}^{\ot n}\ot$, where $F_{\! R,u}^i\bigl(\ov{E}^{\ot n} \bigr)$ is as in Notation~\ref{notation F sub R}.
\end{notation}

\begin{proposition}\label{propiedad de wh phi} Let $R$ be a $k$-subalgebra of $A$. If $R$ is stable under $\chi$ and $\mathcal{F}$ takes its values in $R\ot_k V$, then
$$
\wh{\phi}\bigl([m\ot_{\hs A}\wt{\gamma}_{\hs A}(\bv_{1i})\ot\ov{\ba}_{1,n-i}]\bigr) = \bigl[m\ot \Sh(\bv_{1i}\ot_k \ov{\ba}_{1,n-i})\bigr] + \bigl[m\ot_{\hs A} \phi''_{in}(\bv_{1i}\ot_k \ov{\ba}_{1,n-i})\bigr],
$$
where $m\in M$, $a_1,\dots,a_{n-i}\in A$, $v_1,\dots,v_i\in V$ and $\phi''_{in}(\bv_{1i}\ot_k \ov{\ba}_{1,n-i})$ is as in Notation~\ref{complemento}.
\end{proposition}

\begin{proof} This is an immediate consequence of Remark~\ref{nota de propiedad de phi}.
\end{proof}

To prove Propositions~\ref{propiedad de hat psi} and~\ref{propiedad de hat omega} we will need some technical results that we establish in an appendix. For $n\ge 0$ and $i\ge -1$ we set $F^i_{\!R}\bigl(\wh{X}_n(M)\bigr) \coloneqq \bigoplus_{s=0}^i \wh{X}^1_{s,n-s}(R,M)\index{fl@$F^i_R\bigl(\wh{X}_n(M)\bigr)$|dotfillboldidx}$.

\begin{proposition}\label{propiedad de hat psi} Let $R$ be a $k$-subalgebra of $A$. Assume that $R$ is stable under $\chi$ and that $\mathcal{F}$ takes its values in $R\ot_k V$. Let $m\in M$, $v, v_1,\dots,v_i\in V$ and $a,a_{i+1},\dots, a_n\in A$. The map $\wh{\psi}_n$ has the following properties:

\begin{enumerate}[itemsep=0.7ex, topsep=1.0ex, label=\emph{(\arabic*)}]

\item $\wh{\psi}([m\ot\ov{\gamma}(\bv_{1i})\ot\ov{\jmath}_{\nu}(\ov{\ba}_{i+1,n})]) = [m\ot_{\hs A} \wt{\gamma}_{\hs A}(\bv_{1i}) \ot \ov{\ba}_{i+1,n}]$.

\item Let $x_1,\dots,x_n\in\jmath_{\nu}(A)\cup \gamma(V)$. If there exist indices $j_1<j_2$ such that $x_{j_1}\in \jmath_{\nu}(A)$ and $x_{j_2}\in\gamma(V)$, then $\wh{\psi}([m\ot\ov{\bx}_{1n}]) = 0$.

\item If $\bx = \bigl[m\ot\ov{\gamma}(\bv_{1,i-1})\ot \ov{a\xcdot \gamma(v)}\ot \ov{\jmath}_{\nu}(\ov{\ba}_{i+1,n})\bigr]$, then
$$
\begin{aligned}
\quad\qquad\wh{\psi}(\bx) &\equiv \bigl[m\ot_{\hs A}\wt{\gamma}_{\hs A}(\bv_{1,i-1}) \ot_{\hs A} a\xcdot\wt{\gamma}(v)\ot \ov{\ba}_{i+1,n}\bigr]\\
& +\sum\bigl[\gamma(v^{(l)})\xcdot m\ot_{\hs A}\wt{\gamma}_{\hs A}(\bv_{1,i-1})\ot \ov{a} \ot\ov{\ba}_{i+1,n}^{(l)}\bigr],
\end{aligned}\quad\mod{F^{i-2}_{\!R}\bigl(\wh{X}_n(M)\bigr)},
$$
where $\sum_l \ov{\ba}_{i+1,n}^{(l)}\ot_k v^{(l)}\coloneqq \ov{\chi}(v\ot_k \ov{\ba}_{i+1,n})$.

\item If $\bx = \bigl[m\ot\ov{\gamma}(\bv_{1,j-1})\ot \ov{a\xcdot\gamma(v)}\ot \ov{\gamma}(\bv_{j+1,i}) \ot\ov{\jmath}_{\nu}(\ba_{i+1,n})\bigr]$ with $j<i$, then
$$
\qquad\qquad\wh{\psi}(\bx)\equiv \bigl[m\ot_{\hs A}\wt{\gamma}_{\hs A}(\bv_{1,j-1}) \ot_{\hs A} a\xcdot\wt{\gamma}(v) \ot_{\hs A}\wt{\gamma}_{\hs A} (\bv_{j+1,i})\ot \ov{\ba}_{i+1,n}\bigr] \quad\mod{F^{i-2}_{\!R}\bigl(\wh{X}_n(M)\bigr)}.
$$

\item If $\bx = \bigl[m\ot\ov{\gamma}(\bv_{1,i-1})\ot\ov{\jmath}_{\nu}(\ba_{i,j-1})\ot \ov{a\xcdot\gamma(v)}\ot \ov{\jmath}_{\nu}(\ba_{j+1,n}) \bigr]$ with $j>i$, then
$$
\qquad\qquad\wh{\psi}(\bx)\equiv\sum \gamma(v^{(l)})\xcdot m\ot_{\hs A} \wt{\gamma}_{\hs A}(\bv_{1,i-1}) \ot \ov{\ba}_{i,j-1}\ot \ov{a}\ot \ov{\ba}_{j+1,n}^{(l)}] \quad \mod{F^{i-2}_{\!R}\bigl(\wh{X}_n(M)\bigr)},
$$
where $\sum_l \ov{\ba}_{j+1,n}^{(l)}\ot_k v^{(l)}\coloneqq  \ov{\chi}(v\ot_k \ov{\ba}_{j+1,n})$.

\item Let $x_1,\dots,x_n\in E$ satisfying $\#\{l:x_l\notin\jmath_{\nu}(A) \cup \gamma(V)\} = 1$. If there exist $j_1 < j_2$ such that $x_{j_1}\in \jmath_{\nu}(A)$ and $x_{j_2}\in \gamma(V)$, then
$$
\qquad \wh{\psi}([m\ot\ov{\bx}_{1n}])\in F^{i-2}_{\!R}\bigl(\wh{X}_n(M)\bigr),
$$
where $i\coloneqq \#\{l:x_l\notin \jmath_{\nu}(A)\}$.

\end{enumerate}

\end{proposition}

\begin{proof} This follows immediately from Proposition~\ref{propA.7}.
\end{proof}

\begin{proposition}\label{propiedad de hat omega} Let $x_1,\dots, x_n\in E$ be such that $\#\{l\!:\!x_l\!\notin\!\jmath_{\nu}(A) \cup \gamma(V)\}\! =\! 1$, let $m\in M$ and let $i\coloneqq \#\{l:x_l\notin \jmath_{\nu}(A)\}$. There exists $\byy \in \jmath_{\nu}(A)\ot F_{\!A,0}^i\bigl(\ov{E}^{\ot {n+1}}\bigr)$ such that $\wh{\omega}([m\ot\ov{\bx}_{1n}]) = [m\ot_{\hs A} \byy]$.
\end{proposition}

\begin{proof} This is an immediate consequence of Proposition~\ref{prop A.9}.
\end{proof}

\section{Hochschild cohomology of general crossed products}\label{section: Hochschild cohomology of general crossed products}
In this section we are in the same conditions as in the previous one. By definition the Hochschild cohomology $\Ho_{\hs K}^*(E,M)\index{hh$@$\Ho_{\hs K}^*(E,M)$|dotfillboldidx}$, of the $K$-algebra $E$ with co\-e\-fficients in an $E$-bimodule $M$, is the cohomology of the normalized Hochschild cochain complex $\cramped{\bigl(\Hom_{K^e}\bigl(\ov{E}^{\ot *},M\bigr),b^*\bigr)}$, where $b^*$ is the canonical Hochschild boundary map. It is well known that $\Ho_{\hs K}^*(E,M)$ is the $\Ext$ functor relative to $\Upsilon$. Since $(X_*,d_*)$ is a $\Upsilon$-relative pro\-jective resolution of $E$, the Hochschild cohomology of $K$-algebra $E$ with co\-e\-fficients in $M$ is the cohomology of the cochain complex $\cramped{\Hom_{E^e}\bigl((X_*,d_*),M\bigr)}$. When $K$ is separable, $\Ho_{\hs K}^*(E,M)$ is the absolute Hochschild cohomology of $E$ with coefficients in $M$. For each $s\!\ge\! 0$, we will denote by $\cramped{\Hom_{\hs A} (\wt{E}^{\ot_{\hs A} s},M)}$ the abelian group of left $A$-linear maps from $\cramped{\wt{E}^{\ot_{\hs A} s}}$ to $M$. As it is well known, $\cramped{\Hom_{\hs A}(\wt{E}^{\ot_{\hs A} s},M)}$ is an $A$-bimodule via $(a\xcdot\alpha)(\wt{\bx}_{1s})\coloneqq \alpha(\wt{\bx}_{1s}\xcdot a)$ and $(\alpha\xcdot a)(\wt{\bx}_{1s})\coloneqq  \alpha(\wt{\bx}_{1s})\xcdot a$. For each $r,s\ge 0$, write
\begin{equation}\label{pepe3}
\wh{X}^{rs}(M)\coloneqq \Hom_{(A,K)}\bigl(\wt{E}^{\ot_{\hs A} s}\ot\ov{A}^{\ot r},M\bigr) \simeq\Hom_{\hs K^e} \bigl(\ov{A}^{\ot r},\Hom_{\hs A} \bigl(\wt{E}^{\ot_{\hs A} s},M\bigr)\bigr)\index{xl@$\wh{X}^{rs}(M)$|dotfillboldidx}.
\end{equation}
Note that, since $\wt{E}^{\ot_{\hs A} 0}=\jmath_{\nu}(A)$ and $\ov{A}^{\ot 0}=K$, we have
\begin{equation}\label{eq20}
\wh{X}^{r0}(M)\simeq\Hom_{\hs K^e}\bigl(\ov{A}^{\ot r},M\bigr)\qquad\text{and}\qquad\wh{X}^{0s}(M)\simeq \Hom_{(A,K)}\bigl(\wt{E}^{\ot_{\hs A} s},M\bigr).
\end{equation}
It is easy to check that the $k$-linear map
$
\zeta^{rs}\colon\Hom_{E^e}(X_{rs},M)\longrightarrow\wh{X}^{rs}(M)\index{zg@$\zeta^{rs}$|dotfillboldidx},
$
given by
$$
\zeta(\alpha)(\wt{\bx}_{1s}\ot\ov{\ba}_{1r})\coloneqq \alpha(1_E\ot_{\hs A}\wt{\bx}_{1s} \ot\ov{\ba}_{1r}\ot 1_E),
$$
is an isomorphism. For each $l\le s$, we let $\wh{d}_l^{rs}\colon\wh{X}^{r+l-1,s-l}(M)\longrightarrow \wh{X}^{rs}(M)\index{df@$\wh{d}_l^{rs}$|dotfillboldidx}$ denote the map induced by $\Hom_{E^e}(d^l_{rs},M)$. Via the above identifications, $\Hom_{E^e}((X_*,d_*),M)$ becomes the complex $(\wh{X}^*(M),\wh{d}^*)$, where
$$
\wh{X}^n(M)\coloneqq \bigoplus_{r+s = n}\wh{X}^{rs}(M)\index{xm@$\wh{X}^n(M)$|dotfillboldidx}\quad\text{and}\quad \wh{d}^n(\alpha)\coloneqq \sum^{r+1}_{l=0} \wh{d}_l^{r-l+1,n-r+l-1} (\alpha)\index{dg@$\wh{d}^n$|dotfillboldidx} \quad\text{for all $\alpha\in \wh{X}^{r,n-r-1}(M)$.}
$$
We have thus proven the following result:

\begin{theorem}\label{Hochschild cohomology} The Hochschild cohomology $\Ho_{\hs K}^*(E,\hs M)$, of the $K$-algebra $E$ with coefficients in $M$, is the co\-homology of $(\wh{X}^*(M),\wh{d}^*)$.
\end{theorem}

\begin{remark}\label{remark para cohomologia: caso F toma valores en K} It follows immediately from Corollary~\ref{caso complej doble} that, if $\mathcal{F}$ takes its values in $K\ot_k V$, then $(\wh{X}^*(M),\wh{d}^*)$ is the total complex of the double complex $(\wh{X}^{**}(M),\wh{d}_0^{**},\wh{d}_1^{**})$.
\end{remark}

\begin{remark}\label{remark para cohomologia: caso K=A} If $K=A$, then $(\wh{X}^*(M),\wh{d}^*) = (\wh{X}^{0*}(M),\wh{d}_1^{0*})$.
\end{remark}

For each $s\ge 0$, we let $\Hom_{\hs A}(E^{\ot_{\hs A} s},M)$ denote the abelian group of left $A$-linear maps from $E^{\ot_{\hs A} s}$ to $M$. Likewise $\Hom_{\hs A}(\wt{E}^{\ot_{\hs A} s},M)$, the group $\Hom_{\hs A}(E^{\ot_{\hs A} s},M)$ is an $A$-bimodule via
$$
(a\xcdot\alpha)(\bx_{1s})\coloneqq \alpha(\bx_{1s}\xcdot a)\qquad\text{and}\qquad (\alpha\xcdot a) (\bx_{1s}) \coloneqq  \alpha(\bx_{1s})\xcdot a,
$$
where $\bx_{1s}\coloneqq x_1\ot_{\hs A}\dots\ot_{\hs A}x_s\in E^{\ot_{\hs A}^s}$. For $r,s\ge 0$, let
$$
X^{rs}(M)\coloneqq \Hom_{(A,K)}\bigl(E^{\ot_{\hs A} s}\ot\ov{A}^{\ot r},M\bigr) \simeq \Hom_{\hs K^e}\bigl(\ov{A}^{\ot r},\Hom_{\hs A}\bigl(E^{\ot_{\hs A} s},M \bigr)\bigr)\index{xp@$X^{rs}(M)$|dotfillboldidx}.
$$
Similarly like for $\wh{X}^{rs}(M)$, we have canonical identifications $X^{rs}(M)\simeq \Hom_{E^e}(X'_{rs},M)$. For $r\ge 0$, $s>0$ and $0\le i\le s$, let
$\mu^i\colon X^{r,s-1}(M)\longrightarrow X^{rs}(M)\index{zmd@$\mu^i$|dotfillboldidx}$ be the map induced by $\mu'_i$. It is easy to see that
$$
\mu^i(\alpha)(\bx_{1s}\ot\ov{\ba}_{1r}) = \begin{cases} x_1\xcdot\alpha\bigl(x_{2s}\ot\ov{\ba}_{1r}\bigr) &\text{if $i=0$,}\\ \alpha\bigl(\bx_{1,i-1}\ot_{\hs A} x_ix_{i+1} \ot_{\hs A} \bx_{i+2,s}\ot \ov{\ba}_{1r}\bigr) & \text{if $0<i< s$,}\\ \sum_l \alpha\bigl(\bx_{1,s-2}\ot_{\hs A} x_{s-1}\jmath_{\nu}(a)\ot \ov{\ba}_{1r}^{(l)}\bigr)\xcdot \gamma\bigl(v^{(l)}\bigr)&\text{if $i=s$,}\end{cases}
$$
where $x_1,\dots,x_{s-1}\in E$, $a_1,\dots,a_r\in A$, $x_s\coloneqq \jmath_{\nu}(a)\gamma(v)$ with $a\in A$ and $v\in V$, and
$$
\sum_l \ov{\ba}_{1r}^{(l)}\ot_k v^{(l)} \coloneqq  \ov{\chi}\bigl(v\ot_k \ov{\ba}_{1r}\bigr).
$$

\begin{notation}\label{notation wh X sub u supra rs } Given a $k$-subalgebra $R$ of $A$ and $0\le u\le r$, we let $T^u_{rs}(M)$ denote the $k$-sub\-mo\-du\-le of $\Hom_{E^e}\bigl(X_{rs},M\bigr)$ consisting of all the $E$-bimodule maps that factorize through the $E$-subbimodule of $X_{r+u,s-u-1}$ generated by all the simple tensors $1_E\ot_{\hs A}\wt{\bx}_{1,s-u-1}\ot\ov{\ba}_{1,r+u}\ot 1_E$, with at least $u$ of the $a_j$'s in $R$. We set $\wh{X}_u^{rs}(R,M)\coloneqq \zeta^{rs}(T^u_{rs}(M))\index{xq@$\wh{X}_u^{rs}(R,M)$|dotfillboldidx}$.
\end{notation}

\begin{theorem}\label{formula para wh{d} 1 en cohomologia} The following assertions hold:

\begin{enumerate}[itemsep=0.7ex, topsep=1.0ex, label=\emph{(\arabic*)}]

\item Via the identifications~\eqref{pepe3}, the morphism $\wh{d}_0\colon\wh{X}^{r-1,s}(M)\to\wh{X}^{rs}(M)$ is $(-1)^s$-times the coboundary map of the normalized cochain Hochschild complex of the $K$-algebra $A$ with coefficients in $\Hom_{\hs A}(\wt{E}^{\ot_{\hs A} s},M)$, considered as an $A$-bimodule as at the beginning of this section.

\item The morphism $\wh{d}_1\colon\wh{X}^{r,s-1}(M)\to\wh{X}^{rs}(M)$ is induced by $\sum_{i=0}^s (-1)^i \mu^i$.

\item Let $R$ be a $k$-subalgebra of $A$. If $R$ is stable under $\chi$ and $\mathcal{F}$ takes its values in $R\ot_k V$, then $\wh{d}_l\bigl(\wh{X}^{r+l-1,s-l}(M)\bigr) \subseteq \wh{X}_{l-1}^{rs}(R,M)$, for each $r\ge 0$ and $1<l\le s$.
\end{enumerate}
\end{theorem}

\begin{proof} Item~(1) follows from the very definition of $d^0$; item~(2), from Theorem~\ref{formula para d^1}; and item~(3), from The\-orem~\ref{formula para d^1b}(3).
\end{proof}

\begin{theorem}\label{formula para wh{d}_2} Under the hypothesis of Theorem~\ref{formula para d^2} the map $\wh{d}_2\colon \wh{X}^{r+1,s-2}(M)\to \wh{X}^{rs}(M)$ is given by
$$
\wh{d}_2(\alpha)(\wt{\gamma}_{\hs A}(\bv_{1s})\ot \ov{\ba}_{1r}) = (-1)^{s+1} \alpha\bigl(\wt{\gamma}_{\hs A}(\bv_{1,s-2}) \ot \mathfrak{T}(v^{(1)}_{s-1}v^{(1)}_s,\ov{\ba}_{1r})\bigr)\xcdot \gamma(v_{s-1}^{(2)}v_s^{(2)}),
$$
where $v_1,\dots,v_s\in H$, $a_1,\dots,a_r\in A$ and $\mathfrak{T}(v_{s-1},v_s,\ov{\ba}_{1r})$ is as in Theorem~\ref{formula para d^2}.
\end{theorem}

\begin{proof} This follows immediately from Theorem~\ref{formula para d^2}.
\end{proof}

\begin{remark}\label{cohomologia de las filas} Note that $\Ho^r\bigl(\wh{X}^{*s}(M),\wh{d}_0^{*s}\bigr) = \Ho_{{\hs K}}^r\bigr(A,\Hom_{{\hs A}} \bigl(\wt{E}^{\ot_{\hs A} s},M\bigr)\bigl)$.
\end{remark}

\subsection{Comparison maps}\label{morfismos de comparacion en cohomologia}
Let $\cramped{\bigl(\Hom_{K^e}\bigl(\ov{E}^{\ot *},M\bigr),b^*\bigr)}$ be the normalized Hochschild cochain complex of the $K$-algebra $E$ with coefficients in $M$. Recall that there is a canonical identification
$$
\bigl(\Hom_{K^e}\bigl(\ov{E}^{\ot *},M\bigr),b^*\bigr)\simeq\Hom_{E^e}\bigl(\bigl(E\ot\ov{E}^{\ot *}\ot E,b'_*\bigr),M\bigr).
$$
Let
$$
\wh{\phi}^*\colon\!\bigl(\Hom_{K^e}\bigl(\ov{E}^{\ot *}\!,M\bigr),b^*\!\bigr)\!\longrightarrow\!\bigl(\wh{X}^*(M),\wh{d}^*\bigr) \index{zvb@$\wh{\phi}^n$|dotfillboldidx} \quad\,\text{and}\,\quad \wh{\psi}^*\!\colon\bigl(\wh{X}^*(M),\wh{d}^*\bigr)\!\longrightarrow\! \bigl(\Hom_{K^e}\bigl(\ov{E}^{\ot *}\!,M\bigr),b^*\!\bigr)\index{zyb@$\wh{\psi}^n$|dotfillboldidx}
$$
be the morphisms of cochain complexes induced by $\phi_*$ and $\psi_*$ respectively. Proposition~\ref{homotopia} implies that $\wh{\phi}^*\xcirc\wh{\psi}^* =\ide^*$ and $\wh{\psi}^*\xcirc\wh{\phi}^*$ is homotopically equivalent to the identity map. An homotopy $\wh{\omega}^{*+1}$ from $\wh{\psi}^*\xcirc\wh{\phi}^*$ to $\ide^*$, is the family of maps
$$
\bigl(\wh{\omega}^{n+1}\colon\Hom_{K^e}\bigl(\ov{E}^{\ot {n+1}},M\bigr)\longrightarrow\Hom_{K^e}\bigl(\ov{E}^{\ot n},M\bigr)\bigr)_{n\ge 0}
\index{zyyb@$\wh{\omega}^n$|dotfillboldidx},
$$
induced by $\cramped{\bigl(\omega_{n+1}\colon E\ot\ov{E}^{\ot n}\ot E\longrightarrow E\ot\ov{E}^{\ot {n+1}}\ot E\bigr)_{n\ge 0}}$.

\subsection[The filtrations of \texorpdfstring{$(\wh{X}^*(M),\wh{d}^*)$}{(X(M),d)} and \texorpdfstring{$\bigl({\Hom}_{K^e}(\ov{E}^{\ot *},M),b^*\bigr)$}{Hom(E,M),b}]{The filtrations of \texorpdfstring{$\bm{(\wh{X}^*(M),\wh{d}^*)}$}{(X(M),d)} and \texorpdfstring{$\bm{\bigl({\Hom}_{K^e}(\ov{E}^{\ot *},M),b^*\bigr)}$}{Hom(E,M),b}}
The complex $(\wh{X}^*(M),\wh{d}^*)$ is filtered by
$$
F_0(\wh{X}^*(M))\supseteq F_1(\wh{X}^*(M))\supseteq F_2(\wh{X}^*(M))\supseteq F_3(\wh{X}^*(M))\supseteq F_4(\wh{X}^*(M))\supseteq\dots,
$$
where $F_i(\wh{X}^n(M)) \coloneqq \bigoplus_{s=i}^n\wh{X}^{n-s,s}(M)\index{fm$@$F_i(\wh{X}^n(M))$|dotfillboldidx}$. Hence, by Remark~\ref{cohomologia de las filas} there is an spectral sequence
\begin{equation}
E_l^{rs}\Longrightarrow\Ho_{\hs K}^{r+s}(E,M),\qquad\text{with $E_1^{rs}=\Ho_{\hs K}^r(A,\Hom_{\hs A} (\wt{E}^{\ot_{\hs A} s},M))$}.\label{eq10}
\end{equation}
We let $\cramped{F_i\bigl({\Hom}_{K^e}\bigl(\ov{E}^{\ot *},M\bigr)\bigr)}\index{fn@$F_i\bigl({\Hom}_{K^e}\bigl(\ov{E}^{\ot *},M\bigr)\bigr)$ |dotfillboldidx}$ denote the $k$-submodule of $\cramped{\Hom_{K^e}\bigl(\ov{E}^{\ot *},M\bigr)}$ consisting of all morphisms $\cramped{\alpha\in {\Hom}_{K^e}\bigl(\ov{E}^{\ot *},M\bigr)}$, such that $\cramped{\alpha\bigl(F^{i-1}\bigl(\ov{E}^{\ot *}\bigr)\bigr) = 0}$, where $\cramped{F^i\bigl(\ov{E}^{\ot n}\bigr)}$ is as in Notation~\ref{notation F supra i}. The normalized Hochschild co\-chain complex $\cramped{\bigl({\Hom}_{K^e}\bigl(\ov{E}^{\ot *},M\bigr), b^*\bigr)}$ is filtered by
\begin{equation}
F_0\bigl(\Hom_{K^e}(\ov{E}^{\ot *},M)\bigr)\supseteq F_1\bigl(\Hom_{K^e}(\ov{E}^{\ot *},M)\bigr)\supseteq F_2\bigl(\Hom_{K^e}(\ov{E}^{\ot *},M)\bigr) \supseteq\dots.\label{eq11}
\end{equation}
The spectral sequence associated to this filtration is called the {\em cohomological Hochschild-Serre spectral sequence of} $E$ {with coefficients in} $M$.

\begin{theorem}\label{phi, psi y omega preservan filtraciones en cohomologia} The maps $\wh{\phi}^*$, $\wh{\psi}^*$ and $\wh{\omega}^*$ preserve filtrations.
\end{theorem}

\begin{proof} This follows immediately from Theorem~\ref{las func phi psi y omega preservan filtraciones}.
\end{proof}

\begin{corollary}\label{suc esp isom en cohom} The cohomological Hochschild-Serre spectral sequence is isomorphic to the spectral sequence~\eqref{eq10}.
\end{corollary}

Recall that the cup product in $\HH_K^*(E)$, of $\beta\in\Hom_{K^e}(\ov{E}^{\ot m},E)$ and $\beta'\in\Hom_{K^e}(\ov{E}^{\ot n},E)$, is given by $(\beta\smile\beta')(\ov{\bx}_{1,m+n}) \coloneqq \beta(\ov{\bx}_{1m})\beta'(\ov{\bx}_{m+1,m+n})\index{zzh@$\smile$|dotfillboldidx}$.

\begin{corollary}\label{es multiplicativa} When $M=E$ the spectral sequence~\eqref{eq10} is multiplicative.
\end{corollary}

\begin{proof} By Corollary~\ref{suc esp isom en cohom} and the fact that the filtration~\eqref{eq11} satisfies $F_i\smile F_j\subseteq F_{i+j}$.
\end{proof}

\begin{proposition}\label{propiedad de wh phi en coh} Let $R$ be a $k$-subalgebra of $A$ and let $\rho_l$ be the left action of $A$ on $M$. If $R$ is stable under $\chi$ and $\mathcal{F}$ takes its values in $R\ot_k V$, then for all $\beta\in\Hom_{K^e}(\ov{E}^{\ot n},M)$,
$$
\wh{\phi}(\beta)\bigl(\wt{\gamma}_{\hs A}(\bv_{1i})\ot\ov{\ba}_{1,n-i}\bigr) = \beta\bigl( \Sh(\bv_{1i}\ot_k \ov{\ba}_{1,n-i})\bigr) + \rho_l\bigl((E\ot \beta)\bigl(\phi''_{in}(\bv_{1i}\ot_k \ov{\ba}_{1,n-i})\bigr)\bigr),
$$
where $a_1,\dots,a_{n-i}\in A$, $v_1,\dots,v_i\in V$ and $\phi''_{in}(\bv_{1i}\ot_k \ov{\ba}_{1,n-i})$ is as in Notation~\ref{complemento}.
\end{proposition}

\begin{proof} This is an immediate consequence of Remark~\ref{nota de propiedad de phi}.
\end{proof}

We will use the following proposition in the next section.

\begin{proposition}\label{propiedad de hat psi en cohomologia} Let $0\le s\le n$ and $\alpha\in \wh{X}^{n-s,s}(M)$. Let $a_1,\dots, a_i\in A$, $v_{i+1},\dots,v_n\in V$ and $x_1,\dots,x_n\in \jmath_{\nu}(A)\cup\gamma(V)$. The map $\wh{\psi}(\alpha)$ have the following properties:

\begin{enumerate}[itemsep=0.7ex, topsep=1.0ex, label=\emph{(\arabic*)}]

\item If $i=s$, then $\wh{\psi}(\alpha)\bigl(\ov{\gamma}(\bv_{1s})\ot\ov{\jmath}_{\nu}(\ba_{s+1,n})\bigr) = \alpha \bigl(\wt{\gamma}_{\hs A} (\bv_{1s})\ot \ov{\ba}_{s+1,n}\bigr)$.

\item If $i\ne s$, then $\wh{\psi}(\alpha)\bigl(\ov{\gamma}(\bv_{1i})\ot\ov{\jmath}_{\hs A}(\ba_{i+1,n})\bigr) = 0$.

\item If there exist $j_1 < j_2$ such that $x_{j_1}\in \jmath_{\nu}(A)$ and $x_{j_2}\in \gamma(V)$, then $\wh{\psi}(\alpha)(\ov{\bx}_{1n})=0$.

\end{enumerate}
\end{proposition}

\begin{proof} This follows immediately from items~(1) and~(2) of Proposition~\ref{propA.7}.
\end{proof}

\section{The cup and cap products for general crossed products}\label{The cup and cap products for general crossed products}
In this section we obtain formulas involving the complexes $(\wh{X}^*(E),\wh{d}^*)$ and $(\wh{X}_*(M),\wh{d}_*)$ that induce the cup product of $\HH_{\hs K}^*(E)$ and the cap product of $\Ho^{\hs K}_*(E,M)$. First of all recall that for $m\le n$, the cap product $\smallfrown\colon\Ho^{\hs K}_n(E,M)\times\HH_{\hs K}^m(E)\longrightarrow \Ho^{\hs K}_{n-m}(E,M)$, is induced by the map
$$
\smallfrown\colon \bigl(M\ot\ov{E}^{\ot n}\ot\bigr)\times \Hom_{{\hs K}^e}(\ov{E}^{\ot m},E)\longrightarrow M\ot\ov{E}^{\ot {n-m}}\ot,
$$
defined by $[m\ot\ov{\bx}_{1n}]\smallfrown\beta \coloneqq  [m\xcdot \beta(\ov{\bx}_{1m})\ot\ov{\bx}_{m+1,n}]\index{zzi@$\smallfrown$|dotfillboldidx}$. When $m>n$ we set $[m\ot\ov{\bx}_{1n}]\smallfrown\beta \coloneqq 0$.

\begin{definition}\label{producto bullet} Let $\alpha\in \wh{X}^{rs}(E)$ and $\alpha'\in \wh{X}^{r's'}(E)$.  We define $\alpha\bullet \alpha'\in \wh{X}^{r'',s''}(E)$, where $r'' \coloneqq r+r'$ and $s'' \coloneqq s+s'$, by
$$
(\alpha\bullet\alpha')\bigl(\wt{\gamma}_{\hs A}(\bv_{1s''})\ot \ov{\ba}_{1r''}\bigr)\coloneqq  \sum (-1)^{rs'}\alpha \bigl(\wt{\gamma}_{\hs A} (\bv_{1s}) \ot\ov{\ba}_{1r}^{(l)}\bigr) \alpha'\bigl(\wt{\gamma}_{\hs A}(\bv_{s+1,s''}^{(l)})\ot \ov{\ba}_{r+1,r''}\bigr)\index{zzj@$\bullet$|dotfillboldidx},
$$
where $v_1,\dots,v_{s''}\in V$, $a_1,\dots, a_{r''}\in A$ and $\sum_l \ov{\ba}_{1r}^{(l)}\ot_k \bv_{s+1,s''}^{(l)} \coloneqq  \ov{\chi}\bigl(\bv_{s+1,s''} \ot_k \ov{\ba}_{1r}\bigr)$.
\end{definition}

\begin{theorem}\label{th cup product} Let $\alpha\!\in\! \wh{X}^{rs}(E)$ and $\alpha'\!\in\! \wh{X}^{r's'}(E)$. Set $s''\coloneqq s+s'$ and $n \coloneqq  r+r'+s+s'$. If $\mathcal{F}$ takes its values in $K\ot_k V$, then, for all $v_1,\dots,v_i\in V$ and $a_{i+1},\dots,a_n\in A$, we have
\begin{equation*}
\wh{\phi}\bigl(\wh{\psi}(\alpha)\smile \wh{\psi}(\alpha')\bigr)\bigl(\wt{\gamma}_{\hs A}(\bv_{1i})\ot \ov{\ba}_{i+1,n}\bigr) = \begin{cases} (\alpha\bullet \alpha')\bigl(\wt{\gamma}_{\hs A}(\bv_{1s''})\ot \ov{\ba}_{s''+1,n}\bigr) &\text{if $i = s''$,}\\ 0 &\text{otherwise,}\end{cases}
\end{equation*}
\end{theorem}

\begin{proof} For the sake of brevity we set $T\coloneqq \Sh(\bv_{1i}\ot_k \ov{\ba}_{i+1,n})$. By Proposition~\ref{propiedad de wh phi en coh},
$$
\wh{\phi}\bigl(\wh{\psi}(\alpha)\smile\wh{\psi}(\alpha')\bigr) \bigl(\wt{\gamma}_{\hs A}(\bv_{1i})\ot \ov{\ba}_{i+1,n}\bigr) = \bigl(\wh{\psi}(\alpha)\smile \wh{\psi}(\alpha') \bigr)(T).
$$
Since, by the de\-finition of $\smile$ and Proposition~\ref{propiedad de hat psi en cohomologia},

\begin{itemize}

\smallskip

\item[-] if $i\ne s''$, then$\bigl(\wh{\psi}(\alpha)\smile \wh{\psi}(\alpha') \bigr)(T) = 0$,

\smallskip

\item[-] if $i = s''$, then
$$
\qquad\qquad\bigl(\wh{\psi}(\alpha)\!\smile\!\wh{\psi}(\alpha')\bigr)(T) = \sum_l (-1)^{s'r}\alpha\bigl( \wt{\gamma}_{\hs A} (\bv_{1s})\ot \ov{\ba}_{i+1,i+r}^{(l)} \bigr) \alpha'\bigl(\wt{\gamma}_{\hs A}(\bv_{s+1,i}^{(l)})\ot\ov{\ba}_{i+r+1,n}\bigr),
$$
where $\sum_l \ov{\ba}_{i+1,i+r}^{(l)}\ot_k \bv_{s+1,i}^{(l)}\coloneqq  \ov{\chi}\bigl(\bv_{s+1,i}\ot_k \ov{\ba}_{i+1,i+r} \bigr)$,

\smallskip

\end{itemize}
the result follows.
\end{proof}

\begin{corollary}\label{cor cup product} If $\mathcal{F}$ takes its values in $K\ot_k V$, then the cup product of $\HH_{\hs K}^*(E)$ is induced by the operation $\bullet$ in $(\wh{X}^*(E),\wh{d}^*)$.
\end{corollary}

\begin{proof} This follows immediately from Theorem~\ref{th cup product}.
\end{proof}

\begin{remark} Let $M$ be and $E$-bimodule. There is an evident generalization of the definition of the cup products that makes up $\Ho_{\hs K}^*(E,M)$ in a $\HH_{\hs K}^*(E)$-bimodule. It is clear that obvious generalizations of Theorem~\ref{th cup product} and Corollary~\ref{cor cup product} hold.
\end{remark}

\begin{definition}\label{accion bullet} Let $\alpha\in \wh{X}^{r's'}(E)$ and let $\bx\coloneqq [m\ot_{\hs A} \wt{\gamma}_{\hs A}(\bv_{1s})\ot \ov{\ba}_{1r}]\in \wh{X}_{rs}(M)$, where $m\in M$, $v_1,\dots,v_s\in V$ and $a_1,\dots,a_r\in A$. If $r'\le r$ and $s'\le s$, then we define $\bx\diamond \alpha\in \wh{X}_{r-r',s-s'}(M)$ by
$$
\bx \diamond\alpha \coloneqq \sum_l (-1)^{r'(s-s')}\bigl[m\xcdot \alpha\bigl(\wt{\gamma}_{\hs A}(\bv_{1s'})\ot \ov{\ba}_{1r'}^{(l)}\bigr) \ot_{\hs A} \wt{\gamma}_{\hs A}(\bv_{s'+1,s}^{(l)})\ot\ov{\ba}_{r'+1,r}\bigr]\index{zzjj@$\diamond$|dotfillboldidx},
$$
where $\sum_l \ov{\ba}_{1r'}^{(l)}\ot_k \bv_{s'+1,s}^{(l)} \coloneqq  \ov{\chi}\bigl(\bv_{s'+1,s} \ot_k \ov{\ba}_{1r'}\bigr)$. If $r'>r$ or $s'>s$, then we set $\bx\diamond\alpha \coloneqq 0$.
\end{definition}

\begin{theorem}\label{th cap product} Let $m\in M$, $v_1,\dots,v_s\in V$, $a_1,\dots,a_r\in A$ and $\alpha\in \wh{X}^{r's'}(E)$. If  $\mathcal{F}$ takes its values in $K\ot_k V$, then $\wh{\psi}\bigl(\wh{\phi}([m\ot_{\hs A}\wt{\gamma}_{\hs A}(\bv_{1s})\ot\ov{\ba}_{1r}]) \smallfrown \wh{\psi}(\alpha)\bigr) = [m\ot_{\hs A}\wt{\gamma}_{\hs A}(\bv_{1s})\ot\ov{\ba}_{1r}] \diamond\alpha$.
\end{theorem}

\begin{proof} By Proposition~\ref{propiedad de wh phi}, we know that $\wh{\psi}\bigl(\wh{\phi}([m\ot_{\hs A}\gamma_{\hs A}(\bv_{1s})\ot\ov{\ba}_{1r}])\smallfrown \wh{\psi}(\alpha)\bigr) = \wh{\psi}\bigl([m\ot T] \smallfrown \wh{\psi}(\alpha)\bigr)$, where $T\coloneqq \Sh(\bv_{1s}\ot_k \ov{\ba}_{1r})$. moreover, by the definition of $\smallfrown$ and Proposition~\ref{propiedad de hat psi en cohomologia}, we have

\begin{itemize}

\smallskip

\item[-] If $s'>s$ or $r'>r$, then $[m\ot T]\smallfrown \wh{\psi}(\alpha) = 0$.

\smallskip

\item[-] If $s'\le s$ and $r'\le r$, then
$$
\quad\qquad [m\ot T]\smallfrown\wh{\psi}(\alpha)= \sum_l (-1)^{r's-r's'} m\xcdot \alpha \bigl(\wt{\gamma}_{\hs A}(\bv_{1s'})\ot \ov{\ba}_{1r'}^{(l)} \bigr)\ot \Sh\bigl(\bv_{s'+1,s}^{(l)}\ot_k \ov{\ba}_{r'+1,r}\bigr),
$$
where $\sum_l \ov{\ba}_{1r'}^{(l)}\ot_k \bv_{s'+1,s}^{(l)}\coloneqq \ov{\chi}(\bv_{s'+1,s}\ot_k \ov{\ba}_{1r'})$.

\smallskip

\end{itemize}
The result follows now immediately from items~(1) and~(2) of Proposition~\ref{propiedad de hat psi}.
\end{proof}

\begin{corollary}\label{cap product caso simple} If $\mathcal{F}$ takes its values in $K\ot_k V$, then in terms of the complexes $(\wh{X}_*(M),\wh{d}_*)$ and $(\wh{X}^*(E),\wh{d}^*)$, the cap product is induced by the operation $\diamond$.
\end{corollary}

\begin{proof} This follows from Theorem~\ref{th cap product}.
\end{proof}

\section{Cyclic homology of a general crossed product}\label{cyclic-Brzezinski}
In this section we construct a mixed complex computing the cyclic homology of $E$, whose underlying Hochschild complex is $(\wh{X}_*(E),\wh{d}_*)$. The notation $\wh{X}_{n+1}(M)$ is not only meaningful in the case when~$M$ is an $E$-bimodule, but also when $M$ is an $A$-bimodule. In this section we will use it with $M = \jmath_{\nu}(A)$.

\begin{notation} For $i\le n$, we let $F_{\! A,0}^{i,1}\bigl(\ov{E}^{\ot n}\bigr)\index{fo@$F_{\hs A,0}^{i,1}\bigl(\ov{E}^{\ot n}\bigr)$|dotfillboldidx}$ denote the $K$-subbimodule of $\ov{E}^{\ot n}$ generated by all the simple tensors $\bx_{1n}$ such that $\# \{j:x_j\notin \jmath_{\nu}(A)\cup \gamma(V)\}\le 1$ and  $\# \{j:x_j\notin \jmath_{\nu}(A)\}\le i$.
Furthermore we let $\cramped{\jmath_{\nu}(K)\bar{\ot} F_{\hs A,0}^{i,1} \bigl(\ov{E}^{\ot n}\bigr)\bar{\ot}}\index{fp@$\jmath_{\nu}(K)\bar{\ot} F_{\hs A,0}^{i,1} \bigl(\ov{E}^{\ot n}\bigr)\bar{\ot}$|dotfillboldidx}$ denote the image of the canonical map $\jmath_{\nu}(K)\ot F_{\! A,0}^{i,1} \bigl(\ov{E}^{\ot n}\bigr)\ot \longrightarrow E\ot \ov{E}^{\ot n}\ot$.
\end{notation}

\begin{lemma}\label{B circ wh{omega} circ B circ wh{phi}=0} Let $\wh{\phi}$ and $\wh{\omega}$ be as in Subsection~\ref{morfismos de comparacion en homologia}. The composition $B\xcirc\wh{\omega}\xcirc B\xcirc \wh{\phi}$, where the map $B_*\colon E\ot\ov{E}^{\ot *}\ot\longrightarrow E\ot\ov{E}^{\ot {*+1}}\ot$ is the Connes operator, is the zero map.
\end{lemma}

\begin{proof} By Proposition~\ref{propiedad de wh phi}, Remark~\ref{imagen de Sh} and the very definition of $B$,
$$
B\xcirc\wh{\phi}\bigl(\wh{X}_{n-i,i}(E)\bigr)\subseteq \jmath_{\nu}(K)\bar{\ot} F_{\hs A,0}^{i+1,1} \bigl(\ov{E}^{\ot {n+1}}\bigr) \bar{\ot}.
$$
So, by Proposition~\ref{propiedad de hat omega} we have
$
\wh{\omega}\xcirc B\xcirc\wh{\phi}\bigl(\wh{X}_{n-i,i}(E)\bigr)\subseteq \jmath_{\nu}(K)\bar{\ot} F_{\hs A,0}^{i+1}\bigl(\ov{E}^{\ot {n+2}}\bigr) \bar{\ot} \subseteq\ker B,
$
as desired.
\end{proof}

For each $n\ge 0$, let $\wh{D}_n\colon\wh{X}_n(E)\to \wh{X}_{n+1}(E)$ be the map $\wh{D}\coloneqq  \wh{\psi}\xcirc B \xcirc \wh{\phi}$.

\begin{theorem}\label{complejo mezclado que da la homologia ciclica} $\bigl(\wh{X}_*(E),\wh{d}_*,\wh{D}_*\bigr)$ is a mixed complex that yields the Hochschild, cyclic, negative and periodic homologies of the $K$-algebra $E$. Moreover we have chain complexes maps
$$
\begin{tikzpicture}[label distance=5mm]
\node {$\Tot\bigl(\BP(\wh{X}_*(E) ,\wh{d}_*,\wh{D}_*)\bigr)$};\draw (6,0) node {$\Tot\bigl(\BP(E\ot\ov{E}^{\ot *}\ot, b_*,B_*)\bigr)$};\draw[<-] (2,0.12) -- node[above=-2pt,font=\scriptsize] {$\wh{\Psi}$}(3.7,0.12);\draw[->] (2,-0.12) -- node[below=-2pt,font=\scriptsize] {$\wh{\Phi}$}(3.7,-0.12);
\end{tikzpicture}
$$
given by $\wh{\Phi}_n(\bx u^i)\coloneqq  \wh{\phi}(\bx)u^i + \wh{\omega}\xcirc B\xcirc\wh{\phi}(\bx) u^{i-1}$ and $\wh{\Psi}_n(\bx u^i)\coloneqq  \sum_{j\ge 0} \wh{\psi}\xcirc (B\xcirc\wh{\omega})^j(\bx) u^{i-j}$. These maps satisfy $\wh{\Psi}\xcirc \wh{\Phi} = \ide$ and $\wh{\Phi}\xcirc\wh{\Psi}$ is homotopically equivalent to the identity map. A homotopy $\wh{\Omega}_{*+1}\colon\wh{\Phi}_* \xcirc\wh{\Psi}_*\to \ide_*$ is given by $\wh{\Omega}_{n+1}(\bx u^i)\coloneqq  \sum_{j\ge 0}\wh{\omega}\xcirc (B\xcirc\wh{\omega})^j(\bx)u^{i-j}$.
\end{theorem}

\begin{proof} For each $i\ge 0$, let
\begin{align*}
& \wh{\phi}u^i\colon\wh{X}_{n-2i}(E)u^i\longrightarrow \bigl(E\ot\ov{E}^{\ot {n-2i}}\ot \bigr)u^i,\\
& \wh{\psi}u^i\colon \bigl(E\ot\ov{E}^{\ot {n-2i}}\ot\bigr)u^i\longrightarrow \wh{X}_{n-2i}(E)u^i
\intertext{and}
& \wh{\omega}u^i\colon \bigl(E\ot\ov{E}^{\ot {n-2i}}\ot \bigr)u^i\longrightarrow \bigl(E\ot \ov{E}^{\ot {n+1-2i}}\ot \bigr)u^i,
\end{align*}
be the maps defined by $\wh{\phi}u^i(\bx u^i)\!\coloneqq \! \wh{\phi}(\bx)\hs u^i$, etcetera. By the results in Subsection~\ref{comparison maps} we have a special de\-for\-mation retract
$$
\begin{tikzpicture}[label distance=5mm]
\node {$\Tot\bigl(\BC(\wh{X}_*(E),\wh{d}_*,0)\bigr)$};\draw (6,0) node {$\Tot\bigl(\BC(E\ot\ov{E}^{\ot *}\ot, b_*,0)\bigr)$};\draw[<-] (1.8,0.12) -- node[above=-2pt,font=\scriptsize] {$\bigoplus_{i\ge 0} \wh{\psi}u^i$} (3.8,0.12);\draw[->] (1.8,-0.12) -- node[below=-2pt,font=\scriptsize]{$\bigoplus_{i\ge 0} \wh{\phi} u^i$}(3.8,-0.12);\draw (10.5,0) node {$\bigoplus_{i\ge 0} \wh{\omega}u^i$.};
\end{tikzpicture}
$$
Applying the perturbation lemma to this datum endowed with the perturbation induced by $B$, and taking into account Lemma~\ref{B circ wh{omega} circ B circ wh{phi}=0}, we obtain the special deformation retract
$$
\begin{tikzpicture}[label distance=5mm]
\node {$\Tot\bigl(\BP(\wh{X}_*(E),\wh{d}_*,\wh{D}_*)\bigr)$};\draw (6,0) node {$\Tot\bigl(\BP(E\ot\ov{E}^{\ot *}\ot, b_*,B_*)\bigr)$};\draw[<-] (2,0.12) -- node[above=-2pt,font=\scriptsize] {$\wh{\Psi}$}(3.7,0.12);\draw[->] (2,-0.12) -- node[below=-2pt,font=\scriptsize] {$\wh{\Phi}$}(3.7,-0.12); \draw (10.25,0) node {$\wh{\Omega}_{*+1}.$};
\end{tikzpicture}
$$
It is easy to see that $\wh{\Phi}$, $\wh{\Psi}$ and $\wh{\Omega}$ commute with the canonical surjections
\begin{align*}
&\Tot\bigl(\BC(\wh{X}_*(E),\wh{d}_*,\wh{D}_*)\bigr)\longrightarrow \Tot\bigl(\BC(\wh{X}_*(E),\wh{d}_*,\wh{D}_*)\bigr)[2]\\
\shortintertext{and}
&\Tot\bigl(\BC(E\ot\ov{E}^{\ot *}\ot,b_*,B_*)\bigr)\longrightarrow \Tot\bigl(\BC(E\ot\ov{E}^{\ot *}\ot,b_*,B_*)\bigr)[2].
\end{align*}
A standard argument, from these facts, finishes the proof.
\end{proof}

\begin{remark} If $K$ is a separable algebra, then the mixed complex $\bigl(\wh{X}_*(E),\wh{d}_*,\wh{D}_*\bigr)$ gives the Hochs\-child, cyclic, negative and periodic absolute homologies of $E$.
\end{remark}

\begin{definition}\label{def wh{D}0, etc} For $r,s\ge 0$, let $\wh{D}^0_{rs}\colon \wh{X}_{rs}\to \wh{X}_{r,s+1}$ and $\wh{D}^1_{rs}\colon \wh{X}_{rs}\to \wh{X}_{r+1,s}$ be the maps defined by

\begin{itemize}[itemsep=0.7ex, topsep=1.0ex]

\item[-] If $\bx = [\jmath_{\nu}(a_0)\gamma(v_0)\ot_{\hs A} \wt{\gamma}_{\hs A}(\bv_{1s})\ot \ov{\ba}_{1r}]$, with $a_0,\dots,a_r\in A$ and $v_0,\dots,v_s\in V$, then
\begin{align*}
\qquad\quad &\wh{D}^0(\bx)\coloneqq  \sum_{j=0}^s\sum_l (-1)^{js+s} \bigl[1_E\ot_{\hs A} \wt{\gamma}_{\hs A}(\bv_{j+1,s}^{(l)})\ot_{\hs A} \stackon[-8pt]{$\jmath_{\nu}(a_0) \gamma(v_0)$}{\vstretch{1.5}{\hstretch{2.8} {\widetilde{\phantom{\;\;\;\;\;\;}}}}}\ot_{\hs A}\wt{\gamma}_{\hs A} (\bv_{1j})\ot \ov{\ba}_{1r}^{(l)}\bigr]
\shortintertext{and}
&\wh{D}^1(\bx)\coloneqq  \sum_{j=0}^r\sum_l (-1)^{jr+r+s}\bigl[\gamma(v_0^{(l)})\ot_{\hs A} \wt{\gamma}_{\hs A}(\bv_{1s}^{(l)}) \ot \ov{\ba}_{j+1,r}\ot \ov{a}_0\ot \ov{\ba}_{1j}^{(l)} \bigr],
\end{align*}
where $\sum_l \ov{\ba}_{1r}^{(l)}\ot_k \bv_{j+1,s}^{(l)}\coloneqq  \ov{\chi}(\bv_{j+1,s} \ot_k \ov{a}_{1r})$ and $\sum_l \ov{\ba}_{1j}^{(l)} \ot_k v_0^{(l)}\ot_k \bv_{1s}^{(l)}\coloneqq  \ov{\chi}(\bv_{0s}\ot_k \ov{\ba}_{1j})$.

\item[-] If $\bx = [\jmath_{\nu}(a_0)\ot_{\hs A}\wt{\gamma}_{\hs A}(\bv_{1s})\ot \ov{\ba}_{1r}]$ with $a_0,\dots,a_r\in A$ and $v_1,\dots,v_s\in V$, then
$$
\qquad\quad \wh{D}^0(\bx)\coloneqq 0 \quad\text{and}\quad \wh{D}^1(\bx)\coloneqq \sum_{j=0}^r\sum_l (-1)^{jr+r+s} \bigl[1_E\ot_{\hs A} \wt{\gamma}_{\hs A}(\bv_{1s}^{(l)})\ot \ov{\ba}_{j+1,r}\ot \ov{a}_0\ot \ov{\ba}_{1j}^{(l)}\bigr],
$$
where $\sum_l \ov{\ba}_{1j}^{(l)} \ot_k \bv_{1s}^{(l)}\coloneqq  \ov{\chi}(\bv_{1s}\ot_k \ov{\ba}_{1j})$.
\end{itemize}

\end{definition}

\begin{proposition}\label{Connes operator} Let $R$ be a $k$-subalgebra of $A$. Assume that $R$ is stable under $\chi$ and $\mathcal{F}$ takes its values in $R\ot_k V$. Let $a_0,\dots,a_{n-i}\in A$ and $v_0,\dots,v_i\in V$. The Connes operator $\wh{D}$ satisfies the following prop\-er\-ties:

\begin{enumerate}[itemsep=0.7ex, topsep=1.0ex, label=\emph{(\arabic*)}]

\item If $\bx = [\jmath_{\nu}(a_0)\gamma(v_0)\ot_{\hs A} \wt{\gamma}_{\hs A}(\bv_{1i})\ot \ov{\ba}_{1,n-i}]$, then
$$
\wh{D}(\bx) = \wh{D}^0(\bx) + \wh{D}^1(\bx)\qquad \mod{F_R^{i-1}\bigl(\wh{X}_{n+1}(E)\bigr) + F_R^i\bigl(\wh{X}_{n+1}(\jmath_{\nu}(A)) \bigr).}
$$

\item If $\bx = [\jmath_{\nu}(a_0)\ot_{\hs A}\wt{\gamma}_{\hs A}(\bv_{1i})\ot \ov{\ba}_{1,n-i}]$, then $\wh{D}(\bx) = \wh{D}^1(\bx) \mod{F_R^{i-1}\bigl(\wh{X}_{n+1}(\jmath_{\nu}(A))\bigr)}$.
\end{enumerate}

\end{proposition}

\begin{proof} (1)\enspace We must compute $\wh{D}(\bx) = \wh{\psi}\xcirc B \xcirc \wh{\phi}(\bx)$. By Proposition~\ref{propiedad de wh phi}
$$
\wh{D}(\bx) = \wh{\psi}\xcirc B\bigl(\bigl[\jmath_{\nu}(a_0)\gamma(v_0)\ot\Sh(\bv_{1i} \ot_k \ov{\ba}_{1,n-i})\bigr]\bigr) + \wh{\psi}\xcirc B \bigl([\jmath_{\nu}(a_0)\gamma(v_0)\ot_{\hs A}\byy]\bigr),
$$
where $[\jmath_{\nu}(a_0)\gamma(v_0)\ot_{\hs A}\byy]\in E\bar{\ot} F_{\! R,1}^{i-1}(\ov{E}^{\ot n})\bar{\ot}$. Moreover

\begin{itemize}

\smallskip

\item[-] $B\bigl(\bigl[\jmath_{\nu}(a_0)\gamma(v_0)\ot\Sh(\bv_{1i}\ot_k \ov{\ba}_{1,n-i})\bigr] \bigr)$ is a sum of circular simple tensors $[1\ot \byy_{1,n+1}]$, with at least $n-i$ of the $y_j$'s in $\jmath_{\nu}(A)$ and at most one $y_j$ in $E\setminus \jmath_{\nu}(A)\cup\gamma(V)$.

\smallskip

\item[-] $B\bigl([\jmath_{\nu}(a_0)\gamma(v_0)\ot_{\hs A}\byy]\bigr)$ is a sum of circular simple tensors $[1\ot\bz_{1,n+1}]$, with at least one $z_j$ in $\jmath_{\nu}(R)$, at least $n-i+1$ of the $z_j$'s in $\jmath_{\nu}(A)$ and at most one $z_j$ in $E\setminus \jmath_{\nu}(A)\cup \gamma(V)$.

\smallskip
\end{itemize}
The result follows now using items~(3)--(6) of Proposition~\ref{propiedad de hat psi} and the definitions of $\Sh$ and $B$.

\smallskip

\noindent (2)\enspace As in the proof of item~(1) we have
$$
\wh{D}(\bx) = \wh{\psi}\xcirc B\bigl([\jmath_{\nu}(a_0)\ot\Sh(\bv_{1i}\ot_k \ov{\ba}_{1,n-i})]\bigr) + \wh{\psi}\xcirc B \bigl([\jmath_{\nu}(a_0)\ot_{\hs A}\byy]\bigr),
$$
where $[\jmath_{\nu}(a_0)\ot_{\hs A} \byy]\in F_{\! R,1}^{i-1}\bigl(E\ot\ov{E}^{\ot n}\ot \bigr)$. Moreover

\begin{itemize}

\smallskip

\item[-] $B\bigl([\jmath_{\nu}(a_0)\ot\Sh(\bv_{1i}\ot_k \ov{\ba}_{1,n-i})]\bigr)$ is a sum of circular simple tensors $[1\ot\byy_{1,n+1}]$, with at least $n-i+1$ of the $y_j$'s in $\jmath_{\nu}(A)$ and each $y_j$'s in $\jmath_{\nu}(A)\cup \gamma(V)$.

\smallskip

\item[-] $B\bigl([\jmath_{\nu}(a_0)\ot_{\hs A}\byy]\bigr)$ is a sum of circular simple tensors $[1\ot\bz_{1,n+1}]$, with at least one $z_j$ in $\jmath_{\nu}(R)$, at least $n-i+2$ of the $z_j$'s in $\jmath_{\nu}(A)$, and each $z_j$ in $\jmath_{\nu}(A)\cup \gamma(V)$,.

\smallskip

\end{itemize}
The result follows now from items~(1)--(2) of Proposition~\ref{propiedad de hat psi} and the definitions~of~$\Sh$~and~$B$.
\end{proof}

\begin{remark}\label{remark: caso F toma valores en K'''} If $\mathcal{F}$ takes its values in $K\ot_k V$, then $\wh{D}= \wh{D}^0 + \wh{D}^1$.
\end{remark}

\begin{remark} If $K=A$, then $\bigl(\wh{X}_*(E),\wh{d}_*,\wh{D}_*\bigr) = \bigl(\wh{X}_{0*}(E),\wh{d}^1_{0*}, \wh{D}^0_{0*}\bigr)$.
\end{remark}

\subsection{The spectral sequences} In this subsection we study two spectral sequences. The first one generalizes those obtained in~\cite{CGG}*{Section 3.1} and~\cite{ZH}*{Theorem 4.7}, while the second one generalizes those obtained in~\cites{AK, KR} and \cite{CGG}*{Section 3.2}. Let $\wh{d}^0_{rs}$ and $\wh{d}^1_{rs}$ be as at the beginning of Section~\ref{section: Hochschild homology of general crossed products} and let $\wh{D}^0_{rs}$ and $\wh{D}^1_{rs}$ be as in Definition~\ref{def wh{D}0, etc}.

\subsubsection{The first spectral sequence} Recall from Remark~\ref{homologia de las filas} that
$\Ho_r\bigl(\wh{X}_{*s},\wh{d}^0_{*s}\bigr) = \Ho^{\hs K}_r\bigr(A,E\ot_{\hs A} \wt{E}^{\ot_{\hs A} s} \bigl)$. Let
\begin{align*}
\breve{d}_{rs}\colon \Ho^{\hs K}_r\bigr(A,E\ot_{\hs A} \wt{E}^{\ot_{\hs A} s} \bigl) \longrightarrow \Ho^{\hs K}_r\bigr(A,E\ot_{\hs A} \wt{E}^{\ot_{\hs A} {s-1}} \bigl)\\
\intertext{and}
\breve{D}_{rs}\colon \Ho^{\hs K}_r\bigr(A,E\ot_{\hs A} \wt{E}^{\ot_{\hs A} s} \bigl)\longrightarrow \Ho^{\hs K}_r\bigr(A,E\ot_{\hs A} \wt{E}^{\ot_{\hs A} {s+1}} \bigl)
\end{align*}
be the maps induced by $\wh{d}^1$ and $\wh{D}^0$, respectively.

\begin{proposition}\label{pepito} For each $r\ge 0$,
$$
\breve{\Ho}^{\hs K}_r\bigr(A,E\ot_{\hs A} \wt{E}^{\ot_{\hs A} *} \bigl) \coloneqq  \bigl(\Ho^{\hs K}_r\bigr(A,E\ot_{\hs A} \wt{E}^{\ot_{\hs A} *}\bigr),\breve{d}_{r*}, \breve{D}_{r*}\bigr)
$$
is a mixed complex. Moreover there is a convergent spectral sequence $(E^v_{sr},\partial^v_{sr})_{v\ge 0} \Longrightarrow \HC^{\hs K}_{r+s}(E)$, such that $E^2_{sr} = \HC_s\bigl(\breve{\Ho}^{\hs K}_r\bigr(A,E\ot_{\hs A} \wt{E}^{\ot_{\hs A} *}\bigl)\bigr)$ for all $r,s\ge 0$.
\end{proposition}

\begin{proof} For each $s,n\ge 0$, let
$
F^s\bigl(\Tot(\BC(\wh{X},\wh{d},\wh{D})_n)\bigr) \coloneqq  \bigoplus_{j\ge 0} F^{s-2j}(\wh{X}_{n-2j}) u^j,
$
where $F^{s-2j}(\wh{X}_{n-2j})$ is the filtration introduced in Subsection~\ref{filtraciones en homologia}. Let $(E^v_{sr},\partial^v_{sr})_{v\ge 0}$ be the spectral sequence associated with the filtration
$$
F^0\bigl(\Tot(\BC(\wh{X}_*,\wh{d}_*,\wh{D}_*))\bigr) \subseteq F^1\bigl(\Tot(\BC(\wh{X}_*,\wh{d}_*,\wh{D}_*))\bigr)  F^2\bigl(\Tot(\BC(\wh{X}_*,\wh{d}_*,\wh{D}_*))\bigr) \subseteq \subseteq \cdots
$$
of $\Tot\bigl(\BC(\wh{X}_*,\wh{d}_*,\wh{D}_*)\bigr)$. A straightforward computation shows that

\begin{itemize}

\smallskip

\item[-] $E^0_{sr} = \bigoplus_{j\ge 0} \wh{X}_{r,s-2j} u^j$,

\smallskip

\item[-] $\partial^0_{sr}\colon E^0_{sr}\longrightarrow E^0_{s,r-1}$ is $\bigoplus_{j\ge 0} \wh{d}^0_{r,s-2j} u^j$,

\smallskip

\item[-] $E^1_{sr} = \bigoplus_{j\ge 0} \Ho_r\bigl(\wh{X}_{*,s-2j}, \wh{d}^0_{*,s-2j}\bigr) u^j$,

\smallskip

\item[-] $\partial^1_{sr}\colon E^1_{sr}\longrightarrow E^1_{s-1,r}$ is $\bigoplus_{j\ge 0} \breve{d}_{r,s-2j}u^j + \bigoplus_{j\ge 1} \breve{D}_{r,s-2j}u^{j-1}$.

\smallskip

\end{itemize}
From this it follows easily that $\breve{\Ho}^{\hs K}_r\bigr(A,E\ot_{\hs A} \wt{E}^{\ot_{\hs A} *}\bigl)$ is a mixed complex and
$$
E^1_{sr} = \bigoplus_{j\ge 0} \Ho^{\hs K}_r\bigr(A,E\ot_{\hs A} \wt{E}^{\ot_{\hs A} {s-2j}} \bigl)u^j \quad\text{and}\quad E^2_{sr} = \HC_s\Bigl(\breve{\Ho}^{\hs K}_r\bigr(A,E\ot_{\hs A} \wt{E}^{\ot_{\hs A} *} \bigl)\Bigr).
$$
In order to finish the proof note that the filtration of $\Tot\bigl(\BC(\wh{X}_*,\wh{d}_*, \wh{D}_*)\bigr)$ introduced above is canonically bounded, and so, by Theorem~\ref{complejo mezclado que da la homologia ciclica}, the spectral sequence $(E^v_{sr},\partial^v_{sr})_{v\ge 0}$ converges to the cyclic homology of the $K$-algebra $E$.
\end{proof}

\subsubsection{The second spectral sequence} Assume that $\mathcal{F}$ takes its values in $K\ot_k V$. Let
\begin{align*}
\check{d}_{rs}\colon \Ho^{\hs A}_s\bigr(E,A\ot\ov{A}^{\ot r}\ot E\bigl) \longrightarrow \Ho^{\hs A}_s\bigr(E,A\ot\ov{A}^{\ot {r-1}}\ot E\bigl)\\
\intertext{and}
\check{D}_{rs}\colon \Ho^{\hs A}_s\bigr(E,A\ot\ov{A}^{\ot r}\ot E\bigl)\longrightarrow \Ho^{\hs A}_s \bigr(E,A\ot\ov{A}^{\ot {r+1}}\ot E\bigl)
\end{align*}
be the maps induced by $\wh{d}^0$ and $\wh{D}^1$, respectively.

\begin{proposition}\label{pepitos} For each $s\ge 0$,
\[
\check{\Ho}^{\hs A}_s\bigr(E,A\ot\ov{A}^{\ot *}\ot E\bigl)\coloneqq  \Bigl(\Ho^{\hs A}_s\bigr(E,A\ot\ov{A}^{\ot *} \ot E\bigl), \check{d}_{*s},\check{D}_{*s}\Bigr)
\]
is a mixed complex. Moreover there is a convergent spectral sequence $(\mathfrak{E}^v_{rs},\mathfrak{d}^v_{rs})_{v\ge 0} \Longrightarrow \HC^{K}_{r+s}(E)$, such that $\mathfrak{E}^2_{rs} = \HC_r\bigl(\check{\Ho}^{\hs A}_s\bigr(E,A\ot\ov{A}^{\ot *}\ot E\bigl)\Bigr)$ for all $r,s\ge 0$.
\end{proposition}

\begin{proof} For each $r,n\ge 0$, let
$$
\mathfrak{F}^r\bigl(\Tot(\BC(\wh{X},\wh{d},\wh{D})_n)\bigr) \coloneqq  \bigoplus_{j\ge 0} \mathfrak{F}^{r-2j}(\wh{X}_{n-2j}) u^j,
$$
where $\mathfrak{F}^{r-2j}(\wh{X}_{n-2j})\coloneqq  \bigoplus_{i\le r-2j} \wh{X}_{i,n-i-2j}$. Consider the spectral sequence $(\mathfrak{E}^v_{rs},\mathfrak{d}^v_{rs})_{v\ge 0}$, associated with the filtration
$$
\mathfrak{F}^0\bigl(\Tot(\BC(\wh{X}_*,\wh{d}_*,\wh{D}_*))\bigr) \subseteq \mathfrak{F}^1\bigl(\Tot(\BC(\wh{X}_*,\wh{d}_*,\wh{D}_*))\bigr) \subseteq \mathfrak{F}^2\bigl(\Tot(\BC(\wh{X}_*,\wh{d}_*,\wh{D}_*))\bigr) \subseteq \cdots
$$
of $\Tot\bigl(\BC(\wh{X}_*,\wh{d}_*,\wh{D}_*)\bigr)$. A straightforward computation shows that

\begin{itemize}

\smallskip

\item[-] $\mathfrak{E}^0_{rs} = \bigoplus_{j\ge 0} \wh{X}_{r-2j,s} u^j$,

\smallskip

\item[-] $\mathfrak{d}^0_{rs}\colon \mathfrak{E}^0_{rs}\longrightarrow \mathfrak{E}^0_{r,s-1}$ is $\bigoplus_{j\ge 0} \wh{d}^1_{r-2j,s} u^j$,

\smallskip

\item[-] $\mathfrak{E}^1_{rs} = \bigoplus_{j\ge 0} \Ho_s\bigl(\wh{X}_{r-2j,*}, \wh{d}^1_{r-2j,*}\bigr) u^j$,

\smallskip

\item[-] $\mathfrak{d}^1_{rs}\colon \mathfrak{E}^1_{rs}\longrightarrow \mathfrak{E}^1_{r-1,s}$ is $\bigoplus_{j\ge 0} \check{d}_{r-2j,s} u^j + \bigoplus_{j\ge 1} \check{D}_{r-2j,s} u^{s-j}$.

\smallskip

\end{itemize}
From this and Remark~\ref{homologia de las filas} it follows that $\check{\Ho}^{\hs A}_s\bigr(E,A\ot \ov{A}^{\ot *}\ot E\bigl)$ is a mixed complex,
$$
\mathfrak{E}^1_{rs} = \bigoplus_{j\ge 0}\Ho^{A}_{s-j}\bigr(E,A\ot\ov{A}^{\ot {r-j}} \ot E\bigl) \quad\text{and}\quad  \mathfrak{E}^2_{rs} = \HC_r\Bigl(\check{\Ho}^{\hs A}_s\bigr(E,A\ot\ov{A}^{\ot *}\ot E\bigl)\Bigr).
$$
In order to finish the proof note that the filtration of $\Tot\bigl(\BC(\wh{X}_*,\wh{d}_*, \wh{D}_*)\bigr)$ introduced above is canonically bounded, and so, by Theorem~\ref{complejo mezclado que da la homologia ciclica}, the spectral sequence $(\mathfrak{E}^v_{sr}, \mathfrak{d}^v_{sr})_{v\ge 0}$ converges to the cyclic homology of the $K$-algebra $E$.
\end{proof}

\section*{Appendix}

\appendix

\setcounter{secnumdepth}{0}
\setcounter{section}{1}
\setcounter{theorem}{0}
\renewcommand\thesection{\Alph{section}}
\renewcommand\sectionmark[1]{}

For each $-1\le i< n$ and each $k$-subalgebra $R$ of $A$ we set $F_{U,R}^i(X_n)\coloneqq \bigoplus_{s=0}^i U^1_{n-s,s}(R)\index{fq@$F_{U,R}^i(X_n)$|dotfillboldidx}$, where $U^1_{n-s,s}(R)$ is as in Notation~\ref{notacion X supra u sub rs (R)}. Note that $F_{U,R}^i(X_n) = 0$.

\begin{lemma}\label{lemma A6} Let $R$ be a $k$-subalgebra of $A$. Assume that $R$ is stable under $\chi$ and that $\mathcal{F}$ takes its values in $R\ot_k V$. Let $r,s\ge 0$ and let $n\coloneqq r+s$. For each $\bz\in X_{rs}$ it is true that:

\begin{enumerate}[itemsep=0.7ex, topsep=1.0ex, label=\emph{(\arabic*)}]

\item If $\bz\in E\xcdot W_{rs}$, then $\ov{\sigma}(\bz)= \sigma^0(\bz) \in E\xcdot L_{r+1,s}$.

\item If $\bz\in L^u_{rs}(R)\xcdot E$, then $\sigma^l(\bz)\in U^{l+u}_{r+l+1,s-l}(R)$ for all $0\le l\le s$.

\item If $\bz\in E\xcdot U_{rs}$, then $\sigma^l(\bz)=0$ for $0\le l\le s$.

\item If $\bz\in L_{rs}\xcdot E$ and $r>0$, then $\ov{\sigma}(\bz)\equiv\sigma^0(\bz)$ modulo $F_{U,R}^{s-1}(X_{n+1})$.

\item If $\bz\in L_{0n}\xcdot E$, then $\ov{\sigma}(\bz)\equiv\sigma^0(\bz)- \sigma^0\xcirc\sigma^{-1}\xcirc\upsilon(\bz)$ modulo $F_{U,R}^{n-1}(X_{n+1})$.

\item If $\bz\in E\xcdot U_{rs}$ and $r>0$, then $\ov{\sigma}(\bz)=0$.

\item If $\bz\in E\xcdot U_{0n}$, then $\ov{\sigma}(\bz) =-\sigma^0\xcirc\sigma^{-1}\xcirc\upsilon(\bz)$.

\end{enumerate}
\end{lemma}

\begin{proof} Item~1) is Remark~\ref{ov sigma reducido} and item~(2) holds for $l=0$ by the very definition of $\sigma^0$. Assume that item~(2) is true for all the maps $\sigma^i$ with $i<l$. By items~(2) and~(3) of Theorem~\ref{formula para d^1b},
\begin{align*}
\sigma^l(\bz)& =-\sum_{i=0}^{l-1}\sigma^0\xcirc d^{l-i}\xcirc\sigma^i(\bz)\\
&\in\sum_{i=0}^{l-1}\sigma^0\bigl(d^{l-i}(U^{u+i}_{r+i+1,s-i}(R))\bigr)\\
& \subseteq \sigma^0(E\xcdot U^{u+l-1}_{r+l,s-l}(R))+\sigma^0(L^{u+l-1}_{r+l,s-l}(R)\xcdot \gamma(V)\gamma(V)).
\end{align*}
So, $\sigma^l(\bz)\in U^{u+l}_{r+l+1,s-l}(R)$, since by the definition of $\sigma^0$, equality~\eqref{ec2} and Lemma~\ref{lema gamma(V)gamma(V)},
$$
\sigma^0(E\xcdot U^{u+l}_{r+l,s-l}(R)) = 0\qquad\text{and}\qquad \sigma^0(L^{u+l-1}_{r+l,s-l}(R)\xcdot \gamma(V)\gamma(V))\subseteq U^{u+l}_{r+l+1,s-l}(R).
$$
Item~(3) follows easily by induction on $l$ using the recursive definition of $\sigma^l$ (the initial case $l=0$ is equality~\eqref{ec2}). Items~(4), (5), (6) and~(7) follow from Proposition~\ref{prim prop of ov sigma} and items~(2) and~(3).
\end{proof}

\begin{remark}\label{Prev a A7} From the definition of $\ov{\sigma}$ and item~(2) of the previous lemma it follows that
$$
\ov{\sigma}\bigl(F_{\!U,R}^i(X_n)\xcdot E\bigr)\subseteq F_{\!U,R}^i(X_{n+1})\qquad\text{for $i<n$.}
$$
\end{remark}

\begin{proposition}\label{propA.7} Let $R$ be a $k$-subalgebra of $A$. Assume that $R$ is stable under $\chi$ and that $\mathcal{F}$ takes its values in $R\ot_k V$. Let $v, v_1,\dots,v_i \in V$ and $a,a_{i+1},\dots,a_n \in A$. The map $\psi_n$ has the following properties:

\begin{enumerate}

\smallskip

\item $\psi\bigl(1_E\ot\ov{\gamma}(\mathrm{v}_{1i})\ot\ov{\jmath}_{\nu}(\ba_{i+1,n})\ot 1_E\bigr) = 1_E\ot_{\hs A} \wt{\gamma}_{\hs A}(\bv_{1i}) \ot \ov{\ba}_{i+1,n}\ot 1_E$.

\smallskip

\item Let $x_1,\dots,x_n\in\jmath_{\nu}(A)\cup \gamma(V)$. If there exist indices $j_1<j_2$ such that $x_{j_1}\in \jmath_{\nu}(A)$ and $x_{j_2}\in\gamma(V)$, then $\psi(1_E\ot \ov{\bx}_{1n}\ot 1_E) = 0$.

\smallskip

\item If $\bx = 1_E\ot\ov{\gamma}(\bv_{1,i-1})\ot\ov{a\xcdot\gamma(v)}\ot \ov{\jmath}_{\nu}(\ba_{i+1,n}) \ot 1_E$, then
$$
\begin{aligned}
\qquad\psi(\bx) &\equiv 1_E\ot_{\hs A}\wt{\gamma}_{\hs A}(\bv_{1,i-1})\ot_{\hs A} a\xcdot\wt{\gamma}(v)\ot\ov{\ba}_{i+1,n} \ot 1_E\\
& + \sum 1_E\ot_{\hs A}\wt{\gamma}_{\hs A}(\bv_{1,i-1})\ot\ov{a}\ot\ov{\ba}_{i+1,n}^{(l)}\ot \gamma\bigl(v^{(l)}\bigr)
\end{aligned}\quad\mod{F_{\!U,R}^{i-2}(X_n)},
$$
where $\sum_{l} \ov{\ba}_{i+1,n}^{(l)}\ot_k v^{(l)}\coloneqq \ov{\chi}(v\ot_k \ov{\ba}_{i+1,n})$.

\smallskip

\item If $\bx = 1_E\ot\ov{\gamma}(\bv_{1,j-1})\ot \ov{a\xcdot\gamma(v)}\ot \ov{\gamma}(\bv_{j+1,i})\ot \ov{\jmath}_{\nu}(\ba_{i+1,n})\ot 1_E$ with $j<i$, then
$$
\qquad\qquad\psi(\bx)\equiv 1_E\ot_{\hs A}\wt{\gamma}_{\hs A}(\bv_{1,j-1})\ot_{\hs A} a\xcdot \wt{\gamma}(v) \ot_{\hs A} \wt{\gamma}_{\hs A} (\bv_{j+1,i})\ot \ov{\ba}_{i+1,n}\ot 1_E\quad \mod{F_{\!U,R}^{i-2}(X_n)}.
$$

\smallskip

\item If $\bx = 1_E\ot\ov{\gamma}(\bv_{1,i-1})\ot\ov{\jmath}_{\nu}(\ba_{i,j-1})\ot \ov{a\xcdot \gamma(v)}\ot \ov{\jmath}_{\nu}(\ba_{j+1,n})\ot 1_E$ with $j>i$, then
$$
\quad\qquad\psi(\bx)\equiv\sum 1_E\ot_{\hs A}\wt{\gamma}_{\hs A}(\bv_{1,i-1})\ot \ov{\ba}_{i,j-1}\ot \ov{a}\ot \ov{\ba}_{j+1,n}^{(l)}\ot \gamma(v^{(l)})\quad \mod{F_{\!U,R}^{i-2}(X_n)},
$$
where $\sum_l \ov{\ba}_{j+1,n}^{(l)}\ot_k v^{(l)}\coloneqq \ov{\chi}(v\ot_k \ov{\ba}_{j+1,n})$.

\smallskip

\item Let $x_1,\dots,x_n\in E$ satisfying $\#\{l:x_l\notin \jmath_{\nu}(A) \cup\gamma(V)\} = 1$. If there exist $j_1 < j_2$ such that $x_{j_1}\in \jmath_{\nu}(A)$ and $x_{j_2} \in \gamma(V)$, then
$$
\psi(1_E\ot\ov{\bx}_{1n}\ot 1_E)\in F_{\!U,R}^{i-2}(X_n),
$$
where $i\coloneqq \# \{l:x_l\notin\jmath_{\nu}(A)\}$.

\end{enumerate}

\end{proposition}

\begin{proof} (1) We proceed by induction on $n$. The case $n=0$ is trivial. Suppose $n>0$ and the result is valid for $n-1$. Assume first that $i<n$. By Proposition~\ref{formula para psi_n(y ot 1)}, the inductive hypothesis and item~(1) of Lemma~\ref{lemma A6},
\begin{align*}
\psi\bigl(1_E\ot\ov{\gamma}(\bv_{1i})\ot\ov{\jmath}_{\nu}(\ba_{i+1,n})\ot 1_E\bigr) & = (-1)^n\ov{\sigma}\xcirc \psi\bigl(1_E\ot\ov{\gamma}(\bv_{1i})\ot \ov{\jmath}_{\nu}(\ba_{i+1,n-1}) \ot\jmath_{\nu}(a_n)\bigr)\\
& = (-1)^n\ov{\sigma}\bigl(1_E\ot_{\hs A}\wt{\gamma}_{\hs A}(\bv_{1i})\ot\ov{\ba}_{i+1,n-1} \ot\jmath_{\nu}(a_n) \bigr)\\
& = (-1)^n \sigma^0\bigl(1_E\ot_{\hs A}\wt{\gamma}_{\hs A}(\bv_{1i})\ot\ov{\ba}_{i+1,n-1} \ot\jmath_{\nu}(a_n) \bigr),
\end{align*}
and the result follows from equality~\eqref{segunda cond}, since $1_E\!\in\! K\ot_k V$. Assume now that $i=n$. By Pro\-position~\ref{formula para psi_n(y ot 1)}, the inductive hypothesis, item~(7) of Lemma~\ref{lemma A6} and the definitions of $\upsilon$ and~$\sigma^{-1}$,
\begin{align*}
\psi\bigl(1_E\ot\ov{\gamma}(\bv_{1n})\ot 1_E\bigr) &= (-1)^n\ov{\sigma}\xcirc\psi \bigl(1_E\ot \ov{\gamma}(\bv_{1,n-1}) \ot \gamma(v_n)\bigr)\\
& = (-1)^{n+1}\sigma^0\xcirc\sigma^{-1}\xcirc\upsilon\bigl(1_E\ot_{\hs A} \wt{\gamma}_{\hs A}(\bv_{1,n-1}) \ot\gamma(v_n)\bigr)\\
%
%
& =\sigma^0\bigl(1_E\ot_{\hs A}\wt{\gamma}_{\hs A}(\bv_{1n})\ot_{\hs A} 1_E\bigr).
\end{align*}
The result follows now immediately from equality~\eqref{primera cond}, since $1_E\in K\ot_k V$.

\smallskip

\noindent (2) We proceed by induction on $n$. The case $n=0$ is trivial. Suppose $n>0$ and the result is valid for $n-1$. By Proposition~\ref{formula para psi_n(y ot 1)} and the inductive hypothesis, if there exist $j_1<j_2<n$ such that $x_{j_1}\in\jmath_{\nu}(A)$ and $x_{j_2}\in\gamma(V)$, then
$$
\psi(1_E\ot\ov{\bx}_{1n}\ot 1_E)= (-1)^n\ov{\sigma}\xcirc\psi(1_E\ot\ov{\bx}_{1,n-1}\ot x_n) = (-1)^n\ov{\sigma}(0) = 0.
$$
On the other hand, if $\bx_{1n}\! =\! \gamma(\bv_{1,i-1})\ot \jmath_{\nu}(\ba_{i,n-1})\ot \gamma(v_n)$, then, by Proposition~\ref{formula para psi_n(y ot 1)} and s\-ta\-te\-ment~(1),
$$
\psi(1_E\ot\ov{\bx}_{1n}\ot 1_E) = (-1)^n\ov{\sigma}\xcirc\psi(1_E\ot\ov{\bx}_{1,n-1}\ot x_n) = (-1)^n\ov{\sigma} \bigl(1_E\ot_{\hs A} \wt{\gamma}_{\hs A}(\bv_{1,i-1}) \ot\ov{\ba}_{i,n-1}\ot\gamma(v_n)\bigr),
$$
and the result follows from item~(6) of Lemma~\ref{lemma A6}.

\smallskip

\noindent (3) We proceed by induction on $n$. The case $n=0$ is trivial. Suppose $n>0$ and the result is valid for $n-1$. Assume first that $i<n$. Let
\begin{align*}
\byy & \coloneqq  1_E\ot\ov{\gamma}(\bv_{1,i-1})\ot \ov{a\xcdot\gamma(v)}\ot \ov{\jmath}_{\nu}(\ba_{i+1,n-1}) \ot \jmath_{\nu}(a_n),\\
\bz & \coloneqq  1_E\ot_{\hs A}\wt{\gamma}_{\hs A}(\bv_{1,i-1})\ot_{\hs A} a\xcdot\wt{\gamma}(v) \ot \ov{\ba}_{i+1,n-1} \ot \jmath_{\nu}(a_n)
\shortintertext{and}
\bz' &\coloneqq \sum 1_E\ot_{\hs A}\wt{\gamma}_{\hs A}(\bv_{1,i-1})\ot\ov{a}\ot\ov{\ba}_{i+1,n-1}^{(l)} \ot \gamma \bigl(v^{(l)}\bigr) \jmath_{\nu}(a_n),
\end{align*}
where $\sum_l \ov{\ba}_{i+1,n-1}^{(l)}\ot_k v^{(l)} \coloneqq  \ov{\chi}(v\ot_k \ov{\ba}_{i+1,n-1})$. Clearly $\bz'\in L_{n-i,i-1}\xcdot E$, while $\bz \in W_{n-i-1,i}$ (use equality~\eqref{eq12} repeatedly). By Proposition~\ref{formula para psi_n(y ot 1)} and the inductive hypothesis,
$$
\psi(\bx) = (-1)^n\ov{\sigma}\xcirc\psi(\byy) \equiv (-1)^n\ov{\sigma}(\bz+\bz')\quad\mod{\ov{\sigma} \bigl(F^{i-2}_{U,R}(X_{n-1})\xcdot \jmath_{\nu}(A)\bigr)}
$$
So, by Remark~\ref{Prev a A7} and items~(1) and~(4) of Lemma~\ref{lemma A6}, $\psi(\bx)\equiv (-1)^n\sigma^0(\bz+\bz') \mod{F^{i-2}_{U,R}(X_n)}$.
Using the fact that $\gamma(V)\subseteq K\ot_k V$ and equalities~\eqref{eq12} and~\eqref{segunda cond} in order to compute $\sigma^0(\bz')$, and using the fact that $1_E\in K\ot_k V$ and equality~\eqref{segunda cond} in order to compute $\sigma^0(\bz)$, we obtain the desired expression for $\psi(\bx)$. Assume now that $i=n$. Let
\begin{align*}
\byy &\coloneqq  1_E\ot\ov{\gamma}(\bv_{1,n-1})\ot a\xcdot\gamma(v),\\
\bz &\coloneqq  1_E\ot_{\hs A}\wt{\gamma}_{\hs A}(\bv_{1,n-1})\ot a\xcdot\gamma(v)
\shortintertext{and}
\bz' &\coloneqq  1_E\ot_{\hs A}\wt{\gamma}_{\hs A}(\bv_{1,n-1})\ot_{\hs A} a\xcdot\wt{\gamma}(v)\ot_{\hs A} 1_E.
\end{align*}
By Proposition~\ref{formula para psi_n(y ot 1)}, item~(1) and Lemma~\ref{lemma A6}(5),
$$
\psi(\bx) = (-1)^n\ov{\sigma}\xcirc\psi(\byy) = (-1)^n\ov{\sigma}(\bz)\equiv (-1)^n\sigma^0(\bz)-(-1)^n\sigma^0\xcirc \sigma^{-1}\xcirc\upsilon(\bz) \quad\mod{F^{n-2}_{U,R}(X_n)}.
$$
Hence, by the definitions of $\upsilon$ and $\sigma^{-1}$, $\psi(\bx) \equiv (-1)^n\sigma^0(\bz)+\sigma^0(\bz') \mod{F^{n-2}_{U,R}(X_n)}$. The desired formula for $\psi(\bx)$ follows now easily using the fact that $\gamma(V)\subseteq K\ot_k V$ and equality~\eqref{segunda cond} in order to compute $\sigma^0(\bz)$, and using the fact that $1_E\in K\ot_k V$ and equality~\eqref{primera cond} in order to compute $\sigma^0(\bz')$.

\smallskip

\noindent (4) We proceed by induction on $n$. The case $n=0$ is trivial. Suppose $n>0$ and the result is valid for $n-1$. Assume first that $i<n$. Let
\begin{align*}
& \byy \coloneqq  1_E\ot\ov{\gamma}(\bv_{1,j-1})\ot \ov{a\xcdot\gamma(v)}\ot \ov{\gamma}(\bv_{j-1,i})\ot \ov{\jmath}_{\nu}(\ba_{i+1,n-1}) \ot\jmath_{\nu}(a_n)\\
\shortintertext{and}
& \bz \coloneqq  1_E\ot_{\hs A}\wt{\gamma}_{\hs A}(\bv_{1,j-1})\ot_{\hs A} a\xcdot\wt{\gamma}(v) \ot_{\hs A} \wt{\gamma}_{\hs A} (\bv_{j-1,i}) \ot\ov{\ba}_{i+1,n-1}\ot\jmath_{\nu}(a_n).
\end{align*}
Note that $\bz \in W_{n-i-1,i}$. By Proposition~\ref{formula para psi_n(y ot 1)} and the inductive hypothesis,
$$
\psi(\bx) = (-1)^n\ov{\sigma}\xcirc\psi(\byy)\equiv (-1)^n\ov{\sigma}(\bz)\quad\mod{\ov{\sigma}\bigl(F^{i-2}_{U,R}(X_{n-1})\xcdot \jmath_{\nu}(A)\bigr)}.
$$
So, by Remark~\ref{Prev a A7} and item~(1) of Lemma~\ref{lemma A6}, $\psi(\bx)\equiv (-1)^n\sigma^0(\bz) \mod{F^{i-2}_{U,R}(X_n)}$. The desired formula for $\psi(\bx)$ follows from equality~\eqref{segunda cond}, because $1_E\in K\ot_k V$. Assume now that $j<i-1$ and $i=n$. Let
\begin{align*}
&\byy \coloneqq  1_E\ot\ov{\gamma}(\bv_{1,j-1})\ot \ov{a\xcdot\gamma(v)}\ot \ov{\gamma}(\bv_{j+1,n-1})\ot \gamma(v_n),\\
&\bz \coloneqq  1_E\ot_{\hs A}\wt{\gamma}_{\hs A}(\bv_{1,j-1})\ot_{\hs A} a\xcdot\wt{\gamma}(v) \ot_{\hs A} \wt{\gamma}_{\hs A}(\bv_{j+1,n-1}) \ot\gamma(v_n)
\shortintertext{and}
&\bz' \coloneqq  1_E\ot_{\hs A}\wt{\gamma}_{\hs A}(\bv_{1,j-1})\ot_{\hs A} a\xcdot\wt{\gamma}(v) \ot_{\hs A} \wt{\gamma}_{\hs A}(\bv_{j+1,n}) \ot_{\hs A} 1_E.
\end{align*}
Clearly $\bz \in U_{0,n-1}$. By Proposition~\ref{formula para psi_n(y ot 1)} and the inductive hypothesis,
$$
\psi(\bx) = (-1)^n\ov{\sigma}\xcirc\psi(\byy)\equiv (-1)^n\ov{\sigma}(\bz)\quad \mod{\ov{\sigma}\bigl(F^{n-3}_{U,R}(X_{n-1})\xcdot E\bigr)}.
$$
Consequently, by Remark~\ref{Prev a A7} and item~(7) of Lemma~\ref{lemma A6}, $\psi(\bx)\equiv (-1)^{n+1}\sigma^0\xcirc\sigma^{-1}\xcirc \upsilon(\bz) \mod{F^{n-3}_{U,R}(X_n)}$, So, by the definitions of $\upsilon$ and $\sigma^{-1}$, $\psi(\bx)\equiv\sigma^0(\bz') \mod{F^{n-3}_{U,R}(X_n)}$. Thus, the formula for $\psi(\bx)$ follows from equality~\eqref{primera cond}, because $1_E\in K\ot_k V$. Assume finally that $j = i-1$ and $i=n$. Let
\begin{align*}
&\byy\coloneqq  1_E\ot\ov{\gamma}(\bv_{1,n-2})\ot \ov{a\xcdot\gamma(v)}\ot \gamma(v_n),\\
&\bz \coloneqq  1_E\ot_{\hs A}\wt{\gamma}_{\hs A}(\bv_{1,n-2})\ot_{\hs A} a\xcdot\wt{\gamma}(v) \ot\gamma(v_n),\\
&\bz' \coloneqq  1_E\ot_{\hs A}\wt{\gamma}_{\hs A}(\bv_{1,n-2})\ot \ov{a}\ot\gamma(v) \gamma(v_n)
\shortintertext{and}
&\bz'' \coloneqq  1_E\ot_{\hs A}\wt{\gamma}_{\hs A}(\bv_{1,n-2})\ot_{\hs A} a\xcdot\wt{\gamma}(v) \ot_{\hs A}\wt{\gamma}(v_n)\ot_{\hs A} 1_E.
\end{align*}
Clearly $\bz'\in L_{1,n-2}\xcdot E$ and $\bz \in U_{0,n-1}$. By Proposition~\ref{formula para psi_n(y ot 1)} and item~(3),
$$
\psi(\bx)= (-1)^n\ov{\sigma}\xcirc\psi(\byy) \equiv (-1)^n\ov{\sigma}(\bz+\bz')\quad\mod{\ov{\sigma}\bigl(F^{n-3}_{U,R}(X_{n-1})\xcdot E\bigr)}.
$$
Hence, by Remark~\ref{Prev a A7} and items~(4) and~(7) of Lemma~\ref{lemma A6},
$$
\psi(\bx)\equiv (-1)^{n+1}\sigma^0\xcirc\sigma^{-1}\xcirc \upsilon(\bz) + (-1)^n \sigma^0(\bz') \quad \mod{F^{n-3}_{U,R}(X_n)}.
$$
By Lemma~\ref{lema gamma(V)gamma(V)} and the definition of $\sigma^0$, we have $\sigma^0(\bz')\in U_{2,n-2}^1(R)$. Consequently, by the~defi\-nitions of $\upsilon$ and $\sigma^{-1}$, $\psi(\bx)\equiv \sigma^0(\bz'') \mod{F^{n-2}_{U,R}(X_n)}$. Since $1_E\in K\ot_k V$, the formula for $\psi(\bx)$ follows now immediately from equality~\eqref{primera cond}.

\smallskip

\noindent (5) We proceed by induction on $n$. The case $n=0$ is trivial. Suppose $n>0$ and the result is valid for $n-1$. Assume first that $j<n$. Let
\begin{align*}
&\byy \coloneqq  1_E\ot\ov{\gamma}(\bv_{1,i-1})\ot\ov{\jmath}_{\nu}(\ba_{i,j-1})\ot \ov{a\xcdot\gamma(v)} \ot\ov{\jmath}_{\nu}(\ba_{j+1,n-1}) \ot\jmath_{\nu}(a_n)
\shortintertext{and}
&\bz \coloneqq \sum 1_E\ot_{\hs A}\wt{\gamma}_{\hs A}(\bv_{1,i-1})\ot\ov{\ba}_{i,j-1}\ot\ov{a}\ot \ov{\ba}_{j+1,n-1}^{(l)} \ot\gamma(v^{(l)}\bigr)\jmath_{\nu}(a_n),
\end{align*}
where $\sum_l \ov{\ba}_{j+1,n-1}^{(l)}\ot_k v^{(l)}\coloneqq \ov{\chi}(v\ot_k \ov{a}_{j+1,n-1})$. Note that $\bz\in L_{n-i,i-1}\xcdot E$. By Proposition~\ref{formula para psi_n(y ot 1)} and the inductive hypothesis,
$$
\psi(\bx) = (-1)^n\ov{\sigma}\xcirc\psi(\byy)\equiv (-1)^n\ov{\sigma}(\bz)\quad \mod{\ov{\sigma}\bigl(F^{i-2}_{U,R}(X_{n-1})\xcdot \jmath_{\nu}(A)\bigr)}.
$$
Thus, by Remark~\ref{Prev a A7} and item~(4) of Lemma~\ref{lemma A6}, $\psi(\bx)\equiv (-1)^n\sigma^0(\bz) \mod{F^{i-2}_{U,R}(X_n)}$.
The expression for $\psi(\bx)$ follows now from the fact that $\gamma(V)\subseteq K\ot_k V$ and equalities~\eqref{eq12} and~\eqref{segunda cond}. We next consider the case $j=n$. Let
$$
\byy \coloneqq  1_E\ot\ov{\gamma}(\bv_{1,i-1})\ot\ov{\jmath}_{\nu}(\ba_{i,n-1})\ot a\xcdot \gamma(v)\quad\text{and}\quad \bz \coloneqq  1_E\ot_{\hs A}\wt{\gamma}_{\hs A}(\bv_{1,i-1})\ot\ov{\ba}_{i,n-1}\ot a\xcdot\gamma(v).
$$
Note that $\bz\in L_{n-i,i-1}\xcdot E$. By Proposition~\ref{formula para psi_n(y ot 1)} and item~(1), $\psi(\bx) = (-1)^n\ov{\sigma}\xcirc\psi(\byy) = (-1)^n\ov{\sigma}(\bz)$. Thus, by item~(4) of Lemma~\ref{lemma A6}, $\psi(\bx)\equiv (-1)^n\sigma^0(\bz)\quad\mod{F^{i-2}_{U,R}(X_n)}$. The expression for $\psi(\bx)$ follows now from the fact that $\gamma(V)\subseteq K\ot_k V$ and equality~\eqref{segunda cond}.

\smallskip

\noindent (6) By Proposition~\ref{formula para psi_n(y ot 1)},
$$
\psi(1_E\ot\ov{\bx}_{1n}\ot 1_E) = (-1)^n\ov{\sigma}\xcirc\psi(1_E\ot\ov{\bx}_{1,n-1}\ot x_n).
$$
Assume first that $x_n\notin \jmath_{\nu}(A)\cup\gamma(V)$. Then, by item~(2), $\psi(\bx) = (-1)^n\ov{\sigma}(0) = 0$. Assume now that $x_n\in\jmath_{\nu}(A)$. By inductive hypothesis, $\psi(\bx)\in\ov{\sigma}\bigl(F^{i-2}_{U,R}(X_{n-1})\xcdot \jmath_{\nu}(A)\bigr)$, and consequently, by Remark~\ref{Prev a A7}, $\psi(\bx)\in F^{i-2}_{U,R}(X_n)$. Assume finally that $x_n\hs\!\in\!\hs\gamma(\!V\!)$. If $j_2$ in the statement can be taken lesser than $n$, then, by in\-duc\-tive hypothesis, $\psi(\bx)\in\ov{\sigma}\bigl(F^{i-3}_{U,R}(X_{n-1})\xcdot\gamma(V)\bigr)$, and so, again by Remark~\ref{Prev a A7},
$$
\psi(\bx)\in F^{i-3}_{U,R}(X_n)\subseteq F^{i-2}_{U,R}(X_n).
$$
If $j_2$ is necessarily equals $n$, then, by items~(3), (4) and~(5),
$$
\psi(\bx)\in\ov{\sigma}\bigl(U_{n-i,i-1}+ L_{n-i+1,i-2}\xcdot\gamma(V)\gamma(V)\bigr)+\ov{\sigma}\bigl(F^{i-3}_{U,R}(X_{n-1})\xcdot \gamma(V)\bigr).
$$
So, by Remark~\ref{Prev a A7} and items~(4) and~(6) of Lemma~\ref{lemma A6},
$$
\psi(\bx)\in\sigma^0\bigl(L_{n-i+1,i-2} \xcdot\gamma(V)\gamma(V)\bigr)+F^{i-3}_{U,R}(X_n),
$$
and the result follows from Lemma~\ref{lema gamma(V)gamma(V)}, the fact that $\gamma(V)\subseteq K\ot_k V$ and equality~\eqref{segunda cond}.
\end{proof}

\begin{proposition}\label{proposition A8} Let $R$ be a $k$-subalgebra of $A$. Assume that $R$ is stable under $\chi$ and that $\mathcal{F}$ takes its values in $R\ot_k V$. Let $v,v_1,\dots,v_i\in V$ and $a,a_{i+1},\dots,a_n\in A$. The map $\phi_n \xcirc \psi_n$ has the following properties:

\begin{enumerate}

\item If $\bx \coloneqq  1_E\ot\ov{\gamma}(\bv_{1i})\ot\ov{\jmath}_{\nu}(\ba_{i+1,n})\ot 1_E$, then
$$
\qquad\phi\xcirc\psi(\bx) \equiv 1_E\ot\Sh(\bv_{1i}\ot_k \ov{\ba}_{i+1,n})\ot 1_E \quad \mod{\jmath_{\nu}(A) \bar{\ot} F_{\! R,1}^{i-1}(\ov{E}^{\ot n})\bar{\ot} \jmath_{\nu}(K).}
$$

\smallskip

\item Let $x_1,\dots,x_n\in\jmath_{\nu}(A)\cup\gamma(V)$. If there exist indices $j_1<j_2$ such that $x_{j_1}\in \jmath_{\nu}(A)$ and $x_{j_2}\in\gamma(V)$, then $\phi\xcirc \psi(1_E \ot\ov{\bx}_{1n}\ot 1_E) = 0$.

\smallskip

\item If $\bx= 1_E\ot\ov{\gamma}(\bv_{1,i-1})\ot\ov{a\xcdot\gamma(v)}\ot \ov{\jmath}_{\nu}(\ba_{i+1,n}) \ot 1_E$, then
\begin{align*}
\qquad \phi\xcirc\psi(\bx) &\equiv \sum \jmath_{\nu}(a^{(j)})\ot \Sh\bigl(\bv_{1,i-1}^{(j)}\ot_k v\ot_k \ov{\ba}_{i+1,n}\bigr)\ot 1_E\\
& +\sum 1_E\ot\Sh\bigl(\bv_{1,i-1}\ot_k\ov{a}\ot \ov{\ba}_{i+1,n}^{(l)}\bigr)\ot \gamma\bigl(v^{(l)}\bigr),
\end{align*}
modulo
$\jmath_{\nu}(A) \bar{\ot} F_{\! R,1}^{i-1}(\ov{E}^{\ot n})\bar{\ot} \jmath_{\nu}(K) + \jmath_{\nu}(A) \bar{\ot} F_{\! R,1}^{i-2}(\ov{E}^{\ot n})\bar{\ot} \jmath_{\nu}(K) \xcdot\gamma(V)$,
where
$$
\qquad \sum_j \ov{a}^{(j)}\ot_k \bv_{1,i-1}^{(j)}\coloneqq  \ov{\chi}(\bv_{1,i-1}\ot_k \ov{a}) \quad\text{and}\quad \sum_l \ov{\ba}_{i+1,n}^{(l)}\ot_k v^{(l)} \coloneqq  \ov{\chi}(v\ot_k \ov{\ba}_{i+1,n}).
$$

\smallskip

\item If $\bx = 1_E\ot\ov{\gamma}(\bv_{1,j-1})\ot \ov{a\xcdot\gamma(v)}\ot \ov{\gamma}(\bv_{j+1,i}) \ot \ov{\jmath}_{\nu}(\ba_{i+1,n})\ot 1_E$ with $j<i$, then
$$
\qquad \phi\xcirc \psi(\bx)\equiv \sum \jmath_{\nu}(a^{(l)})\ot \Sh\bigl(\bv_{1,i-1}^{(l)}\ot_{\hs k} v\ot_k \ov{\ba}_{i+1,n}\bigr)\ot 1_E,
$$
modulo
$$
\qquad \jmath_{\nu}(A) \bar{\ot} F_{\! R,1}^{i-1}(\ov{E}^{\ot n})\bar{\ot} \jmath_{\nu}(K) + \jmath_{\nu}(A) \bar{\ot} F_{\! R,1}^{i-2}(\ov{E}^{\ot n})\bar{\ot} \jmath_{\nu}(K) \xcdot\gamma(V),
$$
where $\sum_l \ov{a}^{(l)}\ot_k \bv_{1,i-1}^{(l)}\coloneqq  \ov{\chi}(\bv_{1,i-1}\ot_k \ov{a})$.

\smallskip

\item If $\bx = 1_E\ot\ov{\gamma}(\bv_{1,i-1})\ot\ov{\jmath}_{\nu}(\ba_{i,j-1})\ot \ov{a\xcdot \gamma(v)}\ot \ov{\jmath}_{\nu}(\ba_{j+1,n})\ot 1_E$ with $j>i$, then
$$
\qquad\phi\xcirc \psi(\bx)\equiv\sum_l 1_E\ot\Sh\bigl(\bv_{1,i-1}\ot_k \ov{\ba}_{i,j-1} \ot \ov{a} \ot \ov{\ba}_{j+1,n}^{(l)}\bigr)\ot\gamma \bigl(v^{(l)}\bigr),
$$
modulo $\jmath_{\nu}(A) \bar{\ot} F_{\! R,1}^{i-2}(\ov{E}^{\ot n})\bar{\ot} \jmath_{\nu}(K) \xcdot\gamma(V)$, where $\sum_l \ov{\ba}_{j+1,n}^{(l)}\ot_k v^{(l)} \coloneqq  \ov{\chi}(v\ot_k \ov{\ba}_{j+1,n})$.

\smallskip

\item Let $x_1,\dots,x_n\in E$ satisfying $\#\{l:x_l\notin \jmath_{\nu}(A)\cup \gamma(V)\} = 1$. If there exist $j_1<j_2$ such that $x_{j_1}\in \jmath_{\nu}(A)$ and $x_{j_2}\in \gamma(V)$, then
$$
\qquad \phi\xcirc \psi(1_E\ot\ov{\bx}_{1n}\ot 1_E)\in \jmath_{\nu}(A) \bar{\ot} F_{\! R,1}^{i-2}(\ov{E}^{\ot n})\bar{\ot} \jmath_{\nu}(K) \xcdot\gamma(V),
$$
where $i\coloneqq \#\{l:x_l\notin\jmath_{\nu}(A)\}$.

\end{enumerate}

\end{proposition}

\begin{proof} (1)\enspace This is a consequence of Proposition~\ref{propiedad de phi} and item~(1) of Proposition~\ref{propA.7}.

\smallskip

\noindent (2)\enspace This is a consequence of item~(2) of Proposition~\ref{propA.7}.

\smallskip

\noindent (3)\enspace By item~(3) of Proposition~\ref{propA.7},
$$
\begin{aligned}
\qquad \phi\xcirc \psi(\bx) &\equiv \phi \bigl(1_E\ot_{\hs A} \wt{\gamma}_{\hs A}(\bv_{1,i-1}) \ot_{\hs A} a\xcdot \wt{\gamma}(v)\ot \ov{\ba}_{i+1,n} \ot 1_E\bigr)\\
& + \sum \phi\bigl(1_E\ot_{\hs A} \wt{\gamma}_{\hs A}(\bv_{1,i-1})\ot \ov{a}\ot \ov{\ba}_{i+1,n}^{(l)}\ot \gamma(v^{(l)})\bigr)
\end{aligned}\quad\mod{\phi\bigl(F_{\!U,R}^{i-2}(X_n)\bigr)},
$$
Hence, by equality~\eqref{eq12} and Proposition~\ref{propiedad de phi},
$$
\phi\xcirc \psi(\bx) \equiv \sum \jmath_{\nu}(a^{(j)})\ot \Sh\bigl(\bv_{1,i-1}^{(j)}\ot_k v\ot_k \ov{\ba}_{i+1,n}\bigr)\ot 1_E+ \sum 1_E\ot \Sh\bigl(\bv_{1,i-1}\ot_k \ov{a}\ot \ov{\ba}_{i+1,n}^{(l)}\bigr) \ot \gamma\bigl(v^{(l)}\bigr),
$$
modulo $\jmath_{\nu}(A) \bar{\ot} F_{\! R,1}^{i-1}(\ov{E}^{\ot n})\bar{\ot}\jmath_{\nu}(K) + \jmath_{\nu}(A)\bar{\ot} F_{\! R,1}^{i-2}(\ov{E}^{\ot n}) \bar{\ot} \jmath_{\nu}(K) \xcdot\gamma(V) + \phi \bigl(F_{\!U,R}^{i-2}(X_n)\bigr)$. Since, by Remark~\ref{imagen de Sh}, Proposition~\ref{propiedad de phi}, and the definition of $F_{\!U,R}^{i-2}(X_n)$,
$$
\phi \bigl(F_{\!U,R}^{i-2}(X_n)\bigr)\subseteq \jmath_{\nu}(A) \bar{\ot} F_{\! R,1}^{i-2}(\ov{E}^{\ot n})\bar{\ot} \jmath_{\nu}(K) \xcdot\gamma(V),
$$
the statement is true.

\smallskip

\noindent (4)\enspace Mimic the proof of item~(3).


\smallskip

\noindent (5)\enspace Mimic the proof of item~(3).


\smallskip

\noindent (6)\enspace Mimic the proof of item~(3).
\end{proof}

\begin{proposition}\label{prop A.9} If $x_1,\dots,x_n\in E$ satisfy $\# \{j:x_j\notin \jmath_{\nu}(A)\cup\gamma(V)\}\le 1$, then
$$
\omega(\bx)\in \jmath_{\nu}(A) \bar{\ot} F_{\! A,0}^i(\ov{E}^{\ot {n+1}})\bar{\ot} \jmath_{\nu}(K),
$$
where $\bx\coloneqq 1_E\ot\ov{\bx}_{1n}\ot 1_E$ and $i\coloneqq \# \{j:x_j\notin \jmath_{\nu}(A)\}$.
\end{proposition}

\begin{proof} We first claim that if $\# \{j:x_j\notin \jmath_{\nu}(A)\cup\gamma(V)\} =0$, then $\omega(\bx)=0$. We proceed by induction on $n$. The case $n=1$ is trivial, since $\omega_1 = 0$ by definition. Assume that $n>1$ and the claim holds for $n-1$. Then,
$$
\omega(\bx) = \xi\bigl(\phi\xcirc \psi(\bx) - (-1)^n \omega(1_E\ot\ov{\bx}_{1,n-1}\ot x_n)\bigr) = \xi\xcirc \phi\xcirc \psi(\bx) = 0,
$$
where the first equality holds by Remark~\ref{simp formula w}; the second one, by the inductive hypothesis;~and the last one, by the facts that $\jmath_{\nu}(A) \bar{\ot} F_{\! A,0}^i(\ov{E}^{\ot n})\bar{\ot} \jmath_{\nu}(K)\! \subseteq\! \ker(\xi)$ and, by items~(1) and~(2) of Pro\-position~\ref{proposition A8},
$$
\phi\xcirc \psi(\bx)\in \jmath_{\nu}(A) \bar{\ot} F_{\! A,0}^i(\ov{E}^{\ot n})\bar{\ot} \jmath_{\nu}(K).
$$
We now assume that $\# \{j:x_j\notin \jmath_{\nu}(A)\cup\gamma(V)\} =1$ and we prove the proposition by induction on $n$. This is trivial for $n=1$ since $w_1 = 0$. Suppose that $n>1$ and the proposition is true for $n-1$. Since by Remark~\ref{simp formula w}
$$
\omega(\bx) = \xi\bigl(\phi\xcirc \psi(\bx) - (-1)^n \omega(1_E\ot \ov{\bx}_{1,n-1}\ot x_n)\bigr),
$$
and, by items~(3)--(6) of Proposition~\ref{proposition A8}, Remark~\ref{imagen de Sh}, Proposition~\ref{propiedad de phi} and the definition of $\xi$,
\begin{align*}
\xi\xcirc \phi\xcirc \psi(\bx) & \in \xi\Bigl(\jmath_{\nu}(A) \bar{\ot} F_{\! A,0}^i(\ov{E}^{\ot n})\bar{\ot} \jmath_{\nu}(K) + \jmath_{\nu}(A) \bar{\ot} F_{\! A,0}^{i-1}(\ov{E}^{\ot n})\bar{\ot} \jmath_{\nu}(K)\xcdot \gamma(V)\Bigr)\\
& \subseteq  \jmath_{\nu}(A) \bar{\ot} F_{\! A,0}^i(\ov{E}^{\ot {n+1}})\bar{\ot} \jmath_{\nu}(K),
\end{align*}
in order to finish the proof it suffices to check that
\begin{equation}\label{pepe2}
\xi\xcirc\omega(1_E\ot \ov{\bx}_{1,n-1}\ot x_n)\in \jmath_{\nu}(A) \bar{\ot} F_{\! A,0}^i(\ov{E}^{\ot {n+1}})\bar{\ot} \jmath_{\nu}(K).
\end{equation}
But, since by the inductive hypothesis and the claim,

\begin{itemize}

\smallskip

\item[-] if $x_n\in \jmath_{\nu}(A)$, then $\omega(1_E\ot \ov{\bx}_{1,n-1}\ot x_n)\in F_{\hs A,0}^i(E\ot \ov{E}^{\ot n}\ot E)\xcdot \jmath_{\nu}(A)$,

\smallskip

\item[-] if $x_n\in \gamma(V)$, then $\omega(1_E\ot \ov{\bx}_{1,n-1}\ot x_n)\in F_{\hs A,0}^{i-1}(E\ot \ov{E}^{\ot n}\ot E)\xcdot \gamma(V)$,

\smallskip

\item[-] if $x_n\notin \jmath_{\nu}(A)\cup \gamma(V)$, then $\omega(1_E\ot\ov{\bx}_{1,n-1}\ot x_n)=0$,

\smallskip

\end{itemize}
the condition~\eqref{pepe2} is satisfied.
\end{proof}

\printindex

\end{document}